\documentclass[11pt,leqno]{amsart}
\usepackage{amsmath,amsfonts,amsthm,amscd,amssymb,graphicx}
\usepackage{epstopdf}
\numberwithin{equation}{section}
\usepackage{fullpage}
\usepackage{color}

\newcommand{\Sp}{\operatorname{Sp}}

\renewcommand\a{\alpha}
\renewcommand\b{\beta}

\def\l{\lambda}
\def\eps{\varepsilon }

\renewcommand\a{\alpha}

\renewcommand\b{\beta}

\newcommand\R{\mathbb R}
\newcommand\C{\mathbb C}

\def\eps{\varepsilon}

\def\l{\lambda}

\newcommand\br{\begin{remark}}
\newcommand\er{\end{remark}}
\newcommand\brs{\begin{remarks}}
\newcommand\ers{\end{remarks}}
\newcommand\bp{\begin{pmatrix}}
\newcommand\ep{\end{pmatrix}}
\newcommand{\be}{\begin{equation}}
\newcommand{\ee}{\end{equation}}
\newcommand\ba{\begin{equation}\begin{aligned}}
\newcommand\ea{\end{aligned}\end{equation}}
\newcommand\ds{\displaystyle}

\newcommand{\bap}{\begin{app}}
\newcommand{\eap}{\end{app}}
\newcommand{\begs}{\begin{exams}}
\newcommand{\eegs}{\end{exams}}
\newcommand{\beg}{\begin{example}}
\newcommand{\eeg}{\end{exaplem}}
\newcommand{\bpr}{\begin{proposition}}
\newcommand{\epr}{\end{proposition}}
\newcommand{\bt}{\begin{theorem}}
\newcommand{\et}{\end{theorem}}
\newcommand{\bc}{\begin{corollary}}
\newcommand{\ec}{\end{corollary}}
\newcommand{\bl}{\begin{lemma}}
\newcommand{\el}{\end{lemma}}
\newcommand{\bd}{\begin{definition}}
\newcommand{\ed}{\end{definition}}

\newcommand{\B }{\mathcal{B}}

\newcommand{\RR}{{\mathbb R}}

\newcommand{\ZZ}{{\mathbb Z}}

\newtheorem{theorem}{Theorem}[section]
\newtheorem{proposition}[theorem]{Proposition}
\newtheorem{corollary}[theorem]{Corollary}
\newtheorem{lemma}[theorem]{Lemma}

\theoremstyle{remark}
\newtheorem{remark}[theorem]{Remark}
\newtheorem{remarks}[theorem]{Remarks}
\theoremstyle{definition}
\newtheorem{definition}[theorem]{Definition}

\newtheorem{example}[theorem]{Example}

\newcommand{\lb}{\label}

\newcommand{\ran}{\text{\rm{ran}}}

\newcommand{\beq}{\begin{equation}}
\newcommand{\eeq}{\end{equation}}

\title{
Diffusive stability 
of spatially periodic solutions of the Brusselator model
}

\author{Alim Sukhtayev}
\address{Indiana University, Bloomington}
\email{alimsukh@iu.edu}
\author{Kevin Zumbrun}
\address{Indiana University, Bloomington, IN 47405}
\email{kzumbrun@indiana.edu} 
\thanks{Research of K.Z. was partially supported
under NSF grants no. DMS-0300487 and DMS-0801745.}
\author{Soyeun Jung}
\address{Kongju National University, Korea}
\email{soyjung@kongju.ac.kr}
\thanks{Research of S.J. was supported by the National Research Foundation of Korea (NRF) grant funded by the
Korea government (MSIP) (No. 2016009978).}
\author{Raghavendra Venkatraman}
\address{Indiana University, Bloomington, IN 47405}
\email{rvenkatr@umail.iu.edu}
\thanks{Research of R.V. was partially supported under
	NSF grants no. DMS-1101290 and DMS-1362879.  }
\begin{document}

\begin{abstract}
Applying the Lyapunov-Schmidt reduction approach introduced by Mielke and Schneider in their analysis of the
fourth-order scalar Swift-Hohenberg equation, we carry out a rigorous small-amplitude stability analysis of Turing
patterns for the canonical second-order system of reaction diffusion equations given by the Brusselator model.
Our results confirm that stability is accurately predicted in the small-amplitude limit by the 
formal Ginzburg Landau amplitude equations, 
rigorously validating the standard weakly unstable approximation and Eckhaus criterion.
\end{abstract}

\date{\today}
\maketitle

\tableofcontents

\section{Introduction}

The topic of pattern formation has been the object of considerable attention since the 
fundamental
observation of Turing \cite{T,C} that reaction diffusion systems 
modeling biological/chemical processes
can spontaneously develop patterns through destabilization of the homogeneous state.  
A parallel impetus has come from the study through bifurcation theory
of hydrodynamic pattern formation phenomena such as Taylor-Couette flow and Rayleigh-B\`enard convection
\cite{KS,NPL,M3}.

Going beyond the question of existence, an equally 
fundamental
topic is stability, or ``selection,'' of
periodic patterns, and linear and nonlinear behavior under perturbation \cite{E,NW,M1,M2,M3,S1,S2,DSSS,SSSU,JZ,JNRZ1,JNRZ2}.
Here, two particular landmarks 
are
the formal ``weakly unstable,'' or small-amplitude, theory of Eckhaus \cite{E} 
deriving the Ginzburg Landau equation as a canonical model for behavior near the threshold of instability in
a variety of processes, and the rigorous linear and nonlinear verification of this theory in \cite{M1,M2,S1}
for the Swift-Hohenberg equation, a canonical model for hydrodynamic pattern formation.

The first-mentioned analysis is completely general, and the second in principle equally so.
Indeed, the passage from spectral to nonlinear stability has by now been established for small- and large-amplitude
patterns alike \cite{S1,S2,JZ,JNRZ1,JNRZ2,SSSU}, with in addition considerable information on modulational behavior.
However, up to now the rigorous characterization of spectral stability has been carried out in all details only for
the particular case of the (scalar) Swift-Hohenberg equation \cite{M1,M2,S1}
\be\label{SH}
\partial_tu= -(1+\partial_x^2)^2 u + \varepsilon^2 u- u^3, \quad u\in \R^1,  
\ee
where $\eps \in \R^1$ is a bifurcation parameter.

The purpose of the present paper is to carry out the program of \cite{M1,M2,S1} also for a system of reaction diffusion equations
\be\label{rd}
u_t = D\partial_x^2 u + f(u, \mu),
\quad u\in \R^n, \, D\in \R^{n\times n}, \, \mu \in \R^1,
\ee
in the case $n=2$ originally considered by Turing, {\it rigorously characterizing spectral stability
in the small-amplitude, or weakly-unstable, limit}.

Specifically, we consider the Brusselator model \cite{PL} 
\ba \label{bruss1}
\partial_t u_1 & = D_1 \partial_x^2u_1 +a -(\b+1)u_1+u_1^2u_2 \\
\partial_t u_2 & =D_2 \partial_x^2 u_2+ \b u_1-u_1^2u_2
\ea
a canonical model for pattern formation in autocatalytic chemical reaction, 
with equilibrium states $u \equiv (a,\beta/a)$.
Here, $u_j\in \R^1$ represent species concentrations,
$D_j\in \R^1$ species diffusion constants and $a$ and $\beta$ ambient concentrations
of precursor species.
As is standard, we consider $a, D_j$ as model parameters and $\beta>0$ a bifurcation parameter,
for concreteness fixing the ``typical'' values $a=2$, $D_1=4$, $D_2=16$ throughout.
The analysis readily generalizes to general values of $a$, $D_j$.

For this model, there is a Turing instability of the equilibrium state at $\beta=4$, with linear oscillating
modes $c e^{\pm ik_0 x}r$, $k_0=1/2$, $r\in \mathbb R^2$.
Thus, setting $\beta=4 + \eps^2$ following standard convention, 
we expect, similarly as for \eqref{SH}, a smooth branch of solutions 
\be\label{branch}
u=(2,2)^\top+\{ \eps e^{i(k_0+ \eps \omega ) x}r + O(\eps^2)\} + c.c.,
\ee
bifurcating from $\eps=0$, where $c.c.$ denotes complex conjugate, and $\omega$ lies in an appropriate
range consisting of an $\eps$-order perturbation of a fixed open interval $I_E$ determined by
an associated formal amplitude equation given in this case by the real Ginzburg Landau equation \cite{E,M1,M2,M3,S1};
moreover, stability and behavior under perturbation of these solutions
are expected to be governed to lowest order by this same Ginzburg Landau equation, with stability determined by
the simple {\it Eckhaus criterion} that $\omega$ lie in an $\eps$-perturbation of a fixed open subinterval
$I_S$ of the interval of existence $I_E$.

Following the Lyapunov-Schmidt reduction program laid out in \cite{M1,M2,S1}, we rigorously validate
\eqref{branch} as describing the unique branch of solutions bifurcating from equilibrium in a neighborhood of the
Turing instability, and give a detailed description of the spectra of the linearized operator about the bifurcating solution,
showing that it agrees to lowest order with that of the linearization of the Ginzburg Landau equation about
\eqref{branch}.
This verifies in particular that stability is indeed predicted by the simple Eckhaus criterion of the formal theory.

The analysis, and computations, turn out to be surprisingly more complicated than in the Swift-Hohenberg case.
In particular, it is here necessary to compute the $\eps^2$-order and $\eps^3$-order correctors
 in \eqref{branch}, whereas in the Swift-Hohenberg case, due to the twin properties that it is scalar with only third-order 
nonlinearities, the $\eps^2$-order term can be seen to identically vanish and the $\eps^3$-order term need not be computed for
the analysis of the reduced equation.
This amounts to computing third-order instead of first-order Taylor expansions, which, in the vectorial case, grow exponentially
in computation effort with degree.
Moreover, it is not a priori clear that the associated new remainder terms in the ultimately resulting $2\times 2$ reduced equations
will be sufficiently small to yield the desired spectral description.
To carry out the details of the program of \cite{M1,M2,S1} in this more generic case, and to verify that the 
argument indeed closes, 
is one of the main contributions of this paper.

A second contribution is to reframe the stability analysis of the Ginzburg Landau equation in a way
illuminating the connection with Lyapunov--Schmidt reduction, for which, under appropriate interpretation/scaling,
the two processes can be seen not only to generate the same final results but to match operation-by-operation.
In the analyses of the Swift--Hohenberg equation in \cite{M1,M2,S1}, 
these final results were instead obtained by apparently quite different computations, then seen
by direct comparison to correspond, leaving unclear the mechanism by which this correspondence 
should extend to more general models.
An interesting further detail arising in the present case that was not present in the Swift--Hohenberg case is that
the linearized equations are not self-adjoint, so that there is a transition between the large-scale spectrum of
the linearized operator, which is in general complex, and the small-scale spectrum, expected by analogy to the approximating
self-adjoint linearized Ginzburg Landau operator to be real.
By a higher order Lyapunov--Schmidt reduction computation,
we are able to pinpoint the location and nature of this transition rather precisely, showing that
within an order $\eps$ ``Ginzburg Landau'' regime, spectra of the full equations 
are real, while in an order $\eps$ transition regime there exist spectra that are complex.
We hope, finally, that it may be useful simply to gather here in one place the elements of the weakly unstable/small-amplitude
theory in the concrete but general context of reaction diffusion systems.

\subsection{Turing instability and the Ginzburg Landau approximation}\label{s:GLapprox}
We first recall the general Ginzburg Landau approximation, following \cite{M3}.
Consider a general reaction diffusion system \eqref{rd} with $D>0$ diagonal and constant, assuming without loss of generality
$f(0,\mu)\equiv 0$, so that $u\equiv 0$ is an equilibrium solution for all $\mu$.
Following Turing, suppose moreover that matrix $M_\mu:=(\partial f/\partial u)(0,\mu)$ is stable (has eigenvalues of strictly negative real part).
(Here and elsewhere, $\Sp(N)$ denotes spectrum of an operator or matrix $N$).
Then, the dispersion relation $\lambda \in \Sp (-k^2 D + M_\mu)$ determined by the Fourier symbol of the linearized operator about $u\equiv 0$,
where $k\in \R$ denotes Fourier frequency, is evidently stable ($\Re \lambda<0$) for $|k|$ sufficiently small or large; thus, instabilities, should
they occur, must occur for finite wave numbers, bounded away from $0$ and $\pm \infty$.
This cannot happen for $n=1$, as $D$ and $M_\mu$ then commute.
However, it can occur for any $n\geq 2$, with appropriate choices of parameters \cite{C}.
Of particular interest is the transition at value $\mu_0$ from stability to instability of the constant solution $u\equiv 0$, at which 
one or more eigenvalues of the symbol $(-k^2 D + M_\mu)$ pass through the imaginary axis for $k=\pm k_0\neq 0$.
In the case $n=2$ considered by Turing, this crossing necessarily occurs at $\lambda=0$ and involves a simple root \cite{C}.

More generally, for an $n$-dimensional system \eqref{rd}, we denote as a {\it Turing instability}
a system \eqref{rd} and values $\mu_0$, $k_0$ for which $\Re \Sp (-Dk^2+M_\mu)<0$ for all $k\in \R$ for $\mu<\mu_0$ but not for $\mu\geq \mu_0$,
with $\Re \Sp(-k^2 D + M_{\mu_0})<0$ except for simple eigenvalues $\lambda=0$ at $k=\pm k_0$, whose real parts
are nondegenerate maxima with respect to $k$ and grow at nonvanishing rate with respect to $\mu$.
By matrix perturbation theory \cite{K}, these eigenvalues may in the vicinity of $(\mu_0, k_0)$ be extended along with their associated
eigenvectors $r$ as smooth functions
\be\label{disp}
\lambda=\lambda_0(\mu,k), \quad r=r_0(\mu,k).
\ee 
By reflection symmetry of the linearized equations with respect to $x$ (indeed, there holds $O(2)$ symmetry 
given by translation and reflection invariance in $x$), eigenfunctions for a given $\lambda$
value appear in pairs $e^{\pm ikx}r$, whence we may deduce, noting that also by real-valuedness of the operator, $\bar \lambda$, $e^{-ik x}\bar r$
must be another eigenvalue, eigenfunction pair, that simplicity of $\lambda_0(\mu, k)$ implies $(\lambda_0,r_0)(\mu, k)$ real (see, e.g., \cite{PYZ,GS} and references therein).

Thus, we have
\be\label{derivfacts}
(\partial \lambda_0/\partial k)(\mu_0,k_0)=0, \quad
(\partial^2 \lambda_0/\partial k^2)(\mu_0,k_0)= real < 0;
\qquad 
(\partial \lambda_0/\partial \mu)(\mu_0,k_0)=real> 0.
\ee

Motivated by the parabolic behavior \eqref{derivfacts},
introduce the diffusive scaling 
\be\label{dscale}
\mu=\mu_0 +\eps^2, \quad  X=\eps x,\quad T=\eps^2 t,
\ee
and make the multi-scale ansatz
\be\label{multiscale}
u(x,t)\approx  \{ \eps A(\eps x, \eps^2 t)e^{ik_0 x}r + \eps^2 v_2(\eps x, \eps^2 t; x) + \dots \} + c.c.,
\quad A\in \C,
\ee
where $c.c.$ denotes complex conjugate and $v_j$ are to be determined.
Denote $L= D\partial_x^2 + M_{\mu_0}$ and
$f(u, \mu_0)- M_{\mu_0} u =N_2(u,u) + N_3(u,u,u) + O(|u|^4)$, where $N_j$ are symmetric multilinear forms corresponding to mixed directional
derivatives, so that 
\be\label{eq}
D\partial_x^2 +f(u,\mu)= L + \eps^2 \hat M + N_2(u,u) + N_3(u,u,u) + O(\eps^4+|u|^4),
\ee
where $\hat M:=(\partial M_{\mu} /\partial \mu)|{\mu=\mu_0}$.

Then, substituting in \eqref{rd} and matching powers of $\eps$, we obtain, at order $\eps^1$, the equation
$Le^{ik_0x}r_0=0$, as follows automatically by definition of $k_0$, $r_0$.
At order $\eps^2$, we obtain
\be\label{2nd}
\{Lv_2 + 2ik_0 D \partial_X A e^{ik_0 x} r_0 + \frac12 N_2(r,r)( e^{2ik_0 x} + 1)
\} + c.c.=0
\ee
where $L$ is applied in the second (fast) variable only.
By the assumptions on $\Sp(L)$, this is soluble for $v_2$ precisely if the eigenprojection $P$ of $(-k_0^2D+ M_{\mu_0})$ onto its kernel
annihilates the term $ ik_0 D \partial_X A e^{ik_0 x} r_0 $, or 
\be\label{Dannil}
P  D r_0=0 .
\ee
By standard spectral perturbation theory \cite{K}, this is equivalent to
$(\partial \lambda_0/\partial k)(\mu_0,k_0)=0$, as
holds by assumption \eqref{derivfacts}(i). 
At order $\eps^2$, the corresponding solvability condition gives, finally, the {\it real Ginzburg Landau equation}
\be\label{GL}
\partial_T A= d\partial_X^2 A + e A - f|A|^2 A,
\ee
$A\in \C$, for appropriate $d,e,f\in \R$.  Here, the linear coefficients may be computed simply through
$$
e= (\partial \lambda_0/\partial \mu)|_{\mu_0, k_0}, \qquad d=- (1/2)(\partial^2 \lambda_0/\partial k^2)|_{\mu_0,k_0},
$$
while the nonlinear coefficient $f$ is given by a more complicated formula involving $N_2$, $N_3$, $r_0$, $L$;
see Section \ref{s:shortcut} for details.

Equation \eqref{GL} may then be solved explicitly for solutions $A=c(\omega) e^{i\omega X}= e^{i\omega \eps x}$, yielding the aforementioned
prediction \eqref{branch} regarding existence, for 
\be\label{exist1}
\omega \in I_E:= [-\sqrt{e/d},+\sqrt{e/d}].
\ee
Likewise, the linearized stability problem may be solved explicitly \cite{T,TB,M1} 
(see Section \ref{s:GLcompare})
to yield the {\it Eckhaus stability criterion} 
\be\label{eck}
\omega \in I_S:= (-\sqrt{e/3d},+\sqrt{e/3d} ) \subset I_E.
\ee
This recovers the formal theory of Eckhaus \cite{E} as applied to reaction diffusion systems.

\br[Shortcut computations]\label{s:shortcut}
As noted by Mielke \cite[Section 2]{M3}, 
once one knows existence of a valid Ginzburg Landau expansion, one may compute coefficients efficiently by various shortcuts.
For example, consider the dispersion relation $\lambda(\mu, k)$ around the given base (constant) state.
Introducing the Ginzburg-Landau scalings
$\mu= \eps^2$, $k=k_0 + \eps \omega$, expand $\lambda(\eps^2, k_0+ \eps \omega)$ in powers of $\eps$,
we obtain (via the Chain rule):
\be\label{ccexp}
\lambda( \eps^2, k_0+ \eps \omega) =
(\partial \lambda/\partial \mu)|_{0, k_0} \eps^2 + (1/2)(\partial^2 \lambda/\partial_k^2)|_{0,k_0}\eps^2 \omega^2
+ O(\eps^3).
\ee
	Then, the coefficients of \eqref{ccexp} agree with the linear parts of the Ginzburg-Landau equation \eqref{GL2},
	with 
	$e= (\partial \lambda/\partial \mu)|_{0, k_0}$ and 
	$d=- (1/2)(\partial^2 \lambda/\partial_k^2)|_{0,k_0}$.
	This may be proved alternatively by:
	(i) direct comparison, computing via implicit differentiation from the characteristic polynomial, or
	(ii) observation that the Ginzburg-Landau expansion procedure, omitting nonlinear terms, is exactly the spectral expansion procedure
	for determining the Taylor expansion of $\lambda(\eps^2, k_0+ \eps \omega^2)$.
As noted in \cite{M3}, the constant-coefficient dispersion relation is typically computed in the course of locating
Turing instability in the first place, so 
often
already available.  In the scalar case, or for $2\times 2$ systems,
it is also considerably easier to compute the partial derivatives of $\lambda(\cdot, \cdot)$ than to carry out the complete
Ginzburg Landau expansion.
Likewise \cite{M3}, it is not necessary to include slow time-dependence in the derivation of nonlinear coefficients, thus eliminating a number of terms.
\er

\subsection{The Brusselator model}\label{s:bruss}
The Brusselator model \eqref{bruss1} corresponds to reactions
\be\label{reactions}
A \rightarrow X,
\quad
2X + Y \rightarrow 3X\quad
B + X \rightarrow Y + D\quad
X \rightarrow E,
\ee
with $u_1=\{X\}$, $u_2=\{Y\}$, $a=\{A\}$, $\beta=\{B\}$,
in the situation that precursor species A and B are present in inexhaustible, essentially fixed, concentrations,
yielding rate equations for product species X and Y (setting rate constants $=1$) of
\ba\label{brate}
{d \over dt}\left\{ X \right\} &= \left\{A \right\} 
+ \left\{ X \right\}^2 \left\{Y \right\}  - \left\{B \right\} \left\{X \right\} - \left\{X \right\} ,  \\ 
{d \over dt}\left\{ Y \right\} &=   \left\{B \right\} \left\{X \right\} - \left\{ X \right\}^2 \left\{Y \right\}    ,
\ea
where $\{ \cdot \}$ denotes concentration, with fixed point
$\left\{ X \right\} =   \left\{ A \right\}   $, $\left\{ Y \right\} =   \left\{B \right\} /\left\{
 A \right\}$.

The fixed point is stable for $\left\{B \right\} < 1+\left\{A \right\}^2$, 
at which point there is a Hopf bifurcation to chemical oscillation, or ``clock reactions.''
For us, $\{A\}=2$, $\{B\}= 4$, so that we are indeed in the stable regime envisioned by Turing.
As discussed in \cite{C}, Turing instability can occur for $2\times 2$ systems only for ratios of
diffusion constants $D_1/D_2$ rather far from $1$, hence our choice of $D_j$: (from \cite{C})
``Since the diffusion coefficients
of most small ions in water have the same value of about $10^{-9} m^2/sec$, some ingenuity is required to
create a Turing instability. Experimentalists found (by accident!) that one way to achieve a large disparity
in diffusion coefficients was to introduce a third molecule (such as starch...)
that was fixed to an immobile matrix in the solution...''

The real Ginzburg Landau equation corresponding to the Brusselator model with our choice of parameters 
may be computed (see Section \ref{s:GLcompare}) to be
\be\label{Brusselator_GL}
\partial_t A= (32/3)\partial_x^2 A + (2/3) A - |A|^2 A;
\ee
in the notation of \eqref{GL}--\eqref{eck}, $d=32/3$, $e=2/3$, $f=1$.
Hence, for the Brusselator model that we study here, the stability and existence intervals given in \eqref{eck} are
\be\label{beck}
I_S= \Big[-\frac{1}{4\sqrt{3}},\frac{1}{4\sqrt{3}}\Big] \subset I_E=\Big[-\frac{1}{4},\frac{1}{4}\Big].
\ee

\subsection{Diffusive stability condition}\label{s:diff}
We next briefly recall the {\it diffusive stability condition} of Schneider \cite{S1,S2}.
Linearizing \eqref{rd} about a periodic solution $u(x,t)\equiv \bar u(x)$, we obtain 
the linearized equations
\be\label{lin}
u_t=Lu:= D\partial_x^2 u + df(\bar u, \mu) u.
\ee
Differentiating with respect to $x$ the profile ODE $\partial_x^2 D \bar u+ f(\bar u)=0$, we
find that $\partial_x \bar u$ is a bounded solution of the eigenvalue ODE $(L-\lambda)w=0$ for
$\lambda=0$, whence, by general Bloch expansion/Floquet theory \cite{S2,M2,M3}, there exists a
continuous curve $\lambda_*(\sigma)$ of spectra of $L$, defined for $\sigma \in \R$ sufficiently
small. The diffusive stability condition is that this branch be isolated, in the sense that all other
spectra have strictly negative real part, and multiplicity one in the sense of Bloch expansion;
we describe this last more carefully in Section \ref{s:stability}.
The second condition (see \cite{S2}) is that
\be\label{scond}
\Re \lambda_*(\sigma)\leq -\eta |\sigma|^2, \quad \eta>0.
\ee

\subsection{Main results}\label{s:mainresults} 
We are now in position to state our main results.
Let $H^s_{per}([0,2\pi], \RR^2)$ denote the space of $H^s$ functions that are periodic on the interval $[0,2\pi]$.
Making the coordinate shift $u \to u- (2,\beta/2)$ and the shift in $\beta$: $\beta \to b+4$, we may rewrite 
\eqref{bruss1} in the general form \eqref{rd}, with all equilibrium states centered at $u\equiv u_*=0$ and Turing instability occurring at $b=0$.
Introducing the wave number $k$ and making the independent coordinate change $x\to kx$, we may further normalize the set of periodic solutions
with wave number $k$ to periodic solutions on the fixed interval $[0,2\pi]$ of
\be \label{B model1}
0=N(b, k, \tilde u):=k^2D\partial_\xi^2 \tilde u + f(\tilde u),
\qquad N(b,k,0)\equiv 0.
\ee

Set now $b:=\eps^2$, $k=\frac{2\eps\omega+\sqrt{4\eps^2\omega^2+1}}{2}$.
Our first result rigorously characterizes Turing bifurcation of periodic solutions of \eqref{bruss1} from 
equilibrium states $u \equiv (2,\beta/2)$. 

\begin{theorem}[Existence]\label{existence}
	There is an $\eps_0$ such that for all $\eps\in[0,\eps_0)$ and all 
		$\omega\in I_E= [-\frac{1}{4},\frac{1}{4}]$ there is a unique small solution $\tilde u_{\eps,\omega}\in H^2_{per}([0,2\pi], \RR^2)$ of \eqref{B model1} which is even in $w$, positive at $w=0$ and has the expansion formula:
	\be
	\begin{split}
	\tilde u_{\eps,\omega}&=\sqrt{\frac{2}{3}(1-16\omega^2)}\cos \xi \bp 2 \\ -1 \ep\eps+\big(-\frac{4}{3}\sqrt{\frac{2}{3}(1-16\omega^2)}\cos \xi \bp 1 \\ -2 \ep\\
	&-\frac{5}{9}\omega\sqrt{\frac{2}{3}(1-16\omega^2)}\cos \xi \bp 2 \\ -1 \ep\big)\eps^2+\mathcal{O}(\eps^3).
	\end{split}
	\ee
	Note that when $\omega=\pm\frac{1}{4}$, $\tilde u_{\eps,\omega}\equiv0$ reduces to the equilibrium (zero) solution.
\end{theorem}

\begin{proof}
Given in Section \ref{s:existence}.
\end{proof}

Our second result rigorously characterizes diffusive stability/instability of bifurcating solutions.

\begin{theorem}[Stability]\label{stabthm}
	Let $u_{\eps,\omega}$ be the solution from Theorem \ref{existence}.  Then there exist
	 $\tilde\eps_0\in(0,\eps_0]$, where $\eps_0$ is taken from Theorem \ref{existence}, $\sigma_0>0$ and $\delta>0$ such that for all $\eps\in[0,\tilde\eps_0)$, all $\sigma\in[0,\sigma_0)$  and all $\omega\in[-\frac{1}{4},\frac{1}{4}]$ the spectrum of $B(\eps,\omega,\sigma)$ has the decomposition:
	\be
	\begin{split}
		\Sp(B(\eps,\omega,\sigma))=S\cup\{\lambda_1,\lambda_2\}.
	\end{split}
	\ee
	where $\Re\lambda<-\delta$ for $\lambda\in S$ and $|\lambda_j|<<1$.
	 Moreover, for each fixed 
	 $\omega\in I_S= (-\frac{1}{4\sqrt{3}}, \frac{1}{4\sqrt{3}})$ there exists $\hat\eps_0\in(0,\tilde\eps_0)$ such that for all $\eps\in[0,\hat\eps_0)$, all $\sigma\in[0,\sigma_0)$
	 \be\label{lamexp}
	\begin{split}
		&	\Re\lambda_{1}\leq c(\eps,\omega)+\tilde c(\eps,\omega)\sigma-\tilde{\tilde c}_1(\eps,\omega)\sigma^2+\mathcal O(\sigma^3),\\
		&	\Re\lambda_{2}\leq -\tilde{\tilde c}_2(\eps,\omega)\sigma^2+\mathcal O(\sigma^3),
	\end{split}
	\ee
	for $c(\eps)<0<\tilde{\tilde c}_j(\eps,\omega)$, giving \emph{diffusive stability},
	while if 
	$\omega \in I_E\setminus \overline{I_S}=[-\frac{1}{4}, -\frac{1}{4\sqrt{3}})\cup
	(\frac{1}{4\sqrt{3}}, \frac{1}{4}]$, then
	\be\label{unst}
\max_\sigma\{\Re\lambda_1,\Re\lambda_2\}>0,
\ee
giving \emph{diffusive instability}.

\end{theorem}

\begin{proof}
Given in Section \ref{s:stability}.
\end{proof}

Theorems \ref{existence} and \ref{stabthm} together rigorously validate the predictions of the Ginzburg Landau approximation
regarding existence and stability of small bifurcating solutions; cf. \eqref{beck}.
Our third main result addresses reality of the critical modes $\lambda_j(\cdot)$, identifying a transition from
a Ginzburg Landau zone $\sigma \sim \eps$ in which $\lambda_j$ are real to a transition zone
$\sigma\sim\sqrt{\eps}$ for which $\lambda_j$ can be complex.

\begin{theorem}[Reality]\lb{reality} Let $\lambda_1$ and $\lambda_2$ be as in Theorem \ref{existence}. Then there exist $\tilde{\tilde \eps}_0\in(0,\tilde \eps_0)$ and ${\tilde \sigma}_0\in(0, \sigma_0)$ such that for all $\eps\in[0,\tilde{\tilde \eps}_0)$, all $\sigma\in[0,\tilde\sigma_0)$ and all $\omega\in[-\frac{1}{4},0]$, $\lambda_1(\eps,\omega,\sigma)$ and $\lambda_2(\eps,\omega,\sigma)$ are real. Also, for the ansatz $\sigma=\sqrt{6\omega\eps}$ there exist $\tilde{\tilde \eps}_0\in(0,\tilde \eps_0)$ and $\omega_0\in(0,\frac{1}{4})$ such that for all $\eps\in[0,\tilde{\tilde \eps}_0)$ and all $\omega\in(\omega_0,\frac{1}{4}]$, $\lambda_1(\omega,\eps)$ and $\lambda_2(\omega,\eps)$ are complex (not real).
 	
\end{theorem}	
\begin{proof}
	Given in Section \ref{s:stability}.
\end{proof}

Our fourth main result states that, within the Ginzburg Landau regime $\lambda\sim \eps^2$, $\sigma \sim \eps$,
the Ginzburg Landau approximation not only well-predicts stability/instability, but to lowest order also
the linearized dispersion relations for the two smallest eigenmodes.

\bt\label{dispthm}
Setting $\sigma =:2\eps \hat \sigma$, $\lambda_j=:\eps^2 \hat \lambda_j$ in accordance with the Ginzburg Landau 
scaling \eqref{dscale}, $\lambda_j$ as in \eqref{lamexp}, we obtain expansions 
\be\label{exactdisp}
	 		 		 	\begin{split}
	&\hat\lambda_{1}(\hat\sigma)=-\frac{4}{3}(1-16\omega^2)
			-\frac{32(1+16\omega^2)}{3(1-16\omega^2)}\hat\sigma^2+\mathcal O(|\hat\sigma|^3 + \eps(1+\hat \sigma)),\\
	 		 		 	&\hat\lambda_{2}(
	 		 		\hat\sigma)=-\frac{32(1-48\omega^2)}{3(1-16\omega^2)}\hat\sigma^2
		+\mathcal O(|\hat\sigma^3| +\eps \hat \sigma^2).
	 		 		 		\end{split}
	 		 		 			\ee
for $|\hat\sigma|<<1$,
agreeing to lowest order with the corresponding expansions for the associated Ginzburg Landau 
approximation (cf. \eqref{gldisp}).
Moreover, for $|\hat \sigma |\leq C$ and $\omega \in I_E^{int}$,
$\lambda_j$ are real for $\eps<<1$.
\et

\begin{proof}
Given in Section \ref{s:GLcompare}.
\end{proof}

\subsection{Discussion and open problems}\label{s:discussion}
The dispersion relation \eqref{exactdisp} agrees to lowest order with 
\be\label{formal}
\hat \lambda_j(\hat \sigma)\in \Sp
\bp -\frac{32}{3}\hat \sigma^2-\frac{4}{3}(1-16\omega^2) & \frac{64}{3}i\omega\hat\sigma\\ -\frac{64}{3}i\omega\hat \sigma& -\frac{32}{3}\hat \sigma^2 \ep,
\ee
a self-adjoint matrix eigenvalue problem coming from the spectral stability analysis of the approximating Ginzburg Landau
equation linearized about the periodic solution corresponding to wave number $\omega$, for which $\hat \lambda_j$ are
evidently real; see \eqref{gldisp}.
(As described in Section \ref{s:ccred}, the spectral stability problem for the 
Ginzburg Landau equation is reducible to a constant-coefficient analysis; see, for example, \cite{TB}.)

In the Swift--Hohenberg case \cite{M1,M2,S1}, the linearized operator for the full system is self-adjoint, and
so it is known a priori that {\it all} exact eigenvalues are real-valued as well.
In the present (general) situation, this property is replaced by reflection-symmetry of the linearized system, 
which persists even when self-adjointness is lost.
By the same argument applied to the constant-coefficient problem below \eqref{disp}, plus simplicity of eigenvalues
$\lambda_j(\sigma)$ at $\sigma=0$, it follows
that the two small eigenvalues $\lambda_j(\sigma)$ remain real so long as they remain
distinct.
However, the rest may be real or complex, depending on the particular system, as may be readily seen even for the constant-coefficient case described in \eqref{disp}.
Moreover, as $\lambda_j(0)$ differ only to order $\eps^2$, we cannot conclude by this argument reality for $\sigma >> \eps^2$,
in particular not for the range $|\sigma|\leq \sigma_0$ of our main analysis.

This has implications for the stability argument, as we can therefore not use the approach of \cite{M1,M2,S1}
on the intermediate regime $\eps^2<< |\sigma| \leq \sigma_0$ of studying the sign of the better-behaved product
$\lambda_1\lambda_2$ rather than the real parts of individual eigenvalues $\lambda_j$.
We substitute for this a different argument subdividing into cases (i) $|\sigma|<< \eps^2$, 
(ii) $\eps^2/C\leq |\sigma|\leq C\eps^2$, and (iii) $|\sigma|>>\eps^2$, treating (i) and (iii) 
by $|\hat \sigma|\to 0$ and $|\hat \sigma |\to \infty$ asymptotics
and (ii) by continuity from the Ginzburg-Landau approximation.
In this way we obtain finally, by a rather different and more complicated route, all of the information regarding 
the critical modes $\lambda_j$ that was obtained in \cite{M1,M2,S1} for the Swift--Hohenberg case.

In particular, as in \cite{M1,M2,S1}, we rigorously validate the predictions of the formal Ginzburg Landau approximation,
both for existence and stability of periodic solutions bifurcating from a constant, homogeneous state.
In \cite{M1,M2,S1}, this was done by a posteriori comparison of the two sets of results, obtained by 
apparently quite different computations.
Here, with an eye toward greater generality, we explore this issue further, seeking an equivalence also at the
level of computations.
Namely, we show in Section \ref{s:GLcompare} that {\it after appropriate preconditioning,} the two processes of
Ginzburg Landau approximation and rigorous Lyapunov--Schmidt reduction may be matched exactly at each order of
$\eps$ up the the top order $\eps^3$ involved in the Ginzburg Landau expansion: that is, the two methods involve 
solving identical sets of equations with identical compatibility conditions at each step.
This not only verifies their correspondence, but indicates the mechanism by which it arises.

The preconditioning steps are, on the Lyapunov--Schmidt side, to impose the Ginzburg Landau scaling \eqref{dscale},
and, on the Ginzburg Landau side, to make 
the ansatz $A=A_{\omega}+e^{i\omega\hat x}(b_r-ib_i)$, 
where $A_{\omega}(\hat x)=\sqrt{\frac{2}{3}(1-16\omega^2)}e^{i\omega\hat x}$
is the background periodic solution,
observing that this reduces the linearized equations to constant coefficients and the operator $\partial_{\hat x}$ to
multiplication by $i\omega$.
With these adjustments, the correspondence between the two analyses is then revealed.

The latter observation so far as we know is new, and appears to point the way to a more general proof of correspondence not requiring computations; to carry this out in detail would be a very interesting direction for future study.  
The correspondence so obtained is at formal level: at $\eps^3$ order in the reduced spectral problem \eqref{formal},
ignoring higher-order truncation errors.
A second important open problem is to carry out a rigorous analysis, 
as here and in \cite{M1,M2,S1}, (i) verifying that these higher-order truncation errors result in acceptable, higher-order
approximation errors in the resulting eigenmodes $\hat\lambda_j(\hat \sigma)$, and (ii) justifying by separate analysis the
reduction to the ``Ginzburg Landau regime'' $|\hat \lambda, \hat \sigma|\leq C$.
This would have the important contribution of validation/illumination of the 
formal Ginzburg Landau approximation, commonly used
without proof to study stability in studies both mathematical and physical.
See for example \cite{PP-G,GLSS} and references therein.
We hope that our analysis here will serve as a useful blueprint for this more general case.


\section{Existence of periodic solutions from Turing instability}\label{s:existence}
In this section we study existence of periodic solutions, carrying out the proof of Theorem \ref{existence}.
Setting $D_1=4$, $D_2=16$ and $a=2$ (just for convenience, no chemical reason), the Brusselator model is
\be \label{Brusselator model}
\begin{split}
\partial_t u_1 & = 4\partial_x^2u_1 +2-(\b+1)u_1+u_1^2u_2 \\
\partial_t u_2 & =16\partial_x^2 u_2+ \b u_1-u_1^2u_2
\end{split}
\ee
Then the uniform state is $u^*=(2, \frac{\b}{2})$, where $\b$ is the bifurcating parameter. In this section, we consider existence of periodic solutions bifurcating from $u^*=(2, \frac{\b}{2})$ by Turing instability.

\bigskip

We start with the Turing instability. Linearizing the Brusselator model around $u^*$, we have the Jacobian matrix
\be
A=df(u^*)= \bp \b-1 & 4 \\ -\b & -4 \ep.
\notag
\ee
 To satisfy the conditions of Turing instability, $A$ is stable, that is, $det A = 4 >0$ and $trA <0$ which means $\b<5$. In order to find the critical value $\b_c$ (that is, $u^*$ is not stable anymore at $\b=\b_c$ by adding diffusion terms, we consider $det (A-k^2D)=0$ for some wave number $k \neq 0$. By a simple calculation, Turing instability occurs at $\b=4$ with the corresponding wave number $k=\pm \frac{1}{2}$. (That is, the parameter $\b$ should satisfy $4< \b < 5$ and $\b_c= 4$).

\bigskip

Now, we consider the normalized model (obtained by translating $u_1 \longrightarrow u_1+2$, $u_2 \longrightarrow u_2+\frac{\beta}{2}$, $\beta \longrightarrow b+4$):
\be \label{New Brusselator model}
\begin{split}
\partial_t u_1 & = 4\partial_x^2u_1 + (b+3)u_1 + 4u_2 +(\frac{b}{2}+2)u_1^2 + 4u_1u_2+ u_1^2 u_2 ,\\
\partial_t u_2 & =16\partial_x^2u_2 - (b+4)u_1 - 4u_2 -(\frac{b}{2}+2)u_1^2 - 4u_1u_2- u_1^2 u_2 .
\end{split}
\ee
Then  Turing instability occurs at $b=0$ with the corresponding wave number $k=\pm \frac{1}{2}$ and we consider a two-paramametric $(b, k)$ family of stationary solutions $\tilde u_{b,k}$ which bifurcate for $b=0$ from $u^*=(0,0)$. In order to show that there are bifurcating periodic stationary solutions from $u^*=(0, 0)$, we use the Lyapunov-Schmidt reduction.  Let $\tilde u(b,k,\xi)=(\tilde u_1, \tilde u_2)(b,k,\xi)$ be $2\pi$ - periodic in $\xi$ where $\xi=kx$ (that is, we assume $\frac{2\pi}{k}$ - periodic in $x$). We will look at the expression of the periodic solution $\tilde u$ in a neighborhood of $(b,k,u)=(0, \pm\frac{1}{2}, (0, 0))$. By \eqref{New Brusselator model},  $\tilde u$ satisfies
\be \label{B model}
0=N(b, k, \tilde u):=k^2D\partial_\xi^2 \tilde u + f(\tilde u),
\ee
where  $N: \RR^2 \times H^2_{per}([0,2\pi], \RR^2) \longrightarrow  L^2_{per}([0,2\pi], \RR^2) $ is a $C^{\infty}$ mapping and
\be
D= \bp 4 & 0 \\ 0 & 16 \ep,  \quad  \quad f(\tilde u)=\bp f_1(\tilde u_1, u_2) \\ f_2(\tilde u_1, \tilde u_2) \ep = \bp  (b+3)u_1 + 4u_2 +(\frac{b}{2}+2)u_1^2 + 4u_1u_2+ u_1^2 u_2  \\   -(b+4)u_1 - 4u_2 -(\frac{b}{2}+2)u_1^2 - 4u_1u_2- u_1^2 u_2  \ep.
\notag
\ee

\subsection{The Lyapunov-Schmidt reduction for the equation \eqref{B model}}

\noindent We first sketch the Lyapunov-Schmidt reduction for the equation \eqref{B model}. Since $N(0, \pm\frac{1}{2}, \bp 0 \\ 0 \ep) =0$, we want to study the stationary periodic solutions of the equation \eqref{B model} in a neighborhood of $(0, \pm\frac{1}{2}, \bp 0 \\ 0 \ep)$ in $\RR^2 \times H^2_{per}([0,2\pi], \RR^2)$. We define
\be
L_{per}:=\partial_{\tilde u}N(0, \pm\frac{1}{2}, \bp 0 \\ 0 \ep) = \ds\frac{1}{4}D\partial_\xi^2+A_c,
\ee
where
\be
D= \bp 4 & 0 \\ 0 & 16 \ep,  \quad A_c = \bp 3 & 4 \\ -4 & -4 \ep .
\ee
If $L_{per}$ is invertible and $L^{-1}_{per}$ is bounded from $L^2_{per}([0,2\pi], \RR^2)$ to $H^2_{per}([0,2\pi], \RR^2)$, then by the Implicit Function Theorem, in a neighborhood of $(0, \pm\frac{1}{2}, \bp 0 \\ 0 \ep)$, there exist a unique solution $\tilde u(\xi) = \phi (b,k)$ satisfying \eqref{B model} for some $C^{\infty}$ function $\phi$. In this case, however, $L_{per}$ is not invertible, so we apply the Lyapunov-Schmidt reduction. We first denote the kernel and range of $L_{per}$ by $\ker (L_{per})$ and $\ran (L_{per})$, respectively. Moreover, we assume the decompositions:
\be
H^2_{per}([0,2\pi], \RR^2)=\ker(L_{per}) \oplus X_1 \quad \text{and}  \quad L^2_{per}([0,2\pi], \RR^2) = \ran(L_{per}) \oplus Y_1,
\ee
where $X_1$ and $Y_1$ are topological complements of $\ker(L_{per})$ and $\ran (L_{per})$ in $H^2_{per}([0,2\pi], \RR^2)$ and $L^2_{per}([0,2\pi], \RR^2)$. Then there are two continuous projection $P:H^2_{per}([0,2\pi], \RR^2) \longrightarrow H^2_{per}([0,2\pi], \RR^2)$ and $Q : L^2_{per}([0,2\pi], \RR^2)  \longrightarrow L^2_{per}([0,2\pi], \RR^2)$ such that
\be \label{condition of projections}
\ran (P)=\ker (L_{per}) \quad \text{and} \quad \ker(Q)=\ran(L_{per}),
\ee
that is, $P(H^2_{per}([0,2\pi], \RR^2))=\ker(L_{per})$ and $(I-Q)(L^2_{per}([0,2\pi], \RR^2))=\ran(L_{per})$. Now we decompose $\tilde u - \bp 0 \\ 0 \ep \in H^2_{per}([0,2\pi], \RR^2)$ into $U + V$, where $u=P\Big(\tilde u - \bp 0 \\ 0 \ep \Big) \in \ker (L_{per})$ and $V=(I-P)\Big(\tilde u - \bp 0 \\ 0 \ep \Big) \in X_1$. Then one can rewrite the equation \eqref{B model} as
\be\label{two equations}
QN(b, k,  U +V) = 0, \quad (I-Q)N(b, k,  U +V) = 0.
\ee
We first focus on the second equation. Defining
\be
G(b, k, U, V) := (I-Q)N(b, k, U +V) = 0,
\ee
notice that $\ds G(0, \pm\frac{1}{2}, 0,0)=(I-Q)N(0, \pm\frac{1}{2}, \bp 0 \\ 0 \ep)=0$ and
\be
\partial_V G(0, \pm \frac{1}{2}, 0, 0)=(I-Q)\partial_{\tilde u}N(0, \pm\frac{1}{2}, \bp 0 \\ 0 \ep)= (I-Q)L_{per}=L_{per}.
\ee
Since $L_{per}:(I-P)H^2_{per}([0,2\pi], \RR^2)  \longrightarrow (I-Q)L^2_{per}([0,2\pi], \RR^2)$ is bijective, applying the Implicit Function Theorem, $G(b,k,U,V)$ can be solved for $V$ in $(I-P)H^2_{per}([0,2\pi], \RR^2)$ as a function of $(b,k,U)$. More precisely, there exist an open neighborhood $\Omega$ of $U=\bp 0 \\ 0 \ep$ in $\ker(L_{per})$, an open neighborhood $\Gamma$ of $(b,k)=(0, \pm\frac{1}{2})$ in $\RR^2$, and a $C^{\infty}$ function $\Phi : \Omega \times \Gamma \longrightarrow \ker(P) (=X_1) $ such that $\Phi (0, \pm\frac{1}{2}, 0) = 0$ and
\be
(I-Q)N(b,k,U+\Phi (b, k, U) ) \equiv 0,
\ee
for all $(b,k,U) \in \Gamma \times \Omega$. We now substitute $V=\Phi(b,k,U)$ into the first equation of \eqref{two equations} in order to obtain the bifurcation equation:
\be \label{GFE}
QN(b,k, U +\Phi (b, k, U) )= 0.
\ee
Setting
\be
B(b,k,U)=QN(b,k, U +\Phi (b, k, U) ),
\ee
B is a $C^{\infty}$ function from $\Gamma \times \Omega$ to $Y_1$ which has a finite dimension, $B(0, \pm\frac{1}{2},0)=0$ and $\partial_U B(0, \pm\frac{1}{2}, 0)=0$. Actually, solving \eqref{GFE} is equivalent to solving the original equation \eqref{B model}, that is, it is enough to solve the finite -dimensional problem $B(b,k,U)=0$ locally in $\RR^2 \times \ker( L_{per})$.

\begin{remark}
In the above argument, $(I-P)H^2_{per}([0,2\pi], \RR^2) =  \ker(P) = X_1$
\end{remark}

\subsection{Periodic solutions $\tilde u$ of \eqref{B model}}

By linearization of \eqref{B model} about $u^* =(0, 0)$, we have
\be \label{linearization}
\partial_{\tilde u}N(b, k, u^*)[U]=k^2D\partial_\xi^2 U+AU=(k^2D\partial_\xi^2+A)U,
\ee
where
\be
A=df(u^*)= \bp b+3 & 4 \\ -(b+4) & -4 \ep \quad \text{and} \quad U=\tilde u - u^*  \in \RR^2.
\notag
\ee
In particular, putting $b=b_c=0$, $k^2=\frac{1}{4}$, we have $A_c=\bp 3 & 4 \\ -4 & -4\ep$ and
\be \label{linear}
L_{per}U:=\partial_{\tilde u}N\left(0,\pm\frac{1}{2}, (0,0) \right)[U]=\Big(\frac{1}{4}D\partial_\xi^2+A_c\Big)U,
\ee
and the adjoint of $L_{per}$:
\be \label{adjoint of linear}
L^*_{per}U:= \Big(\frac{1}{4}D\partial_\xi^2+A_c^T \Big)U, \quad A_c^T: \text{transpose of $A_c$}.
\ee
Then the kernels of $L_{per}$ and $L^*_{per}$ are spanned by
\be \label{kernel}
U_1(\xi)=\cos \xi\bp 2 \\ -1 \ep \quad  \text{and} \quad U_2(\xi)=\sin \xi\bp 2 \\ -1 \ep,
\ee
\be \label{kernel of adjoint}
\bar U_1(\xi)=\cos \xi\bp 2 \\ 1 \ep \quad  \text{and} \quad \bar U_2(\xi)=\sin \xi\bp 2 \\ 1 \ep.
\ee
Now, in order to use Lyapunov-Schmidt reduction, we first define the zero eigenprojection
\be \label{projection P}
Qu= \langle \bar U_1, u\rangle U_1+\langle \bar U_2, u\rangle U_2, \quad \text{where} \quad \langle u, v \rangle = \frac{1}{3\pi} \int_0^{2\pi}u \cdot v d\xi,
\ee
and define the mapping
\be \label{projection tilde P}
\tilde  Q: L^2_{per}([0,2\pi], \RR^2) \rightarrow \RR^2 ; u \mapsto (\langle \bar U_1, u\rangle, \langle \bar U_2, u\rangle)^{T},
\ee
that is, $\tilde Q$ is just a vector form in $\RR^2$ of the projection $Q$.  
Decomposing $\tilde u  \in H^2_{per}([0,2\pi], \RR^2)$ into $\a_1 U_1+\a_2 U_2 + V$, where $PV=0$, we see that the linearization \eqref{linear} is invertible on $(I-P)H^2_{per}([0,2\pi], \RR^2)$. Moreover, recalling \eqref{B model}, we have
\be \label{projection equations}
\begin{split}
& \tilde  QN(b,k, \a_1 U_1+\a_2 U_2 + V )=0, \\
& (I-Q)N(b,k, \a_1 U_1+\a_2 U_2 + V)=0,
\end{split}
\ee
where
\be
(I-Q)N(b,k, \a_1 U_1+\a_2 U_2 + V):\mathbb{R}^4\times\ran(I-P)\to\ran(I-Q).
\ee

By the Implicit Function Theorem, there exists  an open neighborhood $U\subset\mathbb{R}^4$ of $(0,\frac{1}{2},0,0)$ and a unique function $V: U\to(I-P)H^2_{per}([0,2\pi], \RR^2)$ that solves the second equation of \eqref{projection equations} for $(b,k,\alpha_1,\alpha_2)\in U$. After we substitute $V$ into the first equation of \eqref{projection equations}, the reduced equation (or the bifurcation equation) will be $O(2)$ equivariant. This is due to the fact that the original problem is translation invariant and reflection symmetric. Hence, we can conclude that the reduced equation is of the form

\be\label{be}
 f(b,k,|\a|^2)\bp \a_1 \\ \a_2 \ep=0,
 \ee

  $f$ is a real-valued scalar function (c.f. \cite[Chapters 2,5]{CL}).

Next, let us find asymptotic expansion of $V$ with respect to parameter $\alpha_1$ and set $\alpha_2=0$.\\
1.) First of all, it is clear that $V(b,k,0)=0$.\\
2.) Now, we differentiate the second equation of \eqref{projection equations} with respect to $\alpha_1$.
\be\label{alpha1}
\begin{split}
& \partial_{\alpha_1}(I-Q)N(b,k, \a_1 U_1 + V)=(I-Q)[(k^2D\partial_\xi^2+ A)(U_1+\partial_{\alpha_1}V)\\&+  \bp (b+4) u_1+ 4 u_2+2 u_1  u_2  & 4 u_1+ u_1^2 \\ -(b+4) u_1- 4 u_2-2 u_1  u_2  & -4 u_1- u_1^2 \ep (U_1+\partial_{\alpha_1}V)]=0,
\end{split}
\ee
where $\bp u_1 \\ u_2 \ep=\a_1 U_1 + V$.\\
Hence, by step 1.),
\be\label{partial1}
\begin{split}
& \partial_{\alpha_1}(I-Q)N(b,k, \a_1 U_1 + V)|_{\alpha=0}=(I-Q)[(k^2D\partial_\xi^2+ A)(U_1+\partial_{\alpha_1}V|_{\alpha=0})=0.
\end{split}
\ee
Formula \eqref{partial1} implies that $(I-Q)[(k^2D\partial_\xi^2+ A)\partial_{\alpha_1}V|_{\alpha=0}=-(I-Q)[(k^2D\partial_\xi^2+ A)U_1$.\\
Notice that \be\label{U1}
D\bp 2 \\ -1 \ep = 8 \bp 1 \\ -2 \ep , \quad A\bp 2 \\ -1 \ep = \bp 2b+2 \\ -2b-4 \ep,
\ee

\be\label{QU1}
Q  k^2D\partial^2_\xi(  \a_1 U_1) =0, \quad Q A ( \a_1 U_1 ) = \frac{2b\a_1}{3} \bp 2 \\ -1 \ep \cos \xi,
\ee

\be
\begin{split}
&(I-Q) k^2D\partial^2_\xi(  \a_1 U_1) =k^2D\partial^2_\xi(  \a_1 U_1)=-8k^2\a_1  \bp 1 \\ -2 \ep \cos \xi ,\\
& (I-Q) A ( \a_1 U_1 ) = \frac{2b+6}{3}  \a_1\bp 1 \\ -2 \ep \cos \xi ,
\end{split}
\ee
Therefore,
\be\label{partial11}
(I-Q)[(k^2D\partial_\xi^2+ A)\partial_{\alpha_1}V|_{\alpha=0}=-(-8k^2+\frac{2b+6}{3})\bp 1 \\ -2 \ep \cos \xi.
\ee
Since $\mathbb{R}^2\cos \xi$ is an invariant subspace for the invertible operator $(I-Q)(k^2D\partial_\xi^2+ A)(I-P)$, $\partial_{\alpha_1}V|_{\alpha=0}$ is of the form $\bp a \\ b \ep \cos \xi$. Also, since $\partial_{\alpha_1}V|_{\alpha=0}\in\ran(I-P)$, $\partial_{\alpha_1}V|_{\alpha=0}\in\ker P$. Therefore, $\bp a \\ b \ep$ should be orthogonal to the vector $\bp 2 \\ 1 \ep$, which means that $\partial_{\alpha_1}V|_{\alpha=0}=f_1\bp 1 \\ -2 \ep \cos \xi$. \\
Next, note that
\be
D\bp 1 \\ -2 \ep =  \bp 4 \\ -32 \ep , \quad A\bp 1 \\ -2 \ep = \bp b-5 \\ -b+4 \ep.
\ee

\be\label{1,-8}
Q \Big[D\bp 1 \\ -2 \ep \cos \xi \Big] =  -8   \bp 2 \\ -1 \ep\cos \xi, \quad (I-Q) \Big[D\bp 1 \\ -2 \ep \cos \xi \Big] =  20  \bp 1 \\ -2 \ep\cos \xi
\ee

\be\label{1,-2}
Q \Big[A\bp 1 \\ -2 \ep \cos \xi \Big] = \frac{b-6}{3}  \bp 2 \\ -1 \ep\cos \xi,  \quad (I-Q) \Big[A\bp 1 \\ -2 \ep \cos \xi \Big] = \frac{b-3}{3} \bp 1 \\ -2 \ep \cos \xi.
\ee

Using \eqref{partial11}-\eqref{1,-2}, we derive that
\be
(-20k^2 + \frac{b-3}{3})f_1=-(-8k^2+\frac{2b+6}{3}).
\ee
Hence,
\be
\partial_{\alpha_1}V|_{\alpha=0}=\frac{-8k^2+\frac{2b+6}{3}}{20k^2 - \frac{b-3}{3}}\bp 1 \\ -2 \ep \cos \xi.
\ee

So far, we have shown that $V(b,k,\a_1,0)=f_1\cos \xi\bp 1 \\ -2 \ep \alpha_1+\mathcal{O}(|\alpha_1|^2)$.\\
3.) Next, we would like to compute $\partial^2_{\alpha_1}V|_{\alpha=0}$.\\
We differentiate \eqref{alpha1} with respect to $\alpha_1$.\\
\be\label{alpha1^2}
\begin{split}
& \partial^2_{\alpha_1}(I-Q)N(b,k, \a_1 U_1 + V)=(I-Q)[(k^2D\partial_\xi^2+ A)\partial^2_{\alpha_1}V\\&+  \bp (b+4)+2 u_2  & 0 \\ -(b+4) -2   u_2  & 0 \ep (U_1+\partial_{\alpha_1}V)^{\circ2}+ \bp 4 +2 u_1  & 4 + 2u_1 \\ - 4 -2 u_1  & -4 - 2u_1 \ep \left\{(\hat U_1+\partial_{\alpha_1}\hat V)\circ(U_1+\partial_{\alpha_1}V)\right\}\\
&+\bp (b+4) u_1+ 4 u_2+2 u_1  u_2  & 4 u_1+ u_1^2 \\ -(b+4) u_1- 4 u_2-2 u_1  u_2  & -4 u_1- u_1^2 \ep \partial^2_{\alpha_1}V]=0,
\end{split}
\ee
where $\circ$ indicates the Hadamard product sign.
Therefore,
\be\label{partial1^2}
\begin{split}
& \partial^2_{\alpha_1}(I-Q)N(b,k, \a_1 U_1 + V)|_{\alpha=0}=(I-Q)[(k^2D\partial_\xi^2+ A)\partial^2_{\alpha_1}V\\&+  \bp (b+4)  & 0 \\ -(b+4)    & 0 \ep (U_1+\partial_{\alpha_1}V|_{\alpha=0})^{\circ2}+ \bp 4   & 4  \\ - 4   & -4  \ep \left\{(\hat U_1+\partial_{\alpha_1}\hat V)\circ(U_1+\partial_{\alpha_1}V)\right\}]=0.
\end{split}
\ee

Notice that
\be\label{cos^2}
\begin{split}
& U_1+\partial_{\alpha_1}V|_{\alpha=0}=\bp 2 \\ -1 \ep \cos \xi+f_1\bp 1 \\ -2 \ep \cos \xi=\bp 2+f_1 \\ -1-2f_1 \ep \cos \xi,\\
&(U_1+\partial_{\alpha_1}V|_{\alpha=0})^{\circ2}=\bp (2+f_1)^2 \\ (-1-2f_1)^2 \ep \cos^2 \xi,\\
&(\hat U_1+\partial_{\alpha_1}\hat V|_{\alpha=0})\circ(U_1+\partial_{\alpha_1}V|_{\alpha=0})=\bp -1-2f_1 \\ 2+f_1 \ep \cos \xi\bp 2+f_1 \\ -1-2f_1 \ep \cos \xi\\
&=\bp(-1-2f_1)( 2+f_1) \\ (-1-2f_1)( 2+f_1) \ep \cos \xi=\bp-2-5f_1-2f_1^2 \\ -2-5f_1-2f_1^2  \ep \cos^2 \xi.
\end{split}
\ee
Hence,
\be
\begin{split}
&\bp (b+4)  & 0 \\ -(b+4)    & 0 \ep (U_1+\partial_{\alpha_1}V|_{\alpha=0})^{\circ2}+ \bp 4   & 4  \\ - 4   & -4  \ep \left\{(\hat U_1+\partial_{\alpha_1}\hat V)\circ(U_1+\partial_{\alpha_1}V)\right\}\\
&=(b+4)\bp (2+f_1)^2 \\- (2+f_1)^2 \ep \cos^2 \xi+8\bp-2-5f_1-2f_1^2 \\ 2+5f_1+2f_1^2  \ep \cos^2 \xi\\
&=((b-12)f_1^2+(4b-24)f_1+4b)\bp 1 \\ -1  \ep \cos^2 \xi.
\end{split}
\ee
It is easy to see that 
\be
Q \Big[\bp 1 \\ -1 \ep \Big] = 0, \quad Q \Big[\bp 1 \\ -1 \ep \cos 2\xi \Big] = 0.
\ee
Therefore, using \eqref{partial1^2}, we obtain
\be
\begin{split}\label{formula1^2a}
&(I-Q)(k^2D\partial_\xi^2+ A)\partial^2_{\alpha_1}V|_{\alpha=0}=-((b-12)f_1^2+(4b-24)f_1+4b)\bp 1 \\ -1  \ep \cos^2 \xi.
\end{split}
\ee
Since $\mathbb{R}^2$ and $\mathbb{R}^2\cos 2\xi$ are invariant subspaces for the invertible operator $(I-Q)(k^2D\partial_\xi^2+ A)(I-P)$, $\partial^2_{\alpha_1}V|_{\alpha=0}$ is of the form $\bp a \\ b \ep +\bp \tilde a \\ \tilde b \ep \cos 2\xi$. Note that $\bp a \\ b \ep$ and $\bp \tilde a \\ \tilde b \ep \cos 2\xi$ belong to $\ran(I-P)$. It follows from \eqref{formula1^2a} that
\be
\begin{split}
& A\bp a \\ b \ep=-\frac{1}{2}f_2\bp 1 \\ -1  \ep,\\
& (I-Q)(k^2D\partial_\xi^2+ A)\bp \tilde a \\ \tilde b \ep \cos 2\xi=-\frac{1}{2}f_2\bp 1 \\ -1  \ep\cos 2\xi,
\end{split}
\ee
where $f_2=(b-12)f_1^2+(4b-24)f_1+4b$.
Therefore,
\be\label{ab}
\bp a \\ b \ep=-\frac{1}{4}f_2\bp 0 \\ 1  \ep. 
\ee
And\\
\be
\begin{split}
& (-4k^2\bp 4\tilde a \\ 16\tilde b \ep+\bp (b+3)\tilde a+4\tilde b \\ -(b+4)\tilde a-4\tilde b \ep) \cos 2\xi=-\frac{1}{2}f_2\bp 1 \\ -1  \ep\cos 2\xi.
\end{split}
\ee
Hence, 
\be\label{abtilde}
\begin{split}
\bp \tilde a \\ \tilde b \ep =\frac{f_2}{-256k^4+16bk^2+32k^2-1}\bp -16k^2 \\ \frac{1+16k^2}{4} \ep.
\end{split}
\ee
Collecting terms from \eqref{ab} and \eqref{abtilde}, we arrive at 
\be
\begin{split}
 &\partial^2_{\alpha_1}V|_{\alpha=0}=-\frac{1}{4}f_2\bp 0 \\ 1  \ep+\bp \tilde a \\ \tilde b \ep\cos 2\xi.
 \end{split}
 \ee

4.) Next, we compute  $\partial^3_{\alpha_1}V|_{\alpha=0}$.\\
We differentiate \eqref{alpha1^2} with respect to $\alpha_1$.\\
\be\label{alpha1^3}
\begin{split}
& \partial^3_{\alpha_1}(I-Q)N(b,k, \a_1 U_1+ V)=(I-Q)[(k^2D\partial_\xi^2+ A)\partial^3_{\alpha_1}V\\
&+  \bp 2 \partial_{\alpha_1}u_2  & 0 \\ -2  \partial_{\alpha_1}  u_2  & 0 \ep (U_1+\partial_{\alpha_1}V)^{\circ2}+\bp (b+4)+2 u_2  & 0 \\ -(b+4) -2   u_2  & 0 \ep2\left\{(U_1+\partial_{\alpha_1}V)\circ\partial^2_{\alpha_1}V\right\}
\\&+ \bp 2\partial_{\alpha_1} u_1  & 2\partial_{\alpha_1}u_1 \\-2\partial_{\alpha_1} u_1  & - 2\partial_{\alpha_1}u_1 \ep \left\{(\hat U_1+\partial_{\alpha_1}\hat V)\circ(U_1+\partial_{\alpha_1}V)\right\}\\
&+\bp 4 +2 u_1  & 4 + 2u_1 \\ - 4 -2 u_1  & -4 - 2u_1 \ep [\partial^2_{\alpha_1}\hat V\circ(U_1+\partial_{\alpha_1}V)+(\hat U_1+\partial_{\alpha_1}\hat V)\circ\partial^2_{\alpha_1}V]\\
&+\bp (b+4)\partial_{\alpha_1} u_1+ 4\partial_{\alpha_1} u_2+2 \partial_{\alpha_1}u_1  u_2+2u_1\partial_{\alpha_1}  u_2  & 4\partial_{\alpha_1} u_1+ 2u_1\partial_{\alpha_1}u_1 \\ -(b+4)\partial_{\alpha_1} u_1- 4\partial_{\alpha_1} u_2-2 \partial_{\alpha_1}u_1  -u_2-2u_1\partial_{\alpha_1}  u_2  & -4\partial_{\alpha_1} u_1- 2u_1\partial_{\alpha_1}u_1  \ep \partial^2_{\alpha_1}V\\
&+\bp (b+4) u_1+ 4 u_2+2 u_1  u_2  & 4 u_1+ u_1^2 \\ -(b+4) u_1- 4 u_2-2 u_1  u_2  & -4 u_1- u_1^2 \ep \partial^3_{\alpha_1}V]=0.
\end{split}
\ee

Therefore,
\be\label{partial1^3}
\begin{split}
& \partial^3_{\alpha_1}(I-Q)N(b,k, \a_1 U_1 + V)|_{\alpha=0}=(I-Q)[(k^2D\partial_\xi^2+ A)\partial^3_{\alpha_1}V\\
&+  \bp -2(1+2f_1)\cos \xi  & 0 \\ 2(1+2f_1)\cos \xi  & 0 \ep \bp (2+f_1)^2 \\ (1+2f_1)^2 \ep \cos^2 \xi\\
&+2\bp (b+4)  & 0 \\ -(b+4)  & 0 \ep \left\{\bp 2+f_1 \\ -1-2f_1 \ep \cos \xi\circ(-\frac{1}{4}f_2\bp 0 \\ 1  \ep+\bp \tilde a \\ \tilde b \ep\cos 2\xi)\right\}
\\&+ \bp 2(2+f_1)\cos \xi  & 2(2+f_1)\cos \xi \\-2(2+f_1)\cos \xi   & - 2(2+f_1)\cos \xi \ep \bp-2-5f_1-2f_1^2 \\ -2-5f_1-2f_1^2  \ep \cos^2 \xi\\
&+\bp 4   & 4  \\ - 4   & -4  \ep [(-\frac{1}{4}f_2\bp 1 \\ 0  \ep+\bp \tilde b \\ \tilde a \ep\cos 2\xi)\circ\bp 2+f_1 \\ -1-2f_1 \ep \cos \xi\\
&+\bp -1-2f_1 \\ 2+f_1 \ep \cos \xi\circ(-\frac{1}{4}f_2\bp 0 \\ 1  \ep+\bp \tilde a \\ \tilde b \ep\cos 2\xi)]\\
&+\bp (b+4)(2+f_1)\cos \xi- 4(1+2f_1)\cos \xi  & 4(2+f_1)\cos \xi \\ -(b+4)(2+f_1)\cos \xi+ 4(1+2f_1)\cos \xi   & -4(2+f_1)\cos \xi  \ep (-\frac{1}{4}f_2\bp 0 \\ 1  \ep+\bp \tilde a \\ \tilde b \ep\cos 2\xi)]
\end{split}
\ee

\be\label{partial1^3c}
\begin{split}
&=(I-Q)[(k^2D\partial_\xi^2+ A)\partial^3_{\alpha_1}V-2(1+2f_1)(2+f_1)^2\bp 1 \\ -1 \ep (\frac{1}{4}\cos 3\xi+\frac{3}{4}\cos \xi)\\
& +2\bp (b+4)  & 0 \\ -(b+4)  & 0 \ep ( \frac{1}{4}f_2(1+2f_1)\bp 0 \\ 1  \ep\cos \xi+\bp (2+f_1)\tilde a \\ (-1-2f_1)\tilde b \ep\frac{1}{2}(\cos \xi +\cos 3\xi))
\\&-4(2+f_1)^2(1+2f_1) \bp 1 \\ -1  \ep (\frac{1}{4}\cos 3\xi+\frac{3}{4}\cos \xi)\\
&+\bp 4   & 4  \\ - 4   & -4  \ep [-\frac{1}{4}f_2(2+f_1)\bp 1 \\ 1  \ep\cos \xi+((2+f_1)\tilde b-(1+2f_1)\tilde a)\bp 1 \\ 1 \ep \frac{1}{2}(\cos \xi +\cos 3\xi)]\\
&-(2+f_1)f_2\bp 1 \\ -1  \ep\cos \xi+((b+4)(2+f_1)\tilde a- 4(1+2f_1)\tilde a+4(2+f_1)\tilde b)\bp 1 \\ -1  \ep\frac{1}{2}(\cos \xi +\cos 3\xi)]\\
&=(I-Q)[(k^2D\partial_\xi^2+ A)\partial^3_{\alpha_1}V+\bp 1 \\ -1 \ep\cos \xi\big(-\frac{9}{2}(1+2f_1)(2+f_1)^2+(b+4)(2+f_1)\tilde a-3(2+f_1)f_2\\&+4\{(2+f_1)\tilde b-(1+2f_1)\tilde a\}+\frac{1}{2}\{(b+4)(2+f_1)\tilde a- 4(1+2f_1)\tilde a+4(2+f_1)\tilde b\}\big)+\bp * \\ * \ep\cos 3\xi]\\
&=(I-Q)[(k^2D\partial_\xi^2+ A)\partial^3_{\alpha_1}V+f_3\bp 1 \\ -1 \ep\cos \xi+\bp * \\ * \ep\cos 3\xi]=0,
\end{split}
\ee
where 
\be\label{partial1^3cc}
\begin{split}
f_3&=-\frac{9}{2}(1+2f_1)(2+f_1)^2+(b+4)(2+f_1)\tilde a-3(2+f_1)f_2\\&+4\{(2+f_1)\tilde b-(1+2f_1)\tilde a\}+\frac{1}{2}\{(b+4)(2+f_1)\tilde a- 4(1+2f_1)\tilde a+4(2+f_1)\tilde b\}\\
&=-\frac{9}{2}(1+2f_1)(2+f_1)^2+\frac{3}{2}(b+4)(2+f_1)\tilde a-3(2+f_1)f_2\\&+6\{(2+f_1)\tilde b-(1+2f_1)\tilde a\}.
\end{split}
\ee
Next, it is easy to see that 
\be\label{qcos}
Q \Big[\bp 1 \\ -1 \ep\cos \xi \Big] = \frac{1}{3}U_1, \quad (I-Q) \Big[\bp 1 \\ -1 \ep \cos \xi \Big] = \frac{1}{3}\bp 1 \\ -2 \ep\cos \xi.
\ee
It follows from the last line of \eqref{partial1^3c} and \eqref{qcos} that
\be\label{partial1^3ccc}
\begin{split}
(I-Q)[(k^2D\partial_\xi^2+ A)\partial^3_{\alpha_1}V]=-\frac{1}{3}f_3\bp 1 \\ -2 \ep\cos \xi+\bp * \\ * \ep\cos 3\xi
\end{split}
\ee

Since $\mathbb{R}^2\cos \xi$ and $\mathbb{R}^2\cos 3\xi$ are invariant subspaces for the invertible operator $(I-Q)(k^2D\partial_\xi^2+ A)(I-P)$, $\partial^3_{\alpha_1}V|_{\alpha=0}$ is of the form $f_4\bp 1 \\ -2 \ep\cos \xi +\bp * \\ * \ep \cos 3\xi$.

Using \eqref{partial11}-\eqref{1,-2} and \eqref{partial1^3ccc}, we derive at
\be
(-20k^2 + \frac{b-3}{3})f_4=-\frac{1}{3}f_3.
\ee
Hence,
\be\label{partial1^3cccc}
\partial^3_{\alpha_1}V|_{\alpha=0}=\frac{\frac{1}{3}f_3}{20k^2 - \frac{b-3}{3}}\bp 1 \\ -2 \ep \cos \xi+\bp * \\ * \ep\cos 3\xi.
\ee

So far, we have shown that 
 \be
 \begin{split}\label{V}
 &V(b,k,\a_1,0)=f_1\cos \xi\bp 1 \\ -2 \ep \alpha_1+\frac{1}{2}(-\frac{1}{4}f_2\bp 0 \\ 1  \ep+\bp \tilde a \\ \tilde b \ep\cos 2\xi)\alpha_1^2\\
 &+\frac{1}{6}(\frac{\frac{1}{3}f_3}{20k^2 - \frac{b-3}{3}}\bp 1 \\ -2 \ep \cos \xi+\bp * \\ * \ep\cos 3\xi)\alpha_1^3+\mathcal{O}(|\a_1|^4).
 \end{split}
 \ee

In order to obtain the reduced equation, we substitute \eqref{V} into the first equation from \eqref{projection equations}, obtaining
for $\bp u_1 \\ u_2 \ep=\a_1 U_1 + V$ the equation

\be\label{QN}
 \tilde QN(b,k, \a_1 U_1 + V) =\tilde Q \Big[(k^2D\partial_\xi^2+ A)(\a_1 U_1 + V)
+  \bp  (\frac{b}{2}+2)u_1^2 + 4u_1u_2+ u_1^2 u_2  \\    -(\frac{b}{2}+2)u_1^2 - 4u_1u_2- u_1^2 u_2  \ep  \Big]=0.
\ee

Next, we split the left-hand side of \eqref{QN} into two parts.

a.) {\bf Linear part.}
In order to treat linear terms, we use the following computations.
\be
\begin{split}
	Q  k^2D\partial^2_\xi(  \a_1 U_1 ) &=0, \quad Q A ( \a_1 U_1 ) = \frac{2b\a_1}{3} \bp 2 \\ -1 \ep \cos \xi,\\
	Q\Big[D\bp 1 \\ -2 \ep \cos \xi \Big] &=  -8   \bp 2 \\ -1 \ep\cos \xi,\,\,\,
	Q \Big[A\bp 1 \\ -2 \ep \cos \xi \Big] = \frac{b-6}{3}  \bp 2 \\ -1 \ep\cos \xi.
\end{split}
\ee

Therefore, we have the following expression for the linear part from \eqref{QN}
\be
\begin{split}
& \tilde QN(b,k, \a_1 U_1+ V)=\tilde Q(k^2D\partial_\xi^2+ A)(\a_1 U_1 + V)= \frac{2}{3}b  \a_1 +( 8k^2+\frac{b-6}{3})f_1 \a_1\\& +
\frac{1}{6}( 8k^2+\frac{b-6}{3})\frac{\frac{1}{3}f_3}{20k^2 - \frac{b-3}{3}}  \a_1^3 +\mathcal{O}(|\a_1|^4).
\end{split}
\ee

b.) {\bf Non-linear part.}
We next treat the nonlinear terms
\be\label{nonlinear}
\begin{split}
&\tilde Q\bp  (\frac{b}{2}+2)u_1^2 + 4u_1u_2+ u_1^2 u_2  \\    -(\frac{b}{2}+2)u_1^2 - 4u_2u_2- u_1^2 u_2  \ep =\tilde Q\big( ((\frac{b}{2}+2)u_1^2 + 4u_1u_2+ u_1^2 u_2)\bp 1 \\ -1 \ep\big)\\
&=\tilde Q\big((\frac{b}{2}+2)\{(2+f_1)\a_1\cos \xi+\frac{1}{2}\tilde a\alpha_1^2\cos 2\xi\\
 &+\mathcal{O}(|\a_1|^3)\}^2 + 4\{(2+f_1)\a_1\cos \xi+\frac{1}{2}\tilde a\alpha_1^2\cos 2\xi\\
  &+\mathcal{O}(|\a_1|^3)\}\{(-1-2f_1)\a_1\cos \xi-\frac{1}{8}f_2\alpha_1^2+\frac{1}{2}\tilde b\alpha_1^2\cos 2\xi\\
    &+\mathcal{O}(|\a_1|^3)\}+ \{(2+f_1)\a_1\cos \xi+\mathcal{O}(|\a_1|^2)\}^2\{(-1-2f_1)\a_1\cos \xi\\
    &+\mathcal{O}(|\a_1|^2)\})\bp 1 \\ -1 \ep\big)\\
    &=\tilde Q\big((\frac{b}{2}+2)\{(2+f_1)\tilde a\a_1\alpha_1^2\cos \xi\cos 2\xi+\mathcal{O}(|\a_1|^4)\} \\&+ 4\{(2+f_1)\a_1\cos \xi(-\frac{1}{8}f_2\alpha_1^2+\frac{1}{2}\tilde b\alpha_1^2\cos 2\xi)\\
     &+\frac{1}{2}\tilde a\alpha_1^2\cos 2\xi(-1-2f_1)\a_1\cos \xi+\mathcal{O}(|\a_1|^4)\}\\
      &+ \{(2+f_1)^2(-1-2f_1)\a^3_1\cos^3 \xi+\mathcal{O}(|\a_1|^4)\})\bp 1 \\ -1 \ep\big).
\end{split}
\ee    
By direct computation, we have:
\be\label{fs}
\begin{split}
	Q\Big[D\bp 1 \\ -1 \ep \cos \xi \Big] = \frac{1}{3}U_1 ,\,\,\,
Q\Big[D\bp 1 \\ -1 \ep \cos \xi\cos 2\xi \Big] = \frac{1}{6}U_1 ,\,\,\,
Q\Big[D\bp 1 \\ -1 \ep \cos^3 \xi \Big] = \frac{1}{4}U_1.\,\,\,
\end{split}
\ee
   Using formulas \eqref{nonlinear} and \eqref{fs}, we have the following expression for the non-linear part from \eqref{QN}
\be
\begin{split}\label{f5}      
&\tilde Q\big( ((\frac{b}{2}+2)u_1^2 + 4u_1u_2+ u_1^2 u_2)\bp 1 \\ -1 \ep\big)\\
&=\frac{1}{6}(\frac{b}{2}+2)(2+f_1)\tilde a \a_1^3+4\{\frac{1}{-24}(2+f_1)f_2\a_1^3\\
&+\frac{1}{12}(2+f_1)\tilde b\a_1^3+\frac{1}{12}(-1-2f_1)\tilde a\a_1^3\}\\
&+\frac{1}{4}(2+f_1)^2(-1-2f_1)\a_1^3+\mathcal{O}(|\a_1|^5)\\
&=\frac{1}{6}f_5\a_1^3++\mathcal{O}(|\a_1|^5).
\end{split}
\ee
Hence, taking into account formulas \eqref{f5} and \eqref{be}, the reduced equation has the form:
\be\label{reduced}
\begin{split}
&\big\{ \frac{2}{3}b  +( 8k^2+\frac{b-6}{3})f_1   + 
\frac{1}{6}f_5\big(1+\frac{8k^2+\frac{b-6}{3}}{20k^2 - \frac{b-3}{3}} \big)(\a_1^2+\a_2^2)  +\mathcal{O}(|\a|^4)\big\}\bp \a_1 \\ \a_2 \ep=0.
\end{split}
\ee

From now on, we take without loss of generality $\a_2=0$ and $\a_1=\a$. Our goal is to solve \eqref{reduced} for $\a$ in terms of $b$ and $k$. 
Let us introduce $\mathcal{A}$:
\be
\begin{split}
\mathcal{A}=&\frac{\frac{2}{3}b  +( 8k^2+\frac{b-6}{3})f_1}{-\frac{1}{6}f_5\big(1+\frac{8k^2+\frac{b-6}{3}}{20k^2 - \frac{b-3}{3}}\big)}  ,
\end{split}
\ee
or
\be
\begin{split}
\mathcal{A}=&\frac{\frac{48k^2}{60k^2-b+3}}{-\frac{1}{6}f_5\big(1+\frac{8k^2+\frac{b-6}{3}}{20k^2 - \frac{b-3}{3}}\big)} (b-\frac{(4k^2-1)^2}{4k^2}) .
\end{split}
\ee
Notice that 
\be\label{2/3}
\begin{split}
\frac{\frac{48k^2}{60k^2-b+3}}{-\frac{1}{6}f_5\big(1+\frac{8k^2+\frac{b-6}{3}}{20k^2 - \frac{b-3}{3}}\big)}=\frac{2}{3}+\mathcal{O}((k-\frac{1}{2}),b) .
\end{split}
\ee
Solving \eqref{reduced} is equivalent to solving
\be\label{short}
\begin{split}
&\mathcal{A}-\a^2  +\mathcal{O}(|\a|^4)=0.
\end{split}
\ee
Next, plugging $\a=\sqrt{|\mathcal{A}|}\mathcal{B}$ into \eqref{short}, we obtain
\be
\begin{split}
&\mathcal{A}-|\mathcal{A}|\mathcal{B}^2  +\mathcal{O}(|\mathcal{A}|^2)=0,
\end{split}
\ee
or
\be
\begin{split}\label{B}
&\mathcal{A}(1-\mathcal{B}^2+\mathcal{O}(|\mathcal{A}|))=0\,\,\hbox{if}\,\, \mathcal{A}\geq0,\\
&\mathcal{A}(1+\mathcal{B}^2+\mathcal{O}(|\mathcal{A}|))=0\,\,\hbox{if}\,\, \mathcal{A}\leq0.
\end{split}
\ee
We need to solve \eqref{B} in terms of $\mathcal{A}$. The second equation in \eqref{B} has no solutions. By the Implicit Function Theorem, there exists  an open neighborhood $U\subset\mathbb{R}$ of $0$ and a unique function $\mathcal{B}: U\to\R$ that solves the first equation of \eqref{B} for $\mathcal{A}\in U$. Therefore, we have the restriction $\mathcal{A}\geq0$. Hence, using formula \eqref{2/3} and the restriction on $\mathcal{A}$, we conclude that $(b-\frac{(4k^2-1)^2}{4k^2})$ must be greater or equal to $0$ (note that $b$ must be greater or equal to $0$ as well). Next, we introduce a scaling parameter $\omega$ defined by the equation
\be
\begin{split}\label{w}
\frac{4k^2-1}{2k}=4\sqrt{b}\omega.
\end{split}
\ee
We can solve \eqref{w} for $k$, i.e.
\be
\begin{split}\label{k1}
k=\frac{2\sqrt{b}\omega\pm\sqrt{4b\omega^2+1}}{2}.
\end{split}
\ee
Then, 
\be
\begin{split}
\mathcal{A}\geq0\,\,\hbox{if and only if}\,\, k=\frac{2\sqrt{b}\omega+\sqrt{b\omega^2+1}}{2}\,\,\hbox{and}\,\,\omega\in[-\frac{1}{4},\frac{1}{4}].
\end{split}
\ee
Note that when $\omega=\pm\frac{1}{4}$, $\mathcal{A}=0$.
For convenience, we introduce $\eps$ defined by $\eps=\sqrt{b}$. Then,
\be
\begin{split}\label{k} k=\frac{2\eps\omega+\sqrt{4\eps^2\omega^2+1}}{2}\,\,\hbox{and}\,\,\omega\in[-\frac{1}{4},\frac{1}{4}].
\end{split}
\ee
Next, using the first equation in \eqref{B}, we arrive at the asymptotic formula for $\mathcal{B}$:
\be
\begin{split} 
\mathcal{B}=1+\mathcal{O}(\mathcal{A}),
\end{split}
\ee
which implies that
\be\label{alphatoo}
\a=\sqrt{|\mathcal{A}|}\mathcal{B}=\sqrt{\mathcal{A}}+\mathcal{O}(\mathcal{A}^{3/2}).
\ee
Since, $k$ and $b$ are functions of $\eps$. $\mathcal{A}$ is a function of $\eps$ as well. In particular,
\be
\begin{split} \label{Acal}
	\mathcal{A}=\frac{2}{3}(1-16\omega^2)\eps^2+\frac{20}{27}\omega(16\omega^2-1)\eps^3+\mathcal{O}(\eps^4).
\end{split}
\ee
Therefore, using \eqref{alphatoo}, we arrive at the asymptotic formula for $\a$:
\be
\begin{split} \label{alpha}
	\a=\sqrt{\frac{2}{3}(1-16\omega^2)}\eps-\frac{5}{9}\omega\sqrt{\frac{2}{3}(1-16\omega^2)}\eps^2+\mathcal{O}(\eps^3).
\end{split}
\ee
Note that when $\omega=\pm\frac{1}{4}$, $\a=0$.
By direct computation, we obtain the expansions:
\be\label{expansions}
\begin{split}
	2+f_1=2-\frac{4\omega}{3}\eps+\mathcal{O}(\eps^2) ,\,\,\,
	-1-2f_1=-1+\frac{8\omega}{3}\eps+\mathcal{O}(\eps^2),\\
	\tilde a=\frac{64\omega}{9}\eps+\mathcal{O}(\eps^2) ,\,\,\,
	\tilde b=-\frac{20\omega}{9}\eps+\mathcal{O}(\eps^2),\,\,f_2=16\omega\eps+\mathcal{O}(\eps^2) ,\,\,\,
\end{split}
\ee
Using formulas \eqref{V} and \eqref{expansions}, we 
obtain the result of Theorem \ref{existence}.


\section{Stability of periodic solutions }\label{s:stability}
In this section we study stability of the bifurcating periodic solutions established in Section \ref{s:existence}, carrying out the proof of Theorem \ref{stabthm}.
Linearizing \eqref{Brusselator model} about $\tilde u_{\eps,\omega}$, we have
\be \label{linearization about tilde u}
\hat B_{\eps,\omega}(\partial_\xi)v:=k^2D\partial^2_\xi v + df(\tilde u_{\eps,\omega})v,
\ee
where
\be\label{df}
\begin{split}
df(\tilde u_{\eps,\omega})
& = \bp \eps^2+3 & 4 \\ -(\eps^2+4) & -4 \ep +   \bp (\eps^2+4)\tilde u_1+ 4\tilde u_2+2\tilde u_1 \tilde u_2  & 4\tilde u_1+\tilde u_1^2 \\ -(\eps^2+4)\tilde u_1- 4\tilde u_2-2\tilde u_1 \tilde u_2  & -4\tilde u_1-\tilde u_1^2 \ep \\
& = A + \bp (\eps^2+4)\tilde u_1+ 4\tilde u_2+2\tilde u_1 \tilde u_2  & 4\tilde u_1+\tilde u_1^2 \\ -(\eps^2+4)\tilde u_1- 4\tilde u_2-2\tilde u_1 \tilde u_2  & -4\tilde u_1-\tilde u_1^2 \ep \\
&=\bp 3 & 4 \\ -4 & -4 \ep+4\sqrt{\frac{2}{3}(1-16\omega^2)}\cos \xi\bp 1 & 2 \\ -1 & -2 \ep\eps+\{\bp 1& 0 \\ -1 & 0 \ep \\
&-\frac{4}{3}\sqrt{\frac{2}{3}(1-16\omega^2)}\bp -4& 4 \\ 4 & -4 \ep\cos \xi-\frac{20}{9}\omega\sqrt{\frac{2}{3}(1-16\omega^2)}\cos \xi\bp 1 & 2 \\ -1 & -2 \ep\\
&+\frac{2}{3}(1-16\omega^2)\bp -4& 4 \\ 4 & -4 \ep\cos^2 \xi\}\eps^2+\mathcal O(\eps^3).
\end{split}
\ee
Since $df(\tilde u_{\eps,\omega})$ is $2\pi$-periodic, every coefficient of the linear operator $\hat B_{\eps,\omega}$ is $2\pi$-periodic. By substituting $v(\xi)= e^{i \sigma \xi} V(\xi)$ we define the Bloch operator family: for $\sigma \in \RR$,
\be\label{Bloch}
B(\eps,\omega, \sigma)V= k^2(\eps,\omega)D(\partial_\xi + i\sigma)^2V + df(\tilde u_{\eps,\omega})V,
\ee
where $B(\eps,\omega,\sigma) : H^2_{per}[0,2\pi] \longrightarrow  L^2_{per} [0,2\pi]$ and $k=\frac{2\eps\omega+\sqrt{4\eps^2\omega^2+1}}{2}$. However, in order to study the spectral stability of $\tilde u_{\eps,\omega}$, it is enough to consider $\sigma \in [-\frac{1}{2},\frac{1}{2})$ because for any $\sigma \in \RR$, $\sigma = \sigma^*+m$, where $\sigma^* \in  [-\frac{1}{2},\frac{1}{2})$ and $m \in \ZZ$; hence we consider $e^{im\xi}V(\xi)$ instead of $V(\xi)$. We now  define the operator $B_0$:
\be
B_0(\sigma):=B(0,\omega,\sigma)=\frac{1}{4}D(\partial_\xi + i\sigma)^2 + \bp 3 & 4 \\ -4 & -4 \ep,
\ee
which has constant coefficients. Here, we consider Bloch operators $B(\eps,\omega,\sigma)$ as small perturbations of $B_0(\sigma)$. So we first study the eigenvalue problem of $B_0(\sigma)$:
\be
B_0(\sigma)\bp a_m \\ b_m \ep e^{im\xi} = \mu_m \bp a_m \\ b_m \ep e^{im\xi},
\notag
\ee
which is equivalent to the eigenvalue problem of the matrix $\mathcal{B}$:
\be \label{B0}
\mathcal{B} (\sigma,m) = \bp -(m+\sigma)^2+3 & 4 \\-4 & -4(m+\sigma)^2-4 \ep.
\ee
Since trace of $\mathcal{B}$ is negative, at least one of the eigenvalues of $\mathcal{B}$ has the negative real part. So we need to consider $\sigma \in [-\frac{1}{2},\frac{1}{2})$ such that the determinant of $\mathcal{B}(\sigma,m) = 0$. Notice that
\be
det \mathcal{B} =4 (m+\sigma)^4-8(m+\sigma)^2+4=4((m+\sigma)^2-1)^2.
\ee
$det \mathcal{B} = 0$ becomes $(m+\sigma)^2 = 1$. Since $m\in \ZZ$, the possible values of $m$ are $1$ and $-1$, and so we consider the following ``dangerous set:'' for some sufficiently small $\eta >0$.
\be
\Gamma= \{ \sigma | -\eta <\sigma < \eta \}.
\ee

Therefore, as long as $\sigma$ is bounded away from $0$, the real part of the spectrum of $B_0(\sigma)$ has negative upper bound. Similarly, one can show that the real part of the spectrum of the constant-coefficient operator $k^2(\eps,\omega)D(\partial_\xi + i\sigma)^2+\bp 3 & 4 \\ -4 & -4 \ep$ has negative upper bound (the bound might depend on $\eps$) if $\sigma$ is bounded away from $0$. Finally, since 
$df(\tilde u_{\eps,\omega})-\bp 3 & 4 \\ -4 & -4 \ep$ represents a bounded small perturbation,  the real part of the spectrum of $B(\eps,\omega,\sigma)$ has negative upper bound. for $\sigma \in [-\frac{1}{2},\frac{1}{2})\setminus \Gamma $.
\bigskip

\subsection{Stability of the bifurcating periodic solutions: coperiodic case $\sigma=0$}

We now consider the eigenvalue problem of $B(\eps,\omega,0)$:
\be
0=\Big[ B(\eps,\omega,0)-\l I \Big] W.
\ee
In order to use the Lyapunov-Schmidt reduction, we decompose $W = \b_1 U_{1}+\b_{2} U_{2} + \mathcal V$ and we first solve
\be\label{lambda}
\begin{split}
	0
	& =(I-Q)\Big[ B(\eps,\omega,0)-\l I \Big] (\b_1 U_{1}+\b_{2} U_{2} + \mathcal V),
\end{split}
\ee

where
\be
(I-Q)\Big[ B(\eps,\omega,0)-\l I \Big] (\b_1 U_{1}+\b_{2} U_{2} + \mathcal V):\mathbb{R}\times\mathbb{C}\times\mathbb{R}^2\times\ran(I-P)\to\ran(I-Q).
\ee

By the Implicit Function Theorem, there exists  an open neighborhood $U\subset\mathbb{R}\times\mathbb{C}\times\mathbb{R}^2$ of $(0,0,0,0)$ and a unique function $\mathcal V: U\to(I-P)H^2_{per}([0,2\pi], \RR^2)$ that solves  \eqref{lambda} for $(\eps,\lambda,\b_1,\b_2)\in U$.

Next, it is clear that the relation between $\b$ and $\mathcal V$ is linear. Then, let $\mathcal V(\eps,\omega,\lambda,\b)=\mathcal V_1(\eps,\omega,\lambda)\b_1+\mathcal V_2(\eps,\omega,\lambda)\b_2$. Now let us find asymptotic expansions of $\mathcal V_1$ and $\mathcal V_2$ with respect to parameter $\eps$.\\
1.) First, we compute $\mathcal V_i(0,\omega,\lambda)=\partial_{\b_i}\mathcal V|_{\eps=0}$. We differentiate \eqref{lambda} with respect $\b_i$ and plug in $0$ for $\eps$.

\be
\begin{split}
	0
	& =(I-Q)\Big[ B(0,\omega,0)-\l I \Big] (U_{1} + \partial_{\b_i}\mathcal V|_{\eps=0}).
\end{split}
\ee
Notice that $ B(0,\omega,0)  U_{1}=L_{per}U_{1}=0$ and $(I-Q)U_1=0$. Since $(I-Q)\Big[ B(0,\omega,0)-\l I \Big](I-P)$ is invertible for small values of $\lambda$, we conclude that
\be
\mathcal V_i(0,\omega,\lambda)=\partial_{\b_i}\mathcal V|_{\eps=0}=0.
\ee
2.) Now, we differentiate the second equation of \eqref{lambda} with respect to $\b_1$ and $\eps$ and, then, plug in $0$ for $\eps$.
Note that it follows from \eqref{k} that
\be\label{kk}
k=\frac{1}{2}+\omega\eps+\mathcal O(\eps^2),\,\,\,k^2=\frac{1}{4}+\omega\eps+2\omega^2\eps^2+\mathcal O(\eps^3).
\ee
Therefore,
\be
\begin{split}
	0
	& =(I-Q)\partial_{\eps}B(0,\omega,0)U_{1} + (I-Q)(B(0,\omega,0)-\lambda)\partial_{\eps}\partial_{\b_1}\mathcal V|_{\eps=0},
\end{split}
\ee
or
\be
\begin{split}
	(I-Q)(B(0,\omega,0)-\lambda)\partial_{\eps}\partial_{\b_1}\mathcal V|_{\eps=0}
	& =-(I-Q)\partial_{\eps}B(0,\omega,0)U_{1}.
\end{split}
\ee
Taking into account formulas \eqref{U1}, \eqref{df}, \eqref{Bloch} and \eqref{kk}, we arrive at
\be\label{ebetta}
\begin{split}
	&-(I-Q)\partial_{\eps}B(0,\omega,0)U_{1}=-(I-Q)[\omega D\partial^2_\xi + 4\sqrt{\frac{2}{3}(1-16\omega^2)}\bp 1 & 2 \\ -1 & -2 \ep\cos \xi]U_1\\
	&=-\omega(I-Q) D\partial^2_\xi U_1=8\omega\bp 1 \\ -2 \ep\cos \xi.
\end{split}
\ee
Since $\mathbb{R}^2\cos \xi$ is an invariant subspace for the invertible operator $(I-Q)(B(0,\omega,0)-\lambda)(I-P)$, $\partial_{\eps}\partial_{\b_1}\mathcal V|_{\eps=0}$ is of the form $\bp a \\ b \ep \cos \xi$. Also, since $\partial_{\eps}\partial_{\b_1}\mathcal V|_{\eps=0}\in\ran(I-P)$, $\bp a \\ b \ep$ should be orthogonal to the vector $\bp 2 \\ 1 \ep$, which means that $\partial_{\eps}\partial_{\b_1}\mathcal V|_{\eps=0}=h_1\bp 1 \\ -2 \ep \cos \xi$. \\
Next, note that
\be
\begin{split}\label{h1}
&(I-Q)(B(0,\omega,0)-\lambda)\bp 1 \\ -2 \ep h_1 \cos \xi=h_1(I-Q)\big(\frac{1}{4}D\partial_\xi^2 + \bp 3 & 4 \\ -4 & -4 \ep-\lambda\big)\bp 1 \\ -2 \ep \cos \xi\\
&=(-6-\lambda)h_1\bp 1 \\ -2 \ep  \cos \xi
\end{split}
\ee

Using \eqref{ebetta}-\eqref{h1}, we derive that
$
(-6-\lambda)h_1=8\omega.
$
Hence,
$
\partial_{\eps}\partial_{\b_1}\mathcal V|_{\eps=0}=\frac{8\omega}{-6-\lambda}\bp 1 \\ -2 \ep  \cos \xi.
$
Similarly,
$
\partial_{\eps}\partial_{\b_2}\mathcal V|_{\eps=0}=\frac{8\omega}{-6-\lambda}\bp 1 \\ -2 \ep  \sin \xi.
$
So far, we have shown that 
$$
\mathcal V(\eps,\omega,\lambda,\b)=\mathcal (\frac{8\omega}{-6-\lambda}\cos \xi\bp 1 \\ -2 \ep  \eps+\mathcal O(\eps^2))\b_1+\mathcal (\frac{8\omega}{-6-\lambda}\sin \xi\bp 1 \\ -2 \ep  \eps+\mathcal O(\eps^2))\b_2.
$$
3.) Now, we would like to compute $\partial^2_{\eps}\partial_{\b_i}\mathcal V|_{\eps=0}$. Differentiating the second equation of \eqref{lambda} with respect to $\b_1$ and $\eps$ twice and, then, plugging in $0$ for $\eps$, we obtain
$$
	0
	 =(I-Q)\partial^2_{\eps}B(0,\omega,0)U_{1} + 2(I-Q)\partial_{\eps}B(0,\omega,0)\partial_{\eps}\partial_{\b_1}\mathcal V|_{\eps=0}+(I-Q)(B(0,\omega,0)-\lambda)\partial^2_{\eps}\partial_{\b_1}\mathcal V|_{\eps=0},
$$
or
$$
	(I-Q)(B(0,\omega,0)-\lambda)\partial^2_{\eps}\partial_{\b_1}\mathcal V|_{\eps=0}
	 =-(I-Q)\partial^2_{\eps}B(0,\omega,0)U_{1} - 2(I-Q)\partial_{\eps}B(0,\omega,0)\partial_{\eps}\partial_{\b_1}\mathcal V|_{\eps=0}.
$$
Therefore,
\be
\begin{split}
	&(I-Q)(B(0,\omega,0)-\lambda)\partial^2_{\eps}\partial_{\b_1}\mathcal V|_{\eps=0}
	 =-(I-Q)\partial^2_{\eps}B(0,\omega,0)U_{1} - 2(I-Q)\partial_{\eps}B(0,\omega,0)\partial_{\eps}\partial_{\b_1}\mathcal V|_{\eps=0}\\
	&=-2(I-Q)\big\{(2\omega^2D\partial_\xi^2+\bp 1& 0 \\ -1 & 0 \ep-\frac{4}{3}\sqrt{\frac{2}{3}(1-16\omega^2)}\bp -4& 4 \\ 4 & -4 \ep\cos \xi\\
	&-\frac{20}{9}\omega\sqrt{\frac{2}{3}(1-16\omega^2)}\cos \xi\bp 1 & 2 \\ -1 & -2 \ep+\frac{2}{3}(1-16\omega^2)\bp -4& 4 \\ 4 & -4 \ep\cos^2 \xi)U_1\\
	&+(\omega D\partial^2_\xi + 4\sqrt{\frac{2}{3}(1-16\omega^2)}\bp 1 & 2 \\ -1 & -2 \ep\cos \xi)\frac{8\omega}{-6-\lambda}\bp 1 \\ -2 \ep  \cos \xi\big\}\\
	&=-2(I-Q)\big\{(-16\omega^2\bp 1 \\ -2 \ep  \cos \xi+2\bp 1 \\ -1 \ep  \cos \xi+16\sqrt{\frac{2}{3}(1-16\omega^2)}\bp 1 \\ -1 \ep\frac{1}{2}(1+\cos 2\xi)\\
	&+\frac{-24}{3}(1-16\omega^2)\bp 1 \\ -1 \ep\frac{1}{4}(\cos 3\xi+3\cos \xi)\\
		&+(-4\omega \frac{8\omega}{-6-\lambda}\bp 1 \\ -8 \ep  \cos \xi + 4\sqrt{\frac{2}{3}(1-16\omega^2)}\frac{-24\omega}{-6-\lambda}\bp 1 \\ -1 \ep  \frac{1}{2}(1+\cos 2\xi)\big\}.
\end{split}
\ee		
 Using \eqref{1,-8} and \eqref{qcos}, we arrive at		
	\be
	\begin{split}	
	(I-Q)(B(0,\omega,0)-\lambda)\partial^2_{\eps}\partial_{\b_1}\mathcal V|_{\eps=0}	&=(32\omega^2-\frac{4}{3}+4(1-16\omega^2)+ \frac{320\omega^2}{-6-\lambda})\bp 1 \\ -2 \ep  \cos \xi\\
	&+\bp * \\ * \ep +\bp * \\ * \ep \cos 2\xi+\bp * \\ * \ep \cos 3\xi.
\end{split}
\ee

Since $\mathbb{R}^2$, $\mathbb{R}^2\cos \xi$, $\mathbb{R}^2\cos 2\xi$ and $\mathbb{R}^2\cos 3\xi$ are invariant subspaces for the invertible operator $(I-Q)(B(0,\omega,0)-\lambda)(I-P)$, $\partial^2_{\eps}\partial_{\b_1}\mathcal V|_{\eps=0}$ is of the form
	\be
	\begin{split}	
	\partial^2_{\eps}\partial_{\b_1}\mathcal V|_{\eps=0}	&=h_2\bp 1 \\ -2 \ep  \cos \xi+\bp * \\ * \ep +\bp * \\ * \ep \cos 2\xi+\bp * \\ * \ep \cos 3\xi.
\end{split}
\ee
Similarly,
	\be
	\begin{split}	
	\partial^2_{\eps}\partial_{\b_2}\mathcal V|_{\eps=0}	&=\tilde h_2\bp 1 \\ -2 \ep  \sin \xi+\bp * \\ * \ep +\bp * \\ * \ep \sin 2\xi+\bp * \\ * \ep \sin 3\xi.
\end{split}
\ee
	\be\label{Ve2}
	\begin{split}
\mathcal V(\eps,\omega,\lambda,\b)&=\mathcal (\frac{8\omega}{-6-\lambda}\cos \xi\bp 1 \\ -2 \ep  \eps+\frac{1}{2}\partial^2_{\eps}\partial_{\b_1}\mathcal V|_{\eps=0}\eps^2+\mathcal O(\eps^3))\b_1\\
&+\mathcal (\frac{8\omega}{-6-\lambda}\sin \xi\bp 1 \\ -2 \ep  \eps+\frac{1}{2}\partial^2_{\eps}\partial_{\b_2}\mathcal V|_{\eps=0}\eps^2+\mathcal O(\eps^3))\b_2.
\end{split}
\ee

Note that 
\be\label{Ve^2}
\begin{split}	
	&\mathbb{R}^2, \mathbb{R}^2\cos 2\xi, \mathbb{R}^2\cos 3\xi\in\ker \tilde Q, \\
	&\tilde Q \big(\frac{1}{4}D\partial_\xi^2 + \bp 3 & 4 \\ -4 & -4 \ep\big)\bp 1 \\ -2 \ep \cos \xi=-6\tilde Q\bp 1 \\ -2 \ep \cos \xi=0.
\end{split}
\ee

In order to obtain the reduced equation for the spectral problem, we substitute $W = \b_1 U_{1}+\b_{2} U_{2} + \mathcal V$, where $\mathcal V$ is given by \eqref{Ve2} into the equation 

\be\label{sred}
\begin{split}
	0
	& =\tilde Q\Big[ B(\eps,\omega,0)-\l I \Big] W.
\end{split}
\ee
Using \eqref{QU1}, \eqref{1,-8}, \eqref{qcos}, \eqref{fs}, \eqref{df}, \eqref{kk} and \eqref{Ve^2}, we arrive at
\be
\begin{split}
	0
	& =\tilde Q\Big[ B(\eps,\omega,0)-\l I \Big](\b_1 U_{1}+\b_{2} U_{2})+\tilde Q\Big[ B(\eps,\omega,0)-\l I \Big]\mathcal V\\
	&=\bp \frac{2}{3}\eps^2-2(1-16\omega^2)\eps^2-\lambda & 0 \\ 0 & \frac{2}{3}\eps^2-\frac{2}{3}(1-16\omega^2)\eps^2-\lambda \ep\bp \b_1 \\ \b_2 \ep\\
	&+\bp 8\omega\eps\frac{8\omega\eps}{-6-\lambda} & 0 \\ 0 & 8\omega\eps\frac{8\omega\eps}{-6-\lambda} \ep\bp \b_1 \\ \b_2 \ep+
	\bp \mathcal O(\eps^3) & \mathcal O(\eps^3) \\ \mathcal O(\eps^3) & \mathcal O(\eps^3) \ep\bp \b_1 \\ \b_2 \ep,
\end{split}
\ee
or
\be\label{spmatrix}
\begin{split}
	0
	&=\bp -\frac{4}{3}(1-16\omega^2)\eps^2-\lambda & 0 \\ 0 & -\lambda \ep\bp \b_1 \\ \b_2 \ep\\
	&+
	\bp \mathcal O(\eps^2(\eps+|\lambda|)) & \mathcal O(\eps^3) \\ \mathcal O(\eps^3) & \mathcal O(\eps^2(\eps+|\lambda|)) \ep\bp \b_1 \\ \b_2 \ep.
\end{split}
\ee
Now, we will establish the following refined remainder estimate.
\begin{lemma}\label{remainder}
The remainder in \eqref{spmatrix} has the form
\be
\begin{split}
	\bp \mathcal O(\eps^2(\eps+|\lambda|)) & \mathcal O(\eps^3) \\ \mathcal O(\eps^3) & \mathcal O(\eps^2(\eps+|\lambda|)) \ep=\bp \mathcal O(\eps^2(\eps+|\lambda|)) & \mathcal O(\eps^3|\lambda|) \\  \mathcal O(\eps^3|\lambda|)& \mathcal O(\eps^2|\lambda|) \ep.
\end{split}
\ee
\end{lemma}
\begin{proof}
	All we need to show is that if $\lambda=0$, then the reduced spectral equation is of the form
	\be
	\begin{split}
		0=\bp -\frac{4}{3}(1-16\omega^2)\eps^2+\mathcal O(\eps^3) & 0 \\ 0 & 0 \ep\bp \b_1 \\ \b_2 \ep.
	\end{split}
	\ee
	Now, we plug $0$ for $\lambda$ in \eqref{lambda} and then differentiate it with respect to $\b_1$.
	
	\be
	\begin{split}
		0
		& =(I-Q) B(\eps,\omega,0)( U_{1} + \partial_{\b_1}\mathcal V).
	\end{split}
	\ee
	Let us also  differentiate the second equation of \eqref{projection equations} with respect to $\alpha_1$ and then plug in $0$ for $\a_2$.
	\be
	\begin{split}
		& (I-Q)\big[(k^2D\partial_\xi^2+ A)\\&+  \bp (b+4) u_1+ 4 u_2+2 u_1  u_2  & 4 u_1+ u_1^2 \\ -(b+4) u_1- 4 u_2-2 u_1  u_2  & -4 u_1- u_1^2 \ep\big] (U_1+\partial_{\alpha_1}V)=0,
	\end{split}
	\ee
	Due to the uniqueness part in the Implicit Function Theorem, we conclude that 
	 \be
	 	\begin{split}
	 		\partial_{\b_1}\mathcal V(\eps,\omega,0,\b)=\partial_{\alpha_1}V|_{\alpha_2=0}&=f_1\cos \xi\bp 1 \\ -2 \ep +(-\frac{1}{4}f_2\bp 0 \\ 1  \ep+\bp \tilde a \\ \tilde b \ep\cos 2\xi)\alpha_1\\
	 		 &+\frac{1}{2}(\frac{\frac{1}{3}f_3}{20k^2 - \frac{b-3}{3}}\bp 1 \\ -2 \ep \cos \xi+\bp * \\ * \ep\cos 3\xi)\alpha_1^2+\mathcal{O}(|\a_1|^3).
	 	\end{split}
	 	\ee
	 	Similarly, we conclude that 
	 	\be
	 		 	\begin{split}
	 		 		\partial_{\b_2}\mathcal V(\eps,\omega,0,\b)=\partial_{\alpha_2}V|_{\alpha_2=0}=f_1\sin \xi\bp 1 \\ -2 \ep.
	 		 	\end{split}
	 		 	\ee
	 		 	Therefore, in order to find the entries of the spectral matrix from the reduced equation \eqref{sred} we differentiate the first and second equations of \eqref{reduced} with respect to $\a_1$ and $\a_2$ and then plug $0$ for $\a_2$.
	 		 	  Hence,
	 		\be
	 		\begin{split}
	 		\tilde Q B(\eps,\omega,0) W=\bp \mathcal{A}-3\a^2  +\mathcal{O}(|\a|^4) & 0 \\ 0 & \mathcal{A}-\a^2  +\mathcal{O}(|\a|^4) \ep\bp \b_1 \\ \b_2 \ep=0\\
	 		\end{split},
	 		\ee 	  
	 		where $\mathcal{A}-\a^2  +\mathcal{O}(|\a|^4)$ is exactly the left-hand side of \eqref{short}. Using formulas \eqref{Acal} and \eqref{alpha}, we arrive at
	 		\be
	 		\begin{split}
	 			\tilde Q B(\eps,\omega,0) W=\bp -\frac{4}{3}(1-16\omega^2)\eps^2+\mathcal O(\eps^3) & 0 \\ 0 & 0 \ep\bp \b_1 \\ \b_2 \ep=0\\
	 		\end{split}.
	 		\ee 	
\end{proof}

Using the refined remainder estimate, we obtain the following characterization of co-periodic stability.

\begin{proposition}[Co-periodic stability]\label{coperprop}
	Let $u_{\eps,\omega}$ be the solution from Theorem \ref{existence}.  There exist
	 $\tilde\eps_0\in(0,\eps_0]$, where $\eps_0$ is taken from Theorem \ref{existence}, and $\delta>0$ such that for all $\eps\in[0,\tilde\eps_0)$ and all $\omega\in[-\frac{1}{4},\frac{1}{4}]$ the spectrum of $B(\eps,\omega,0)$ has the decomposition:
	\be
	\begin{split}
		\Sp(B(\eps,\omega,0))=S\cup\{\lambda_1,\lambda_2\},
	\end{split}
	\ee
	where
	\be
	\begin{split}
		\lambda_1(\eps,\omega)&=-\frac{4}{3}(1-16\omega^2)\eps^2+\mathcal O(\eps^3),\\
		\lambda_2(\eps,\omega)&=0.
	\end{split}
	\ee
Moreover, if $\lambda\in S$, then $\Re\lambda<-\delta$.
\end{proposition}

\begin{proof}
	Setting the determinant of the matrix from \eqref{spmatrix} equal to $0$, we obtain
	\be
	 	\begin{split}
	 		 \big(c(\eps,\omega)-\lambda+ \mathcal O(\eps^2|\lambda|)\big)(-\lambda+\mathcal O(\eps^2|\lambda|))+\mathcal O(|\lambda|^2\eps^6)=0,
	 	\end{split}
	\ee
	where $c(\eps,\omega)=-\frac{4}{3}(1-16\omega^2)\eps^2+\mathcal O(\eps^3)$, or
	\be\label{pol}
	\begin{split}
	\lambda^2-\lambda\big(c(\eps,\omega)+\mathcal O(\eps^2|\lambda|)\big)+\mathcal O(\eps^4|\lambda|)+\mathcal O(|\lambda|^2\eps^6)=0.
	\end{split}
	\ee
	
	\be
	\begin{split}
		\lambda_{1,2}&=\frac{c(\eps,\omega)+\mathcal O(\eps^2|\lambda|)\pm\sqrt{(c(\eps,\omega)+\mathcal O(\eps^2|\lambda|))^2+\mathcal O(\eps^4|\lambda|)}}{2}\\
		&=\eps^2\frac{\tilde c(\eps,\omega)+\mathcal O(|\lambda|)\pm\sqrt{(\tilde c(\eps,\omega)+\mathcal O(|\lambda|))^2+\mathcal O(|\lambda|)}}{2},
	\end{split}
	\ee
	where $\tilde c(\eps,\omega)=-\frac{4}{3}(1-16\omega^2)+\mathcal O(\eps)$. 
	Assume that $\omega\neq\pm\frac{1}{4}$, then we can expand \\$\sqrt{(\tilde c(\eps,\omega)+\mathcal O(|\lambda|))^2+\mathcal O(|\lambda|)}$ as $-\tilde c(\eps,\omega)+\mathcal O(|\lambda|)$. Therefore,
		\be
		\begin{split}
			\lambda_{1}&=c(\eps,\omega)+\mathcal O(\eps^2|\lambda|),\\
			\lambda_{2}&=\mathcal O(\eps^2|\lambda|).\\
		\end{split}
		\ee
		Now assume that $\omega=\pm\frac{1}{4}$. Then $c(\eps,\omega)=0$. Instead of \eqref{pol}, we have 
		\be\label{pol1/4}
		\begin{split}
			\lambda^2-\lambda \mathcal O(\eps^2|\lambda|)+\mathcal O(\eps^4|\lambda|^2)=0.
		\end{split}
		\ee
		Then, both $\lambda_{1}$ and $\lambda_{2}$ are of the form $\mathcal O(\eps^2|\lambda|)$.
\end{proof}

\subsection{Stability of the bifurcating periodic solutions: general case}

We now consider the eigenvalue problem of $B(\eps,\omega,\sigma)$:
\be\label{evp}
0=\Big[ B(\eps,\omega,\sigma)-\l I \Big] W.
\ee
In order to use the Lyapunov-Schmidt reduction, we decompose $W = \b_1 U_{1}+\b_{2} U_{2} + \mathcal V$ and we first solve
\be\label{slambda}
\begin{split}
	0
	& =(I-Q)\Big[ B(\eps,\omega,\sigma)-\l I \Big] (\b_1 U_{1}+\b_{2} U_{2} + \mathcal V).
\end{split}
\ee

Next, we go through steps described in the previous section. 
Next, it is clear that the relation between $\b$ and $\mathcal V$ is linear. Then, let $\mathcal V(\eps,\omega,\sigma,\lambda,\b)=\mathcal V_1(\eps,\omega,\sigma,\lambda)\b_1+\mathcal V_2(\eps,\omega,\sigma,\lambda)\b_2$. Now let us find asymptotic expansions of $\mathcal V_1$ and $\mathcal V_2$ with respect to parameter $\eps$.\\
1.) First, we compute $\mathcal V_i(0,\omega,\sigma,\lambda)=\partial_{\b_i}\mathcal V|_{\eps=0}$. \\
a.) Differentiating \eqref{slambda} with respect $\b_1$ and plugging in $0$ for $\eps$, we obtain
$$
	0  =(I-Q)\Big[ B(0,\omega,\sigma)-\l I \Big] (U_{1} + \partial_{\b_1}\mathcal V|_{\eps=0}).
$$
Therefore,
\be\label{eqV}
\begin{split}
	&(I-Q)\Big[ B(0,\omega,\sigma)-\l I \Big]\partial_{\b_1}\mathcal V|_{\eps=0}
	=-(I-Q)\Big[ B(0,\omega,\sigma)-\l I \Big] U_{1}\\
	&=-(I-Q)(L_{per}U_{1}-2i\sigma\bp 2 \\ -4 \ep\sin \xi-\sigma^2\bp 2 \\ -4 \ep\cos \xi)-\lambda(I-Q)U_1\\
	&=4i\sigma\bp 1 \\ -2 \ep\sin \xi+2\sigma^2\bp 1 \\ -2 \ep\cos \xi.
\end{split}
\ee
We conclude that $\mathcal V_1(0,\omega,\sigma,\lambda)$ is of the form
\be\label{Vs}
\mathcal V_1(0,\omega,\sigma,\lambda)=\partial_{\b_1}\mathcal V|_{\eps=0}=h_1\bp 1 \\ -2 \ep\sin \xi+h_2\bp 1 \\ -2 \ep\cos \xi.
\ee
Let us compute $h_i$.  Plugging \eqref{Vs} into the left-hand side of \eqref{eqV}, we obtain
\be
\begin{split}\label{leqV}
	&(I-Q)\Big[ B(0,\omega,\sigma)-\l I \Big]\partial_{\b_i}\mathcal V|_{\eps=0}
	=(I-Q)(L_{per}(h_1\bp 1 \\ -2 \ep\sin \xi+h_2\bp 1 \\ -2 \ep\cos \xi)\\
	&+2i\sigma h_1\bp 1 \\ -8 \ep\cos \xi-2i\sigma h_2\bp 1 \\ -8 \ep\sin \xi-\sigma^2h_1\bp 1 \\ -8 \ep\sin \xi-\sigma^2h_2\bp 1 \\ -8 \ep\cos \xi)\\
	&=(-6-\lambda)h_1\bp 1 \\ -2 \ep\sin \xi+(-6-\lambda)h_2\bp 1 \\ -2 \ep\cos \xi+10i\sigma h_1\bp 1 \\ -2 \ep\cos \xi\\
	&-10i\sigma h_2\bp 1 \\ -2 \ep\sin \xi-5\sigma^2h_1\bp 1 \\ -2 \ep\sin \xi-5\sigma^2h_2\bp 1 \\ -2 \ep\cos \xi.
\end{split}
\ee
Taking into account \eqref{eqV} and \eqref{leqV}, we have a system of linear equations in $h_1$ and $h_2$:
 \be\label{sys}
 \begin{split}
   &-4i\sigma+(-6-\lambda-5\sigma^2)h_1-10i\sigma h_2=0,\\
 	&-2\sigma^2+(-6-\lambda-5\sigma^2)h_2+10i\sigma h_1=0.
 \end{split}
 \ee
Therefore,
 \be\label{h1,2}
 \begin{split}
   &h_1=\frac{4i\sigma(-6-\lambda)}{(-6-\lambda-5\sigma^2)^2-100\sigma^2},\qquad
 	h_2=\frac{2\sigma^2(14-\lambda-5\sigma^2)}{(-6-\lambda-5\sigma^2)^2-100\sigma^2}.
 \end{split}
 \ee
b.) In a similar fashion, we compute $\mathcal V_2(0,\omega,\sigma,\lambda)$. Since
$$
\begin{aligned}
	(I-Q)\Big[ B(0,\omega,\sigma)-\l I \Big]\partial_{\b_2}\mathcal V|_{\eps=0}
	&=-(I-Q)\Big[ B(0,\omega,\sigma)-\l I \Big] U_{2}\\
	& =-4i\sigma\bp 1 \\ -2 \ep\cos \xi+2\sigma^2\bp 1 \\ -2 \ep\sin \xi,
\end{aligned}
$$
we conclude that $\mathcal V_2(0,\omega,\sigma,\lambda)$ is of the form
\be
\mathcal V_2(0,\omega,\sigma,\lambda)=\partial_{\b_2}\mathcal V|_{\eps=0}=\tilde h_1\bp 1 \\ -2 \ep\sin \xi+\tilde h_2\bp 1 \\ -2 \ep\cos \xi.
\ee

Then, similarly as in \eqref{sys}, we have a system of linear equations in $\tilde h_1$ and $\tilde h_2$:
$$
\begin{aligned}
    4i\sigma+(-6-\lambda-5\sigma^2)\tilde h_2+10i\sigma \tilde h_1&=0,\\
 	-2\sigma^2+(-6-\lambda-5\sigma^2)\tilde h_1-10i\sigma \tilde h_2&=0.
\end{aligned}
$$
Therefore,
 \be\label{h1h2}
 \begin{split}
   &\tilde h_1=h_2, \qquad
 	\tilde h_2=-h_1.
 \end{split}
 \ee

Hence,
$$
\begin{aligned}
\mathcal V(\eps,\omega,\sigma,\lambda,\b) &= \Big( h_1\bp 1 \\ -2 \ep  \sin \xi+h_2\bp 1 \\ -2 \ep  \cos \xi+\mathcal O(\eps)\Big)\b_1
\\
&\quad +\Big(h_2\bp 1 \\ -2 \ep  \sin \xi-h_1\bp 1 \\ -2 \ep  \cos \xi +\mathcal O(\eps) \Big)\b_2.
\end{aligned}
$$

2.) Next, we would like to compute $\partial_{\eps}\partial_{\b_i}\mathcal V|_{\eps=0}$.\\
a.) We start with $\partial_{\eps}\partial_{\b_1}\mathcal V|_{\eps=0}$.
Differentiating \eqref{slambda} with respect $\b_1$ and $\eps$, and plugging in $0$ for $\eps$, we obtain
$
	0
	 =(I-Q)\big\{\partial_{\eps}B(0,\omega,\sigma) (U_{1} + \partial_{\b_1}\mathcal V|_{\eps=0})+[ B(0,\omega,\sigma)-\l I]\partial_{\eps}\partial_{\b_1}\mathcal V|_{\eps=0}\big\}.
$

Therefore,
$
	(I-Q)[ B(0,\omega,\sigma)-\l I]\partial_{\eps}\partial_{\b_1}\mathcal V|_{\eps=0}
	 =-(I-Q)\partial_{\eps}B(0,\omega,\sigma) (U_{1} + \partial_{\b_1}\mathcal V|_{\eps=0}),
$
or, using the second line in \eqref{eqV} and the last line in \eqref{leqV},
\be
\begin{split}\label{Vsel}
	&(I-Q)[ B(0,\omega,\sigma)-\l I]\partial_{\eps}\partial_{\b_1}\mathcal V|_{\eps=0}
	=-(I-Q)[\omega D(\partial^2_{\xi}+2i\sigma-\sigma^2)\\
	&+4\sqrt{\frac{2}{3}(1-16\omega^2)}\cos \xi\bp 1 & 2 \\ -1 & -2 \ep] (U_{1} + \partial_{\b_1}\mathcal V|_{\eps=0})=(16i\omega\sigma\bp 1 \\ -2 \ep\sin \xi+8\omega(1+\sigma^2)\bp 1 \\ -2 \ep\cos \xi)\\
	&-(40i\omega\sigma h_1\bp 1 \\ -2 \ep\cos \xi-40i\omega\sigma h_2\bp 1 \\ -2 \ep\sin \xi-20\omega(1+\sigma^2)h_1\bp 1 \\ -2 \ep\sin \xi\\
	&-20\omega(1+\sigma^2)h_2\bp 1 \\ -2 \ep\cos \xi-12\sqrt{\frac{2}{3}(1-16\omega^2)}h_2\bp 1 \\ -1 \ep\cos^2 \xi\\
	&-12\sqrt{\frac{2}{3}(1-16\omega^2)}h_1\bp 1 \\ -1 \ep\cos \xi\sin \xi).
\end{split}
\ee
Hence, $\partial_{\eps}\partial_{\b_1}\mathcal V|_{\eps=0}$ is of the form
\be
 \begin{split}
&\partial_{\eps}\partial_{\b_1}\mathcal V|_{\eps=0}= g_1\bp 1 \\ -2 \ep  \sin \xi+g_2\bp 1 \\ -2 \ep  \cos \xi+G_3 +G_4\sin 2\xi+G_5\cos 2\xi,
\end{split}
 \ee
where $G_i\in\mathbb{R}^2$.

Let us compute $g_i$.  Note that $span \{\bp * \\ * \ep  \sin \xi,\,\bp * \\ * \ep  \cos \xi\}$ is an invariant subspace for the invertible operator $(I-Q)[ B(0,\omega,\sigma)-\l I](I-P)$. Plugging $g_1\bp 1 \\ -2 \ep  \sin \xi+g_2\bp 1 \\ -2 \ep  \cos \xi$ into the left-hand side of \eqref{Vsel} (c.f. \eqref{leqV}), we obtain
\be\label{eqg1g2}
\begin{split}
	&(I-Q)[ B(0,\omega,\sigma)-\l I](g_1\bp 1 \\ -2 \ep  \sin \xi+g_2\bp 1 \\ -2 \ep  \cos \xi)\\
		&=(-6-\lambda)g_1\bp 1 \\ -2 \ep\sin \xi+(-6-\lambda)g_2\bp 1 \\ -2 \ep\cos \xi\\
			&+10i\sigma g_1\bp 1 \\ -2 \ep\cos \xi-10i\sigma g_2\bp 1 \\ -2 \ep\sin \xi-5\sigma^2g_1\bp 1 \\ -2 \ep\sin \xi-5\sigma^2g_2\bp 1 \\ -2 \ep\cos \xi.
\end{split}
\ee
Taking into account \eqref{Vsel} and \eqref{eqg1g2}, we have a system of linear equations in $g_1$ and $g_2$:
 \be\label{g1,2}
 \begin{split}
   &(-6-\lambda)g_1-5\sigma^2 g_1-10i\sigma g_2=16i\omega\sigma+40i\omega\sigma h_2+20\omega(1+\sigma^2)h_1,\\
 	&(-6-\lambda)g_2-5\sigma^2 g_2+10i\sigma g_1=8\omega(1+\sigma^2)-40i\omega\sigma h_1+20\omega(1+\sigma^2)h_2.
 \end{split}
 \ee
b.) In a similar fashion, we compute $\partial_{\eps}\partial_{\b_2}\mathcal V|_{\eps=0}$. Taking into account \eqref{h1h2}, we arrive at
$$
\begin{aligned}
	&(I-Q)[ B(0,\omega,\sigma)-\l I]\partial_{\eps}\partial_{\b_2}\mathcal V|_{\eps=0}
		=-(I-Q)[\omega D(\partial^2_{\xi}+2i\sigma-\sigma^2)\\
		&+4\sqrt{\frac{2}{3}(1-16\omega^2)}\cos \xi\bp 1 & 2 \\ -1 & -2 \ep] (U_{2} + \partial_{\b_2}\mathcal V|_{\eps=0})=(-16i\omega\sigma\bp 1 \\ -2 \ep\cos \xi\\
		&+8\omega(1+\sigma^2)\bp 1 \\ -2 \ep\sin \xi)-(40i\omega\sigma h_2\bp 1 \\ -2 \ep\cos \xi+40i\omega\sigma h_1\bp 1 \\ -2 \ep\sin \xi\\
		&-20\omega(1+\sigma^2)h_2\bp 1 \\ -2 \ep\sin \xi+20\omega(1+\sigma^2)h_1\bp 1 \\ -2 \ep\cos \xi\\
		&+12\sqrt{\frac{2}{3}(1-16\omega^2)}h_1\bp 1 \\ -1 \ep\cos^2 \xi-12\sqrt{\frac{2}{3}(1-16\omega^2)}h_2\bp 1 \\ -1 \ep\cos \xi\sin \xi),
\end{aligned}
$$
Hence, $\partial_{\eps}\partial_{\b_2}\mathcal V|_{\eps=0}$ is of the form
$$
\partial_{\eps}\partial_{\b_2}\mathcal V|_{\eps=0}= \tilde g_1\bp 1 \\ -2 \ep  \sin \xi+\tilde g_2\bp 1 \\ -2 \ep  \cos \xi+ \tilde G_3 +\tilde G_4\sin 2\xi+\tilde G_5\cos 2\xi,
$$
 and we have the system of linear equations in $\tilde g_1$ and $\tilde g_2$:
$$
\begin{aligned}
    (-6-\lambda)\tilde g_1-5\sigma^2 \tilde g_1-10i\sigma \tilde g_2&=8\omega(1+\sigma^2)-40i\omega\sigma h_1+20\omega(1+\sigma^2)h_2,\\
  	(-6-\lambda)\tilde g_2-5\sigma^2 \tilde g_2+10i\sigma \tilde g_1&=-16i\omega\sigma-40i\omega\sigma h_2-20\omega(1+\sigma^2)h_1.
\end{aligned}
$$
Therefore, $\tilde g_1=g_2$ and $\tilde g_2=-g_1$. \\
Overall, we have
\be
 \begin{split}\label{Vse}
&\mathcal V(\eps,\omega,\sigma,\lambda,\b)= (h_1\bp 1 \\ -2 \ep  \sin \xi+h_2\bp 1 \\ -2 \ep  \cos \xi+\big(g_1\bp 1 \\ -2 \ep  \sin \xi+g_2\bp 1 \\ -2 \ep  \cos \xi+G_3\\
&G_4\sin 2\xi+G_5\cos 2\xi\big)\eps+\mathcal O(\eps^2))\b_1+(h_2\bp 1 \\ -2 \ep  \sin \xi-h_1\bp 1 \\ -2 \ep  \cos \xi\\
&+\big(g_2\bp 1 \\ -2 \ep  \sin \xi-g_1\bp 1 \\ -2 \ep  \cos \xi+\tilde G_3+\tilde G_4\sin 2\xi+\tilde G_5\cos 2\xi\big)\eps+\mathcal O(\eps^2))\b_2.
\end{split}
 \ee

In order to obtain the reduced equation for the spectral problem, we plug $W = \b_1 U_{1}+\b_{2} U_{2} + \mathcal V$, where $\mathcal V$ is given by \eqref{Vse},
 into the equation 
\be\label{redeq}
\begin{split}
	0
	& =\tilde Q\Big[ B(\eps,\omega,\sigma)-\l I \Big] W.
\end{split}
\ee

Using \eqref{QU1}, \eqref{1,-8}, \eqref{qcos}, \eqref{fs}, \eqref{df}, \eqref{kk}, \eqref{Ve^2} and Lemma \ref{remainder}, we arrive at
\be\label{spmatrixs}
\begin{split}
	0
	& =\tilde Q\Big[ B(\eps,\omega,\sigma)-\l I \Big](\b_1 U_{1}+\b_{2} U_{2})+\tilde Q\Big[ B(\eps,\omega,\sigma)-\l I \Big]\mathcal V\\
	&=\bp \frac{2}{3}\eps^2-2(1-16\omega^2)\eps^2-\lambda & 0 \\ 0 & \frac{2}{3}\eps^2-\frac{2}{3}(1-16\omega^2)\eps^2-\lambda \ep\bp \b_1 \\ \b_2 \ep\\
	&+\bp 2\sigma^2h_2-4i\sigma h_1 & -2\sigma^2h_1-4i\sigma h_2 \\ 2\sigma^2h_1+4i\sigma h_2 & 2\sigma^2h_2-4i\sigma h_1 \ep\bp \b_1 \\ \b_2 \ep\\
	&+\bp 8\omega(1+\sigma^2)h_2-16i\omega\sigma h_1+2\sigma^2g_2-4i\sigma g_1 & -8\omega(1+\sigma^2)h_1-16i\omega\sigma h_2-2\sigma^2g_1-4i\sigma g_2 \\ 8\omega(1+\sigma^2)h_1+16i\omega\sigma h_2+2\sigma^2g_1+4i\sigma g_2 & 8\omega(1+\sigma^2)h_2-16i\omega\sigma h_1+2\sigma^2g_2-4i\sigma g_1 \ep\bp \b_1 \\ \b_2 \ep\eps\\
	&+\bp * & * \\ * & * \ep\bp \b_1 \\ \b_2 \ep\eps^2+\mathcal O(\eps^3).
	\end{split}
	\ee
	
Some computations we need based on \eqref{h1,2} and \eqref{g1,2} are:	
	\be\label{est}
	\begin{split}
	&2\sigma^2h_2-4i\sigma h_1=\frac{-16}{\lambda+6}\sigma^2+\mathcal O(\sigma^4(1+|\lambda|))=-\frac{8}{3}\sigma^2+\mathcal O(\sigma^2(|\lambda|+\sigma^2)),\\
	&2\sigma^2h_1+4i\sigma h_2=\mathcal O(\sigma^3(1+|\lambda|)),\\
	&8\omega(1+\sigma^2)h_2-16i\omega\sigma h_1+2\sigma^2g_2-4i\sigma g_1=8\omega(\frac{7}{9}\sigma^2)-16i\omega\sigma(-\frac{2}{3}i\sigma)+2\sigma^2(-\frac{4}{3}\omega)-4i\sigma(\frac{16}{9}i\omega\sigma)\\
	&+\mathcal O(\sigma^2(\sigma^2+|\lambda|))=\mathcal O(\sigma^2(\sigma^2+|\lambda|)),\\
	&-8\omega(1+\sigma^2)h_1-16i\omega\sigma h_2-2\sigma^2g_1-4i\sigma g_2=-8\omega h_1-4i\sigma g_2+\mathcal O(\sigma^3(1+|\lambda|))\\
	&=-8\omega(-\frac{2}{3}i\sigma)-4i\sigma(-\frac{4}{3}\omega)+\mathcal O(\sigma(\sigma^2+|\lambda|))=\frac{32}{3}i\omega\sigma+\mathcal O(\sigma(\sigma^2+|\lambda|)).
		\end{split}
		\ee

Taking into account \eqref{spmatrix}, \eqref{spmatrixs} and \eqref{est}, we arrive at
\be\label{det}
\begin{split}
	&0 =m(\eps,\omega,\sigma,\lambda)\bp \b_1 \\ \b_2 \ep:=\bp c(\eps)-\frac{8}{3}\sigma^2-\lambda & \frac{32}{3}i\omega\sigma\eps \\ -\frac{32}{3}i\omega\sigma\eps & -\frac{8}{3}\sigma^2-\lambda \ep\bp \b_1 \\ \b_2 \ep\\
	&+
	\bp \mathcal O(\sigma^2(|\lambda|+\sigma^2)+\sigma^2(\sigma^2+|\lambda|)\eps+(|\lambda|+\sigma)\eps^2) & \mathcal O(\sigma^3(1+|\lambda|)+\sigma(\sigma^2+|\lambda|)\eps+\sigma\eps^2+|\lambda|\eps^3) \\  \mathcal O(\sigma^3(1+|\lambda|)+\sigma(\sigma^2+|\lambda|)\eps+\sigma\eps^2+|\lambda|\eps^3) & \mathcal O(\sigma^2(|\lambda|+\sigma^2)+\sigma^2(\sigma^2+|\lambda|)\eps+(|\lambda|+\sigma)\eps^2) \ep
	\\
	&\times \bp \b_1 \\ \b_2 \ep,
\end{split}
\ee
where $c(\eps)=-\frac{4}{3}(1-16\omega^2)\eps^2+\mathcal O(\eps^3)$.
One can improve the error estimates in \eqref{det} using symmetric properties of the eigenvalue problem \eqref{evp}. In particular, we gain additional information on elements of matrix $m$.
\begin{lemma}
The diagonal elements of matrix $m$ in \eqref{det} are even in $\sigma$ and off-diagonal elements are odd in $\sigma$. Moreover, if $\lambda$ is real then $m_{11}$ and $m_{22}$ are real-valued while $m_{12}$ and $m_{21}$ are purely imaginary. 
\end{lemma}
\begin{proof}
First, we note that \eqref{evp} possesses two symmetries \cite{M2}
\be\label{symms}
\begin{split}
&[ B(\eps,\omega,\sigma)-\l I \Big]R_1=R_1[ B(\eps,\omega,-\sigma)-\l I \Big],\,\lambda\in\mathbb{C}\\
&[ B(\eps,\omega,\sigma)-\l I \Big]R_2=R_2[ B(\eps,\omega,-\sigma)-\l I \Big],\,\lambda\in\mathbb{R}.
\end{split}
\ee
where $R_1W(\xi)=W(-\xi)$, and $R_2W(\xi)=\overline W(\xi)$. \\
Following the steps of proof of Proposition 3.3 \cite[Chapter VII]{GS}, one can show that the reduced equation \eqref{redeq} commutes with symmetries defined in \eqref{symms}, i.e.
\begin{equation}
m(\eps,\omega,\sigma,\lambda)\bp \b_1 \\ -\b_2 \ep=\bp m_{11}(\eps,\omega,-\sigma,\lambda) &m_{12}(\eps,\omega,-\sigma,\lambda)\\ -m_{21}(\eps,\omega,-\sigma,\lambda)&-m_{22}(\eps,\omega,-\sigma,\lambda) \ep\bp \b_1 \\ \b_2 \ep,\,\lambda\in\mathbb{C}.
\end{equation}
And
\begin{equation}
m(\eps,\omega,\sigma,\lambda)\bp \b_1 \\ \b_2 \ep=\overline m(\eps,\omega,-\sigma,\lambda)\bp \b_1 \\ \b_2 \ep,\,\lambda\in\mathbb{R}.
\end{equation}
\end{proof}
\begin{corollary}\label{errcorr}
The error matrix in \eqref{det} has the form
\begin{equation}
\bp \mathcal O(\sigma^2(|\lambda|+\sigma^2)+\sigma^2(\sigma^2+|\lambda|)\eps+(|\lambda|+\sigma^2)\eps^2) & \mathcal O(\sigma^3(1+|\lambda|)+\sigma(\sigma^2+|\lambda|)\eps+\sigma\eps^2) \\  \mathcal O(\sigma^3(1+|\lambda|)+\sigma(\sigma^2+|\lambda|)\eps+\sigma\eps^2) & \mathcal O(\sigma^2(|\lambda|+\sigma^2)+\sigma^2(\sigma^2+|\lambda|)\eps+(|\lambda|+\sigma^2)\eps^2) \ep,
\end{equation}
i.e. for the diagonal entries we conclude that $O(\sigma\eps^2)=O(\sigma^2\eps^2)$ and for the off-diagonal entries we conclude that $O(|\lambda|\eps^3)=0$.
\end{corollary}

\begin{proof}[Proof of Theorem \ref{stabthm}]
	Let us take the determinant of \eqref{det}. 
	\be
	\begin{split}
	\det	m(\eps,\omega,\sigma,\lambda)&:=\lambda^2-\lambda\big(c(\eps)-\frac{16}{3}\sigma^2\\
	&+O(\sigma^2(|\lambda|+\sigma^2)+\sigma^2(\sigma^2+|\lambda|)\eps+(|\lambda|+\sigma^2)\eps^2)\big)\\
		&+(-c(\eps)+\frac{8}{3}\sigma^2)\frac{8}{3}\sigma^2-(\frac{32}{3}\omega\sigma\eps)^2\\
		&+O\big((\eps^2+\sigma^2)(\sigma^2(|\lambda|+\sigma^2)+\sigma^2(\sigma^2+|\lambda|)\eps+(|\lambda|+\sigma^2)\eps^2)\big)\\
		&+O(\sigma\eps)O(\sigma^3(1+|\lambda|)+\sigma(\sigma^2+|\lambda|)\eps+\sigma\eps^2)\\
		&+\big(O(\sigma^3(1+|\lambda|)+\sigma(\sigma^2+|\lambda|)\eps+\sigma\eps^2)\big)^2.
	\end{split}
	\ee
	According to the Weierstrass Preparation Theorem, there exists an analytic function $q(\eps,\sigma,\lambda)$ in a neighborhood of $(0,0,0)$ such that $q(0,0,0)=1$ and
	 	\be\label{wpt}
	 	\begin{split}
	 		q(\eps,\sigma,\lambda)\det m(\eps,\omega,\sigma,\lambda)=\lambda^2+a_1\lambda+a_0.
	 	\end{split}
	 	\ee
	 	Notice that
	 	\be
	 	\begin{split}
	 		&a_0(\eps,\sigma)=q(\eps,\sigma,0)\det m(\eps,\omega,\sigma,0)=\det m(\eps,\omega,\sigma,0)+\mathcal O\big(\sigma^2(\eps+\sigma)^3\big)\\
	 		&=(-c(\eps)+\frac{8}{3}\sigma^2)\frac{8}{3}\sigma^2-(\frac{32}{3}\omega\sigma\eps)^2+\mathcal O\big(\sigma^2(\eps+\sigma)^3\big),\\
	 		&a_1(\eps,\sigma)=q'_{\lambda}(\eps,\sigma,0)\det m(\eps,\omega,\sigma,0)
	 		+q(\eps,\sigma,0)(\det m)'_{\lambda}(\eps,\omega,\sigma,0)\\
	 		&=(\det m)'_{\lambda}(\eps,\omega,\sigma,0)+\mathcal O((\eps+\sigma)^3)=-c(\eps)+\frac{16}{3}\sigma^2+\mathcal O((\eps+\sigma)^3).
	 	\end{split}
	 	\ee
	 	
	 	Therefore, the eigenvalue problems boils down to the second order polynomial
	 	\be
	 	\begin{split}
	 		\lambda^2-\big(c_1(\eps)-\frac{16}{3}\sigma^2+\mathcal O(\sigma^2(\eps+\sigma))\big)\lambda+\frac{32}{9}\sigma^2\big(\eps^2(1-48\omega^2)+2\sigma^2+O((\eps+\sigma)^3)\big)=0,
	 	\end{split}
	 	\ee
	 	where $c_1(\eps,\sigma)=-\frac{4}{3}(1-16\omega^2)\eps^2+\mathcal O(\eps^2(\eps+\sigma))$.
	 	Therefore, the roots are of the form
		\be \label{quad1}
	 		 	\begin{split}
	 		 		\lambda_{1,2}=\frac{c_1(\eps,\sigma)-\frac{16}{3}\sigma^2+\mathcal O(\sigma^2(\eps+\sigma))\pm\sqrt{c_1^2(\eps,\sigma)+\frac{64^2}{9}(\omega\eps\sigma)^2+O(\sigma^2(\eps+\sigma)^3)}}{2}
	 		 	\end{split}
	 		 	\ee
	 		 	
	 	Next, let $\sigma=\eps\hat\sigma$. Then
		\be\label{sqrtexp}
	 		 		 	\begin{split}
	 		 		 		\lambda_{1,2}=\eps^2\frac{\tilde c(\eps,\sigma)-\frac{16}{3}\hat\sigma^2+\mathcal O(\hat\sigma^2(\eps+\sigma))\pm\sqrt{\tilde c^2(\eps,\sigma)+\frac{64^2}{9}(\omega\hat\sigma)^2+O(\hat\sigma^2(1+\hat\sigma)^2(\eps+\sigma))}}{2},
	 		 		 	\end{split}
	 		 		 	\ee
	 		 		 	where $\tilde c(\eps,\sigma)=-\frac{4}{3}(1-16\omega^2)+\mathcal O(\eps+\sigma)$.\\
	 		 		 	Next, we fix $\omega$ such that $\omega^2<\frac{1}{48}$. Then we consider three different cases: 1) $|\hat \sigma|<<1$, 2) $1/C\leq|\hat \sigma|\leq C$, 3) $|\hat \sigma|>>1$.\\
	 		 		 	1) $|\hat \sigma|<<1$. We expand $\lambda_{1,2}$ w.r.t. $\hat \sigma$.\\
	 		 		 	\be\label{eBM}
	 		 		 	\begin{split}
	 		 		 	&\lambda_{1}(\eps,\sigma)=-\frac{4}{3}(1-16\omega^2)\eps^2+\mathcal O(\eps^3)+( c'_{1\sigma}(\eps,0))\sigma-(\frac{8(1+16\omega^2)}{3(1-16\omega^2)}+\mathcal O(\eps))\sigma^2+\mathcal O(\sigma^3),\\
	 		 		 	&\lambda_{2}(\eps,
	 		 		 	\sigma)=-(\frac{8(1-48\omega^2)}{3(1-16\omega^2)}+\mathcal O(\eps))\sigma^2+\mathcal O(\sigma^3).
	 		 		 		\end{split}
	 		 		 			\ee
	 		 		 	2) $1/C\leq|\hat \sigma|\leq C$.\\
						For the unperturbed Ginzburg Landau case, we know (see Section 
						\ref{s:GLcompare} for details) that $\Re\lambda_j\leq\eta<0$. Since $\hat\sigma$ belongs to the compact interval, we deduce that $\Re\lambda_j\leq\hat\eta<0$.\\
	 		 		 	3) $|\hat \sigma|>>1$. The roots $\lambda_1$ and $\lambda_2$ are controlled by $-\frac{16}{3}\sigma^2$. 
	 		 		 	\be
	 		 		 		\begin{split}
	 		 		 			&	\Re\lambda_{1}\leq c(\eps,\omega)+\tilde c(\eps,\omega)\sigma-\tilde{\tilde c}_1(\eps,\omega)\sigma^2+\mathcal O(\sigma^3),\\
	 		 		 			&	\Re\lambda_{2}\leq -\tilde{\tilde c}_2(\eps,\omega)\sigma^2+\mathcal O(\sigma^3),
	 		 		 		\end{split}
	 		 		 		\ee 
	 		 		 		For a fixed $\omega$ such that $\frac{1}{48}<\omega^2\leq\frac{1}{16}$ similar to the case 2 we deduce that  $\max_\sigma\{\Re\lambda_1,\Re\lambda_2\}>0$.
\end{proof}

\subsection{Reality of critical eigenmodes}\label{s:reality}
We now perform a higher-order Lyapunov--Shmidt reduction, addressing the more subtle question of reality of
the critical modes $\lambda_j(\cdot)$.

\begin{proof}[Proof of Theorem \ref{reality}]
	We need further extensions:
	\be
	\begin{split}
		&2\sigma^2h_2-4i\sigma h_1=\frac{-16}{\lambda+6}\sigma^2+\mathcal O(\sigma^4(1+|\lambda|))=-\frac{8}{3}\sigma^2-\frac{38}{27}\sigma^4-\frac{65}{243}\sigma^6+(\frac{{4}}{9}\sigma^2+\frac{43}{27}\sigma^4)\lambda-\frac{2}{27}\sigma^2\lambda^2\\
		&+\mathcal O(\sigma^4(\lambda+\sigma^2)^2+\lambda^3\sigma^2),\\
		&2\sigma^2h_1+4i\sigma h_2=\frac{16}{9}i\sigma^3+\frac{70}{81}i\sigma^5-\frac{28}{27}i\sigma^3\lambda+\mathcal O(\sigma^3(\lambda+\sigma^2)^2),\\
		&8\omega(1+\sigma^2)h_2-16i\omega\sigma h_1+2\sigma^2g_2-4i\sigma g_1=-\frac{212}{81}\omega\sigma^4-\frac{40}{9}\omega\sigma^2\lambda+\mathcal O(\sigma^2(\lambda+\sigma^2)^2),\\
		&-8\omega(1+\sigma^2)h_1-16i\omega\sigma h_2-2\sigma^2g_1-4i\sigma g_2=\frac{32}{3}i\omega\sigma+\frac{16}{27}i\omega\sigma^3-\frac{16}{9}i\omega\sigma\lambda+\mathcal O(\sigma(\lambda+\sigma^2)^2).
	\end{split}
	\ee
		From the co-periodic case, we need the following expansion:
		\be\label{le2}
		\begin{split}
			8\omega\eps\frac{8\omega\eps}{-6-\lambda}=-\frac{32}{3}\omega^2\eps^2+\frac{16}{9}\omega^2\eps^2\lambda+\mathcal O(\eps^2\lambda^2).
		\end{split}
		\ee
	Then,
	\begin{align*}
	\begin{split}
		&m(\eps,\omega,\sigma,\lambda):=\bp m_{11}(\eps,\omega,\sigma,\lambda)&m_{12}(\eps,\omega,\sigma,\lambda)\\
		  m_{21}(\eps,\omega,\sigma,\lambda)&m_{22}(\eps,\omega,\sigma,\lambda)\ep.		
	\end{split}
	\end{align*}
	where
\begin{align*}
		\begin{split}
			m_{11}&=c(\eps)-\lambda-\frac{8}{3}\sigma^2-\frac{38}{27}\sigma^4-\frac{65}{243}\sigma^6+(\frac{{4}}{9}\sigma^2+\frac{43}{27}\sigma^4)\lambda-\frac{2}{27}\sigma^2\lambda^2-\frac{212}{81}\omega\sigma^4\eps-\frac{40}{9}\omega\sigma^2\lambda\eps\\
			&+p_1\sigma^2\eps^2+\frac{16}{9}\omega^2\eps^2\lambda+\mathcal O(\sigma^4(\lambda+\sigma^2)^2+\lambda^3\sigma^2+\sigma^2(\lambda+\sigma^2)^2\eps+(\lambda+\sigma^2)^2\eps^2+(\lambda+\sigma^2)\eps^3),\\
			m_{12}&=-\frac{16}{9}i\sigma^3-\frac{70}{81}i\sigma^5+\frac{28}{27}i\sigma^3\lambda+\frac{32}{3}i\omega\sigma\eps+\frac{16}{27}i\omega\sigma^3\eps-\frac{16}{9}i\omega\sigma\lambda\eps+(\frac{32}{27}\omega^2+\frac{10}{27})i\sigma\eps^2\\			
			&+\mathcal O(\sigma^3(\lambda+\sigma^2)^2+\sigma(\lambda+\sigma^2)^2\eps+(\sigma\lambda+\sigma^3)\eps^2+\sigma\eps^3)),\\			
		\end{split}
	\end{align*}
		\begin{align*}
		\begin{split}				
			m_{21}&=\frac{16}{9}i\sigma^3+\frac{70}{81}i\sigma^5-\frac{28}{27}i\sigma^3\lambda-\frac{32}{3}i\omega\sigma\eps-\frac{16}{27}i\omega\sigma^3\eps+\frac{16}{9}i\omega\sigma\lambda\eps-(\frac{32}{27}\omega^2+\frac{10}{27})i\sigma\eps^2\\			
						&+\mathcal O(\sigma^3(\lambda+\sigma^2)^2+\sigma(\lambda+\sigma^2)^2\eps+(\sigma\lambda+\sigma^3)\eps^2+\sigma\eps^3)),\\
	m_{22}&=-\lambda-\frac{8}{3}\sigma^2-\frac{38}{27}\sigma^4-\frac{65}{243}\sigma^6+(\frac{{4}}{9}\sigma^2+\frac{43}{27}\sigma^4)\lambda-\frac{2}{27}\sigma^2\lambda^2-\frac{212}{81}\omega\sigma^4\eps-\frac{40}{9}\omega\sigma^2\lambda\eps\\
				&+p_2\sigma^2\eps^2+\frac{16}{9}\omega^2\eps^2\lambda+\mathcal O(\sigma^4(\lambda+\sigma^2)^2+\lambda^3\sigma^2+\sigma^2(\lambda+\sigma^2)^2\eps+(\lambda+\sigma^2)^2\eps^2+(\lambda+\sigma^2)\eps^3).\\		
		\end{split}
	\end{align*}
	Let us take the determinant of $m$. Note that $\det m$ is a real-valued function.
	\be\lb{deter}
	\begin{split}
		&\det	m(\eps,\omega,\sigma,\lambda)=\\
		&=\frac{4}{27}\sigma^2\lambda^3+\{(1-\frac{{4}}{9}\sigma^2-\frac{43}{27}\sigma^4+\frac{40}{9}\omega\sigma^2\eps-\frac{16}{9}\omega^2\eps^2)^2-(\frac{28}{27}\sigma^3-\frac{16}{9}\omega\sigma\eps)^2+\frac{32}{81}\sigma^4\}\lambda^2\\
		&+\lambda\big(-c(\eps)+\frac{16}{3}\sigma^2+(\frac{76}{27}-\frac{64}{27})\sigma^4+(\frac{130}{243}-\frac{38}{27}\cdot\frac{8}{9}-\frac{16}{3}\cdot\frac{43}{27}+\frac{32}{9}\cdot\frac{28}{27})\sigma^6\\
		&+(\frac{424}{81}+\frac{16}{3}\cdot\frac{40}{9}-\frac{32}{9}\cdot\frac{16}{9}-\frac{56}{27}\cdot\frac{32}{3})\omega\sigma^4\eps+\frac{4}{9}c(\eps)\sigma^2+(-p_1-p_2-\frac{16}{3}\cdot\frac{16}{9}\omega^2+\frac{64}{3}\cdot\frac{16}{9}\omega^2)\sigma^2\eps^2\\
		&+\mathcal O(\sigma^4(\lambda+\sigma^2)^2+\lambda^3\sigma^2+\sigma^2(\lambda+\sigma^2)^2\eps+(\lambda+\sigma^2)^2\eps^2+(\lambda+\sigma^2)\eps^3)\big)\\
		&+(-c(\eps)+\frac{8}{3}\sigma^2)\frac{8}{3}\sigma^2+(\frac{608}{81}-\frac{16^2}{81})\sigma^6+(\frac{16}{3}\cdot\frac{65}{243}+\frac{38^2}{27^2}-\frac{32}{9}\cdot\frac{70}{81})\sigma^8\\
		&+(\frac{32^2}{27}\sigma^4+(\frac{16}{3}\cdot\frac{212}{81}+\frac{32}{9}\cdot\frac{16}{27}+\frac{140}{81}\cdot\frac{32}{3})\sigma^6)\omega\eps-\frac{38}{27}c(\eps)\sigma^4+(-(\frac{32}{3}\omega\sigma)^2+(-\frac{8}{3}(p_1+p_2)\\
		&+\frac{32}{9}(\frac{32}{27}\omega^2+\frac{10}{27})-\frac{64}{3}\cdot\frac{16}{27}\omega^2)\sigma^4)\eps^2\\
		&+\mathcal O\big((\eps^2+\sigma^2)(\sigma^4(\lambda+\sigma^2)^2+\lambda^3\sigma^2+\sigma^2(\lambda+\sigma^2)^2\eps+(\lambda+\sigma^2)^2\eps^2+(\lambda+\sigma^2)\eps^3)\big)\\
		&+\mathcal O(\sigma\eps)\mathcal O(\sigma^3(\lambda+\sigma^2)^2+\sigma(\lambda+\sigma^2)^2\eps+(\sigma\lambda+\sigma^3)\eps^2+\sigma\eps^3)\\
		&+\big(O(\sigma^3(\lambda+\sigma^2)^2+\sigma(\lambda+\sigma^2)^2\eps+(\sigma\lambda+\sigma^3)\eps^2+\sigma\eps^3)\big)^2.
	\end{split}
	\ee
	According to the Weierstrass Preparation Theorem, there exists a real analytic function $q(\eps,\sigma,\lambda)$ in a neighborhood of $(0,0,0)$ such that $q(0,0,0)=1$ and
	\be\label{wpt2}
	\begin{split}
		q(\eps,\sigma,\lambda)\det m(\eps,\omega,\sigma,\lambda)=\lambda^2+a_1\lambda+a_0.
	\end{split}
	\ee
	
	Let a real analytic function $q(\eps,\sigma,\lambda)$ be of the form
	$
	q(\eps,\sigma,\lambda)=1+\nu_0(\eps,\sigma)+\nu_1(\eps,\sigma)\lambda+\nu_2(\eps,\sigma)\lambda^2+\ldots
	$
	Then,
	$$
	\begin{aligned}
	&a_0(\eps,\sigma)=(1+\nu_0(\eps,\sigma))\det m(\eps,\omega,\sigma,0)
	,\\
	&a_1(\eps,\sigma)=\nu_1(\eps,\sigma)\det m(\eps,\omega,\sigma,0)
	+(1+\nu_0(\eps,\sigma))(\det m)'_{\lambda}(\eps,\omega,\sigma,0).\\	 	
	\end{aligned}
	$$
	Next, in \eqref{wpt2}, we compare the coefficients in front of different powers of $\sigma$, to obtain:
	
	\begin{align*}
	\sigma^0:\quad&(1+\nu_0(\eps,0)+\nu_1(\eps,0)\lambda+\nu_2(\eps,0)\lambda^2+\ldots)((1-\frac{32}{9}\omega^2\eps^2)\lambda^2+\lambda(-c(\eps)+\mathcal O(\lambda^2\eps^2+\lambda\eps^3)))\\
	&=\lambda^2+\lambda(-c(\eps)+\mathcal O(\eps^3)),\\	
	\sigma^1:\quad&(\nu'_{0\sigma}(\eps,0)+\nu'_{1\sigma}(\eps,0)\lambda+\nu'_{2\sigma}(\eps,0)\lambda^2+\ldots)(\lambda^2+\lambda(-c(\eps)+\mathcal O(\lambda\eps^2)))\\
	&=\mathcal O(\eps^2)\lambda,\\ 	
		\end{align*}
	\begin{align*}
	\sigma^2:\quad&(1+\nu_0(\eps,0)+\nu_1(\eps,0)\lambda+\nu_2(\eps,0)\lambda^2+\ldots)\big(\frac{4}{27}\lambda^3+(-\frac{{8}}{9}+\frac{80}{9}\omega\eps)\lambda^2\\
	&+\lambda(\frac{{16}}{3}+\mathcal O(\lambda^3+\lambda^2\eps+\lambda\eps^2+\eps^2))-\frac{8}{3}c(\eps)-(\frac{32}{3}\omega\eps)^2+\mathcal O(\lambda^2\eps^2+\lambda\eps^3+\eps^4)\big)\\
	& +\frac{1}{2}(\nu''_{0\sigma}(\eps,0)+\nu''_{1\sigma}(\eps,0)\lambda+\nu''_{2\sigma}(\eps,0)\lambda^2+\ldots)(\lambda^2+\lambda(-c(\eps)+\mathcal O(\lambda\eps^2)))\\
	&=(\frac{16}{3}+\mathcal O(\eps))\lambda-\frac{8}{3}c(\eps)-(\frac{32}{3}\omega\eps)^2+\mathcal O(\eps^3),\\ 
	\sigma^3:\quad&(\nu'_{0\sigma}(\eps,0)+\nu'_{1\sigma}(\eps,0)\lambda+\nu'_{2\sigma}(\eps,0)\lambda^2+\ldots)\big(\frac{4}{27}\lambda^3+(-\frac{{8}}{9}+\frac{80}{9}\omega\eps)\lambda^2\\
	&+\lambda(\frac{{16}}{3}+\mathcal O(\lambda^3+\lambda^2\eps+\lambda\eps^2+\eps^2))+\mathcal O(\eps^2)\big)+\frac{1}{3!}(\nu'''_{0\sigma}(\eps,0)+\nu'''_{1\sigma}(\eps,0)\lambda+\nu'''_{2\sigma}(\eps,0)\lambda^2+\ldots)\\
	&\times(\lambda^2+\lambda(-c(\eps)+\mathcal O(\lambda\eps^2)))=\mathcal O(1+\eps)\lambda+\mathcal O(\eps^2),\\
	\sigma^4:\quad&(1+\nu_0(\eps,0)+\nu_1(\eps,0)\lambda+\nu_2(\eps,0)\lambda^2+\ldots)\big((\frac{{16}}{81}-\frac{86}{27}+\frac{32}{81})\lambda^2+\lambda(\frac{{4}}{9}+\mathcal O(\lambda^2+\lambda\eps+\eps))
	\\&+\frac{64}{9}+\mathcal O(\lambda^3+\lambda\eps^2+\eps)\big)+\frac{1}{2}(\nu''_{0\sigma}(\eps,0)+\nu''_{1\sigma}(\eps,0)\lambda+\nu''_{2\sigma}(\eps,0)\lambda^2+\ldots)\\
	&\times(\frac{4}{27}\lambda^3+(-\frac{{8}}{9}+\frac{80}{9}\omega\eps)\lambda^2+\lambda(\frac{{16}}{3}+\mathcal O(\lambda^3+\lambda^2\eps+\lambda\eps^2+\eps^2))\\
		&-\frac{8}{3}c(\eps)-(\frac{32}{3}\omega\eps)^2+\mathcal O\big(\lambda^2\eps^2+\lambda\eps^3+\eps^4\big))\\
		& +\frac{1}{4!}(\nu^{(4)}_{0\sigma}(\eps,0)+\nu^{(4)}_{1\sigma}(\eps,0)\lambda+\nu^{(4)}_{2\sigma}(\eps,0)\lambda^2+\ldots)\big(\lambda^2+\lambda(-c(\eps)+\mathcal O(\lambda\eps^2))\big)\\
		&=\mathcal O(1+\eps)\lambda+\mathcal O(1+\eps).
	\end{align*}
	
	Therefore, comparing the coefficients in front of different powers of $\lambda$ leads to
	$$
	\begin{aligned}
	&(1+\nu_0(\eps,0))(1-\frac{32}{9}\omega^2\eps^2+\mathcal O(\eps^3))-\nu_1(\eps,0)c(\eps)=1,\\
	&(1+\nu_0(\eps,0))\mathcal O(\eps^2)+\nu_1(\eps,0)(1-\frac{32}{9}\omega^2\eps^2+\mathcal O(\eps^3))-\nu_2(\eps,0) c(\eps)=0,\\
	&(1+\nu_0(\eps,0))\mathcal O(\eps^2)+\nu_1(\eps,0)\mathcal O(\eps^2)+\nu_2(\eps,0) (1+\mathcal O(\eps^2))+\nu_3(\eps,0)\mathcal O(\eps^2)=0,\\
	&\nu'_{0\sigma}(\eps,0)(1+\mathcal O(\eps^2))-\nu'_{1\sigma}(\eps,0)c(\eps)=0,\\
	&\nu'_{0\sigma}(\eps,0)\mathcal O(\eps^2)+\nu'_{1\sigma}(\eps,0)(1+\mathcal O(\eps^2))-\nu'_{2\sigma}(\eps,0) c(\eps)=0,\\
	&(1+\nu_0(\eps,0))(-\frac{{8}}{9}+\frac{80}{9}\omega\eps+\mathcal O(\eps^2))+\nu_1(\eps,0)(\frac{16}{3}+\mathcal O(\eps^2))+\nu_2(\eps,0)(-\frac{8}{3}c(\eps)-(\frac{32}{3}\omega\eps)^2+\mathcal O(\eps^4))\\
	&+\frac{1}{2}\nu''_{0\sigma}(\eps,0)(1+\mathcal O(\eps^2))-\frac{1}{2}\nu''_{1\sigma}(\eps,0) c(\eps)=0,\\
	&(1+\nu_0(\eps,0))(\frac{4}{27}+\mathcal O(\eps))+\nu_1(\eps,0)(-\frac{{8}}{9}+\mathcal O(\eps))+\nu_2(\eps,0)(\frac{16}{3}+\mathcal O(\eps^2))+\mathcal O(\eps^2)\\
	&+\frac{1}{2}\nu''_{0\sigma}(\eps,0)\mathcal O(\eps^2)+\frac{1}{2}\nu''_{1\sigma}(\eps,0)(1+\mathcal O(\eps^2))-\frac{1}{2}\nu''_{2\sigma}(\eps,0)c(\eps)=0,\\
	&\nu'_{0\sigma}(\eps,0)(-\frac{8}{9}+\mathcal O(\eps))+\nu'_{1\sigma}(\eps,0)(\frac{16}{3}+\mathcal O(\eps^2))+\mathcal O(\eps^2)+\frac{1}{3!}\nu'''_{0\sigma}(\eps,0)(1+\mathcal O(\eps^2))+\mathcal O(\eps^2)=0,\\
		\end{aligned}
		$$
		$$
		\begin{aligned}
	&(1+\nu_0(\eps,0))(-\frac{210}{81}+\mathcal O(\eps))+\nu_1(\eps,0)(\frac{4}{9}+\mathcal O(\eps))+\nu_2(\eps,0)(\frac{64}{9}+\mathcal O(\eps))+\frac{1}{2}\nu''_{0\sigma}(\eps,0)(-\frac{8}{9}+\mathcal O(\eps))\\
	&+\frac{1}{2}\nu''_{1\sigma}(\eps,0)(\frac{16}{3}+\mathcal O(\eps^2))+\mathcal O(\eps^2)+\frac{1}{4!}\nu^{(4)}_{0\sigma}(\eps,0)(1+\mathcal O(\eps^2))+\mathcal O(\eps^2)=0.
	\end{aligned}
	$$
	It is clear that $\nu_0(\eps,0)=\frac{32}{9}\omega^2\eps^2+\mathcal O(\eps^3)$, $\nu_1(\eps,0)=\mathcal O(\eps^2)$, $\nu_2(\eps,0)=\mathcal O(\eps^2)$, $\nu'_{0\sigma}(\eps,0)=\mathcal O(\eps^2)$, $\nu'_{1\sigma}(\eps,0)=\mathcal O(\eps^2)$, $\frac{1}{2}\nu''_{0\sigma}(\eps,0)=\frac{8}{9}-\frac{80}{9}\omega\eps+\mathcal O(\eps^2) $, $\frac{1}{2}\nu''_{1\sigma}(\eps,0)=-\frac{4}{27}+\mathcal O(\eps) $, $\frac{1}{3!}\nu'''_{0\sigma}(\eps,0)=\mathcal O(\eps^2)$, $\frac{1}{4!}\nu^{(4)}_{0\sigma}(\eps,0)=\frac{210}{81}+\frac{64}{81}+\frac{64}{81}+\mathcal O(\eps)=\frac{338}{81}+\mathcal O(\eps)$.\\
	Then,
	$$
	\begin{aligned}
	a_0(\eps,\sigma)&=(1+\nu_0(\eps,\sigma))\det m(\eps,\omega,\sigma,0)\\
	&=(-c(\eps)+\frac{8}{3}\sigma^2)\frac{8}{3}\sigma^2+(\frac{352}{81}+\frac{8^3}{81})\sigma^6+(\frac{244}{3^6}+\frac{8}{9}\cdot\frac{352}{81}+\frac{338}{81}\cdot\frac{64}{9})\sigma^8\\
	&+(\frac{32^2}{27}\sigma^4+(\frac{8384}{3^5}+\frac{8}{9}\cdot\frac{32^2}{27}-\frac{80}{9}\cdot\frac{64}{9})\sigma^6)\omega\eps-\frac{38}{27}c(\eps)\sigma^4+(-(\frac{32}{3}\omega\sigma)^2+(-\frac{8}{3}(p_1+p_2)\\
	&-\frac{2048}{3^5}\omega^2+\frac{320}{3^5}+\frac{32}{9}\cdot\frac{64}{9}\omega^2-\frac{8}{9}\cdot\frac{32^2}{3^2}\omega^2)\sigma^4)\eps^2-\frac{8}{9}\cdot\frac{8}{3}\sigma^4c(\eps)\\
	&+\mathcal O\big(\sigma^{9}+\sigma^7\eps+\sigma^5\eps^2+\sigma^4\eps^3+\sigma^2\eps^4\big)
	,\\
	a_1(\eps,\sigma)&=\nu_1(\eps,\sigma)\det m(\eps,\omega,\sigma,0)
	+(1+\nu_0(\eps,\sigma))(\det m)'_{\lambda}(\eps,\omega,\sigma,0)\\
	&=-c(\eps)+\frac{16}{3}\sigma^2+(\frac{4}{9}+\frac{8}{9}\cdot\frac{16}{3})\sigma^4+(-\frac{1342}{3^5}+\frac{8}{9}\cdot\frac{4}{9}+\frac{338}{81}\cdot\frac{16}{3}-\frac{4}{27}\cdot\frac{64}{9})\sigma^6\\
	&+(\frac{40}{81}-\frac{80}{9}\cdot\frac{16}{3})\omega\sigma^4\eps+(\frac{4}{9}-\frac{8}{9})c(\eps)\sigma^2+(-p_1-p_2+\frac{768}{27}\omega^2+\frac{32}{9}\cdot\frac{16}{3}\omega^2)\sigma^2\eps^2\\
	&+O\big(\sigma^7+\sigma^5\eps+\sigma^3\eps^2+\sigma^2\eps^3+\sigma\eps^4\big).\\	 	
	\end{aligned}
	$$

	Therefore, the eigenvalue problems boils down to the second order polynomial
	
	\begin{align}\lb{pol2}
	\begin{split}
	&\lambda^2+\big(-c(\eps)+\frac{16}{3}\sigma^2+\frac{140}{27}\sigma^4+\frac{434}{27}\sigma^6-\frac{3800}{81}\omega\sigma^4\eps-\frac{4}{9}c(\eps)\sigma^2+(-p_1-p_2+\frac{1280}{27}\omega^2)\sigma^2\eps^2\\
	&+O\big(\sigma^7+\sigma^5\eps+\sigma^3\eps^2+\sigma^2\eps^3+\sigma\eps^4\big)\big)\lambda+(-c(\eps)+\frac{8}{3}\sigma^2)\frac{8}{3}\sigma^2+\frac{32}{3}\sigma^6+\frac{24692}{3^6}\sigma^8\\
	&+\frac{32^2}{27}\omega\sigma^4\eps+\frac{1216}{3^5}\omega\sigma^6\eps-\frac{34}{9}c(\eps)\sigma^4-(\frac{32}{3}\omega\sigma\eps)^2+(-\frac{8}{3}(p_1+p_2)-\frac{2^{12}\cdot5}{3^5}\omega^2+\frac{320}{3^5})\sigma^4\eps^2\\
	&+\mathcal O\big(\sigma^{9}+\sigma^7\eps+\sigma^5\eps^2+\sigma^4\eps^3+\sigma^2\eps^4\big)=0,
	\end{split}
	\end{align}

	yielding the roots
	\be\label{quad2}
	\begin{split}
		\lambda_{1,2}=&\frac{\tilde c(\eps,\sigma)
			\pm\sqrt{c^2(\eps)+\frac{32^2}{9}\sigma^2(\frac{1}{3}\sigma^2-2\omega\eps)^2+\frac{2^8}{3^5}\sigma^4(\frac{179}{3}\sigma^4-494\omega\sigma^2\eps+800\omega^2\eps^2-5\eps^2)+
			}}{2}\\
			&\frac{\overline{\mathcal O\big(\sigma^{9}+\sigma^7\eps+\sigma^5\eps^2+\sigma^4\eps^3+\sigma^2\eps^4+\sigma\eps^6\big)}}{}.
		\end{split}
		\ee
		Next, let us consider different ranges for $\omega$.\\
		Case 1. Assume that $\omega\in(-\frac{1}{4},0]$. It follows from \eqref{quad2} that the discriminant in nonnegative for small enough values of $\sigma$ and $\eps$ which implies that $\lambda_1$ and $\lambda_2$ are real.\\
		Case 2. Assume that $\omega=-\frac{1}{4}$. Then $c(\eps)=0$. Moreover, it follows from \eqref{pol2} that the error term in the discriminant is of the form  $O\big(\sigma^{9}+\sigma^7\eps+\sigma^5\eps^2+\sigma^4\eps^3+\sigma^2\eps^4\big)$ i.e. $\mathcal O(\sigma\eps^6\big)=0$. Therefore, the discriminant is nonnegative for small enough values of $\sigma$ and $\eps$ which implies that $\lambda_1$ and $\lambda_2$ are real.\\
		Case 3. Assume that $\omega\in(0,\frac{1}{4}]$. We use the following ansatz for $\sigma$, i.e. $\sigma^2=6\omega\eps$. Then 
			\be\label{quad3}
			\begin{split}
				\lambda_{1,2}=&\frac{\tilde c(\eps,\sigma)
					\pm\sqrt{-\frac{2^{12}}{3^3}\eps^4(\omega^2-\omega_1)(\omega^2-\omega_2)+\mathcal O(\eps^5)
					}}{2},
				\end{split}
				\ee
				where $\omega_1\approx0.007$, $\omega_2\approx-1.63$. Therefore, for $\omega\in(\sqrt{\omega_1},\frac{1}{4}]$, $\lambda_1$ and $\lambda_2$ are not real.
			\end{proof}


\section{Comparison with Ginzburg Landau approximation}\label{s:GLcompare}
In this section we study in more detail the various operations in the Ginzburg Landau expansion, 
showing that, after natural preconditioning passes on each side, these can be matched step by step
with those of the exact Lyapunov-Schmidt reduction procedure.
This gives a deeper explanation why the two procedures give the same expansion to their common order of approximation.
In the process, we carry out the proof of Theorem \ref{dispthm}.

\subsection{Ginzburg Landau derivation through multiscale expansion}\label{s:GLexp}
We start by deriving in detail the Ginzburg Landau equation as a modulation equation of the 
Brusselator model \eqref{New Brusselator model}, i.e.
\be \label{sec3}
\begin{split}
	\partial_t U & = D\partial_x^2U +\bp 3+\eps^2 & 4 \\ -4+\eps^2 & -4 \ep U +\bp(\frac{\eps^2}{2}+2)u_1^2 + 4u_1u_2+ u_1^2 u_2\\-(\frac{\eps^2}{2}+2)u_1^2 - 4u_1u_2- u_1^2 u_2\ep.
\end{split}
\ee

The derivation is based on the ansatz
\be 
\begin{split}
	& u(t,x)\approx U_A(\hat t,\hat x)=\frac{1}{2}\eps A(\hat t,\hat x)e^{i\frac{1}{2}x}\bp 2 \\ -1 \ep+c.c.+\eps^2\Psi_0(\hat t,\hat x)\\&+\frac{1}{2}\eps^2(e^{i\frac{1}{2}x}\Psi_1(\hat t,\hat x)+e^{ix}\Psi_2(\hat t,\hat x)+c.c.)+\frac{1}{2}\eps^3e^{i\frac{1}{2}x}\Psi_3(\hat t,\hat x)+c.c.+h.o.t.,
\end{split}
\ee
where $(\hat t,\hat x)=(\eps^2 t,\eps x)$.

Substituting this ansatz into \eqref{sec3} and collecting terms of the form $\eps^{j_1}e^{i\frac{1}{2}{j_2}x}$, we arrive at the equations:
\be 
\begin{split}\label{steps}
 \eps e^{i\frac{1}{2}x}:\quad 0=	& \frac{1}{2}(-\frac{1}{4}AD+A\bp 3 & 4 \\ -4 & -4 \ep)\bp 2 \\ -1 \ep,\\
	\eps^2 :\quad 0=	&\bp 3 & 4 \\ -4 & -4 \ep\Psi_0+|A|^2(2\cdot 2e^{i\frac{1}{2}x}e^{-i\frac{1}{2}x}+4\cdot2\cdot(-\frac{1}{2}))\bp 1 \\ -1 \ep,\\
	\eps^2e^{i\frac{1}{2}x} :\quad 0=&\frac{1}{2}\big\{-\frac{1}{4}D\Psi_1+\bp 3 & 4 \\ -4 & -4 \ep\Psi_1+i\partial_{\hat x}A \bp 8 \\ -16  \ep\big\},\\
	\eps^2e^{ix} :\quad 0=	&\frac{1}{2}\big\{-D\Psi_2+\bp 3 & 4 \\ -4 & -4 \ep\Psi_2\big\}+A^2 (2+4\cdot(-\frac{1}{2}))\bp 1 \\ -1  \ep.\\
\end{split}
\ee

Therefore, $\Psi_0=\Psi_2=0$ and $\Psi_1=c\bp 2 \\ -1  \ep+\frac{4}{3}i\partial_{\hat x}A\bp 1 \\ -2  \ep$.\\
We also arrive at the compatibility condition:
 \be 
 \begin{split}
 	\eps^3 e^{i\frac{1}{2}x}:\quad \frac{1}{2}\partial_{\hat t}A\bp 2 \\ -1 \ep=	& \frac{1}{2}\big\{D\partial^2_{\hat x}A\bp 2 \\ -1 \ep+A\bp 2 \\ -2 \ep+iD\partial_{\hat x}\Psi_1-\frac{1}{4}D\Psi_3+\bp 3 & 4 \\ -4 & -4 \ep\Psi_3\big\}\\
 	&+(-\frac{1}{2}\cdot 2e^{i\frac{1}{2}x}e^{-i\frac{1}{2}x}-\frac{1}{2}e^{ix}e^{-i\frac{1}{2}x})|A|^2A\bp 1 \\ -1 \ep,
 \end{split}
 \ee
or
  \be \label{e^3}
  \begin{split}
  	\eps^3 e^{i\frac{1}{2}x}:\quad \bp 2\psi_3+4\tilde\psi_3 \\ -4\psi_3-8\tilde\psi_3 \ep=	& \partial_{\hat t}A\bp 2 \\ -1 \ep-\partial^2_{\hat x}A\bp 8 \\ -16 \ep-A\bp 2 \\ -2 \ep-i\partial_{\hat x}c\bp 8 \\ -16 \ep\\
  	&+\frac{4}{3}\partial^2_{\hat x}A\bp 4 \\ -32 \ep+3|A|^2A\bp 1 \\ -1 \ep.
  \end{split}
  \ee
  Therefore, we have the compatibility condition
  \be 
    \begin{split}\label{com}
    	0=	& 3\partial_{\hat t}A-2A-32\partial^2_{\hat x}A+3|A|^2A.
    \end{split}
    \ee
    Hence, we arrive at the Ginzburg Landau equation:
    \be 
        \begin{split}\label{GL1}
        	\partial_{\hat t}A=	& \frac{32}{3}\partial^2_{\hat x}A+\frac{2}{3}A-|A|^2A.
        \end{split}
        \ee
        
        Note that \eqref{GL1} has the explicit solution:
        \be 
        \begin{split}\label{GL2}
        	A_{\omega}(\hat x)=\sqrt{\frac{2}{3}(1-16\omega^2)}e^{i\omega\hat x},\,\,\omega\in[-\frac{1}{4},\frac{1}{4}].
        \end{split}
        \ee
        Hence, if we go through the steps in formulas \eqref{steps}-\eqref{GL2}, replace $\partial_{\hat x}$ by $i\omega$ and ignore $\hat t $ dependence of $A$, we will arrive at the equation for $A$:
$$
        	-\frac{32}{3}\omega^2A+\frac{2}{3}A-|A|^2A=0.
$$

\subsection{Existence: exact theory vs. Ginzburg-Landau approximation}
Now that we know the precise scaling in the existence part we use the following ansatz to go through the existence steps and compare them to the steps of the Ginzburg Landau derivation:
$$
 u=\cos \xi\bp 2 \\ -1 \ep\alpha\eps +V(\eps,\alpha).
$$
We substitute this ansatz into the equation
\be 
\begin{split}\label{exist}
& (I-Q)N(b(\eps),k(\eps,\omega), u)=0.
\end{split}
\ee

1.) First of all, it is clear that $V(0,\alpha)=0$.\\
2.) Now, we differentiate  equation \eqref{exist} with respect to $\eps$.
\be\label{eps}
\begin{split}
	& \partial_{\eps}(I-Q)N(b(\eps),k(\eps), \alpha\eps U_1+ V)=(I-Q)[((\omega+\mathcal O(\eps))D\partial_\xi^2+ 2\eps\bp 1 & 0 \\ -1 & 0 \ep)(\alpha\eps U_1 +V)\\&+((\frac{1}{4}+\mathcal O(\eps))D\partial_\xi^2+ A)(\alpha U_1+\partial_{\eps}V)\\&+  \bp (\eps^2+4) u_1+ 4 u_2+2 u_1  u_2  & 4 u_1+ u_1^2 \\ -(\eps+4) u_1- 4 u_2-2 u_1  u_2  & -4 u_1- u_1^2 \ep (\alpha U_1+\partial_{\eps}V)]=0,
\end{split}
\ee
where $\bp u_1 \\ u_2 \ep=\alpha\eps U_1 + V$.\\
Hence, by step 1.),
\be
\begin{split}
 (I-Q)L_{per}\partial_{\eps}V|_{\eps=0}=-\alpha(I-Q)L_{per}U_1=0.
 \end{split}
 \ee
 Since $(I-Q)L_{per}(I-P)$ is invertible and $\partial_{\eps}V|_{\eps=0}\in\ran(I-P)$, $\partial_{\eps}V|_{\eps=0}=0$.

3.) Next, we would like to compute $\partial^2_{\eps}V|_{\eps=0}$.  Differentiating \eqref{eps} with respect to $\eps$, we obtain
%
\be\label{eps2}
\begin{split}
	& \partial^2_{\eps}(I-Q)N(b(\eps),k(\eps), \alpha\eps U_1+ V)=(I-Q)[((4\omega^2+\mathcal O(\eps))D\partial_\xi^2+ 2\bp 1 & 0 \\ -1 & 0 \ep)(\alpha\eps U_1 +V)\\
	&+2((\omega+\mathcal O(\eps))D\partial_\xi^2+ 2\eps\bp 1 & 0 \\ -1 & 0 \ep)(\alpha U_1 +\partial_{\eps}V)\\
	&+((\frac{1}{4}+\mathcal O(\eps))D\partial_\xi^2+ A)\partial^2_{\eps}V+2\eps\bp  u_1  & 0 \\ - u_1  & 0 \ep \partial^2_{\eps}V
	\\&+  \bp (\eps^2+4)+2 u_2  & 0 \\ -(\eps^2+4) -2   u_2  & 0 \ep (\alpha U_1+\partial_{\eps}V)^{\circ2}+ \bp 4 +2 u_1  & 4 + 2u_1 \\ - 4 -2 u_1  & -4 - 2u_1 \ep \left\{(\alpha\hat U_1+\partial_{\eps}\hat V)\circ(\alpha U_1+\partial_{\eps}V)\right\}\\&+  \bp (\eps^2+4) u_1+ 4 u_2+2 u_1  u_2  & 4 u_1+ u_1^2 \\ -(\eps+4) u_1- 4 u_2-2 u_1  u_2  & -4 u_1- u_1^2 \ep \partial^2_{\eps}V]=0,
\end{split}
\ee
where $\circ$ indicates the Hadamard product sign.
Therefore,
\be\label{eps^2}
\begin{split}
	& \partial^2_{\eps}(I-Q)N(b(\eps),k(\eps), \alpha\eps U_1+ V)|_{\eps=0}=(I-Q)[2\omega\alpha D\partial_\xi^2 U_1+L_{per}\partial^2_{\eps}V|_{\eps=0}\\&+  \bp 4 & 0 \\ -4   & 0 \ep (\alpha U_1)^{\circ2}+ \bp 4   & 4  \\ - 4   & -4  \ep \left\{(\alpha\hat U_1)\circ(\alpha U_1)\right\}]=0.
\end{split}
\ee

Notice that
\be
\begin{split}
	&(\alpha U_1)^{\circ2}=\alpha^2\bp 4 \\ 1 \ep \cos^2 \xi,\\
	&(\alpha\hat U_1)\circ(\alpha U_1)=-2\alpha^2\bp 1 \\ 1 \ep \cos^2 \xi.\\
\end{split}
\ee

It is easy to see that 
\be
Q \Big[DU_1 \Big] = 0,\quad Q \Big[\bp 1 \\ -1 \ep \Big] = 0, \quad Q \Big[\bp 1 \\ -1 \ep \cos 2\xi \Big] = 0.
\ee
Therefore, using \eqref{eps^2}, we obtain
\ba \label{formula1^2}
	(I-Q)L_{per}\partial^2_{\eps}V|_{\eps=0}& =2\omega\alpha \bp 8 \\ -16  \ep\cos \xi \\
&\quad -8\alpha^2\bp 1 \\ -1  \ep \cos 2\xi-8\alpha^2\bp 1 \\ -1  \ep+8\alpha^2\bp 1 \\ -1  \ep \cos 2\xi+8\alpha^2\bp 1 \\ -1  \ep.
\ea
Since $\mathbb{R}^2$, $\mathbb{R}^2\cos \xi$ and $\mathbb{R}^2\cos 2\xi$ are invariant subspaces for the invertible operator $(I-Q)L_{per}(I-P)$, $\frac{1}{2}\partial^2_{\eps}V|_{\eps=0}$ is of the form $\Psi_0 +\Psi_1 \cos \xi+\Psi_2 \cos 2\xi$. Notice that $\Psi_1$ should be of the form $\phi_1\bp 1 \\ -2  \ep$,
and have the equations for $\Psi_j$:
$$
\begin{aligned}
	\quad 0=	&\bp 3 & 4 \\ -4 & -4 \ep\Psi_0+4\alpha^2\bp 1 \\ -1  \ep-4\alpha^2\bp 1 \\ -1  \ep,\\
	\quad 0=	&-\frac{1}{4}D\Psi_1+\bp 3 & 4 \\ -4 & -4 \ep\Psi_1-\omega\alpha \bp 8 \\ -16  \ep,\\
	\quad 0=	&-D\Psi_2+\bp 3 & 4 \\ -4 & -4 \ep\Psi_2+4\alpha^2\bp 1 \\ -1  \ep-4\alpha^2\bp 1 \\ -1  \ep.
\end{aligned}
$$
Therefore, $\Psi_0=\Psi_2=0$ and $\Psi_1=-\frac{4}{3}\omega\alpha\bp 1 \\ -2  \ep$.

So far, we have shown that 
\be
\begin{split}\label{uu}
	&u=\alpha\cos \xi\bp 2 \\ -1 \ep\eps -\frac{4}{3}\omega\alpha\cos \xi\bp 1 \\ -2  \ep\eps^2+\mathcal O(\eps^3).
\end{split}
\ee
Also, notice that in formula \eqref{uu} $\mathcal O(\eps^3)=\mathcal O(\a\eps^3)$. Then
\be
\begin{split}\label{u}
	&u=\alpha\cos \xi\bp 2 \\ -1 \ep\eps -\frac{4}{3}\omega\alpha\cos \xi\bp 1 \\ -2  \ep\eps^2+\mathcal O(\a\eps^3).
\end{split}
\ee

In order to obtain the reduced equation, we plug \eqref{u} into the first equation from \eqref{projection equations}, obtaining
for $\bp u_1 \\ u_2 \ep=u$:
$$
\begin{aligned}
& \tilde QN(b(\eps),k(\eps), u)=\tilde Q[(k^2(\eps)D\partial_\xi^2+ A)u\\&+  \bp  (\frac{\eps^2}{2}+2)u_1^2 + 4u_1u_2+ u_1^2 u_2  \\    -(\frac{\eps^2}{2}+2)u_1^2 - 4u_1u_2- u_1^2 u_2  \ep ]=0.
\end{aligned}
$$

a.) {\bf Linear part.} By direct calculation, we have:
$$
D\bp 2 \\ -1 \ep = 8 \bp 1 \\ -2 \ep , \quad A\bp 2 \\ -1 \ep = \bp 2\eps^2+2 \\ -2\eps^2-4 \ep,
$$
$$
Q  k^2(\eps)D\partial_\xi(  \a\eps U_1) =0, \quad Q A ( \a\eps U_1 ) = \frac{2\eps^3\a}{3} \bp 2 \\ -1 \ep \cos \xi ,
$$
$$
Q\bp 1 \\ -2 \ep =  \bp 4 \\ -32 \ep , \quad A\bp 1 \\ -2 \ep = \bp \eps^2-5 \\ -\eps^2+4 \ep,
$$
$$
Q\Big[D\bp 1 \\ -2 \ep \cos \xi \Big] =  -8   \bp 2 \\ -1 \ep\cos \xi,\,\,\,
Q \Big[A\bp 1 \\ -2 \ep \cos \xi \Big] = \frac{\eps^2-6}{3}  \bp 2 \\ -1 \ep\cos \xi.
$$

Therefore,
\be
\begin{split}
& \tilde Q(k^2(\eps)D\partial_\xi^2+ A)(\a U_1)= \frac{2\a}{3}\eps^3 +( 8(\frac{1}{4}+\omega\eps)+\frac{-6}{3})(-\frac{4}{3}\omega\alpha\eps^2) +\mathcal O(\a\eps^4).
\end{split}
\ee

b.) {\bf Non-linear part.}
By direct calculation, we have:
$$
\begin{aligned}
&\tilde Q\bp  (\frac{\eps^2}{2}+2)u_1^2 + 4u_1u_2+ u_1^2 u_2  \\    -(\frac{b}{\eps^2}+2)u_1^2 - 4u_2u_2- u_1^2 u_2  \ep =\tilde Q\big( ((\frac{\eps^2}{2}+2)u_1^2 + 4u_1u_2+ u_1^2 u_2)\bp 1 \\ -1 \ep\big)\\
&=-\a^3\eps^3+\mathcal O(\a\eps^4).
\end{aligned}
$$

Hence, we have the reduced equation:
\be
\begin{split}
\frac{2}{3}\a\eps^3-\frac{32}{3}\omega^2\alpha\eps^3 -\a^3\eps^3+\mathcal O(\a\eps^4)=0.
\end{split}
\ee
It is equivalent to 
\be
\begin{split}
	-\frac{32}{3}\omega^2\alpha+\frac{2}{3}\a -\a^3+\mathcal O(\a\eps)=0.
\end{split}
\ee

Note that, under the imposed scaling, the computations of the reduced (existence) equation derived by Lyapunov--Schmidt
reduction agree at each order with derived by formal Ginzburg Landau approximation.

\subsection{Stability: exact vs. Ginzburg-Landau linearized dispersion relations}\label{s:ccred}
Next, we derive the linearized dispersion relations for the Ginzburg Landau equation
\be 
\begin{split}
	\partial_{\hat t}A=	& \frac{32}{3}\partial^2_{\hat x}A+\frac{2}{3}A-|A|^2A.
\end{split}
\ee
In order to study the linearized stability of $A_{\omega}(\hat x)=\sqrt{\frac{2}{3}(1-16\omega^2)}e^{i\omega\hat x},\,\,\omega\in[-\frac{1}{4},\frac{1}{4}]$ we use the ansatz $A=A_{\omega}+e^{i\omega\hat x}(b_r-ib_i)$ and linearize in $(b_r,b_i)$. Therefore, we obtain
\be 
\begin{split}\label{spect}
	\partial_{\hat t}\bp b_r \\ b_i  \ep=	& \bp \frac{32}{3}\partial^2_{\hat x}-\frac{4}{3}(1-16\omega^2) & \frac{64}{3}\omega\partial_{\hat x}\\ -\frac{64}{3}\omega\partial_{\hat x}& \frac{32}{3}\partial^2_{\hat x} \ep\bp b_r \\ b_i  \ep.
\end{split}
\ee

Now we would like to derive equation \eqref{spect} using linearized equation 
\be
\begin{split}\label{B1}
\partial_{\hat t}B
	& =D\partial^2_{\hat x}B+\bp \eps^2+3 & 4 \\ -(\eps^2+4) & -4 \ep +   \bp (\eps^2+4)\tilde u_1+ 4\tilde u_2+2\tilde u_1 \tilde u_2  & 4\tilde u_1+\tilde u_1^2 \\ -(\eps^2+4)\tilde u_1- 4\tilde u_2-2\tilde u_1 \tilde u_2  & -4\tilde u_1-\tilde u_1^2 \ep B,
\end{split}
\ee
where

\be 
\begin{split}
&\bp	\tilde u_1\\ \tilde u_2\ep=U_A(\hat t,\hat x)=\frac{1}{2}\eps A_{\omega}(\hat x)e^{i\frac{1}{2}x}\bp 2 \\ -1 \ep+c.c.+\frac{1}{2}\eps^2e^{i\frac{1}{2}x}\Psi_1(\hat x)+c.c.+h.o.t.,\\
&\Psi_1(\hat x)=-\frac{4}{3}\omega A_{\omega}(\hat x)\bp 1 \\ -2  \ep,
\end{split}
\ee
and the ansatz
\be 
\begin{split}
	&\B(\hat t,\hat x)=b_re^{i(\omega\hat x+\frac{1}{2}x)}\bp 2 \\ -1 \ep-ib_ie^{i(\omega\hat x+\frac{1}{2}x)}\bp 2 \\ -1 \ep+c.c.+\eps e^{i(\omega\hat x+\frac{1}{2}x)}(\Phi_1(\hat t,\hat x)-i\tilde\Phi_1(\hat t,\hat x)+c.c.)\\
	&\eps^2 e^{i(\omega\hat x+\frac{1}{2}x)}(\Phi_2(\hat t,\hat x)-i\tilde\Phi_2(\hat t,\hat x)+c.c.)+h.o.t.
\end{split}
\ee
	
Substituting this ansatz into \eqref{B} and collecting terms of the form $(-i)^{j_1}\eps^{j_2}e^{i(\omega\hat x+\frac{1}{2}x)}$, we arrive at the equations:

\be 
\begin{split}
	e^{i(\omega\hat x+\frac{1}{2}x)}:\quad 0=	& \big[-\frac{1}{4}b_rD+b_r\bp 3 & 4 \\ -4 & -4 \ep\big]\bp 2 \\ -1 \ep,\\
	-ie^{i(\omega\hat x+\frac{1}{2}x)}:\quad 0=	& \big[-\frac{1}{4}b_iD+b_i\bp 3 & 4 \\ -4 & -4 \ep\big]\bp 2 \\ -1 \ep,\\
	\eps e^{i(\omega\hat x+\frac{1}{2}x)} :\quad 0=&-\omega b_rD \bp 2 \\ -1  \ep+\partial_{\hat x}b_iD \bp 2 \\ -1  \ep-\frac{1}{4}D\Phi_1+\bp 3 & 4 \\ -4 & -4 \ep\Phi_1,\\
	-i\eps e^{i(\omega\hat x+\frac{1}{2}x)} :\quad 0=&-\partial_{\hat x} b_rD \bp 2 \\ -1  \ep-\omega b_iD \bp 2 \\ -1  \ep-\frac{1}{4}D\tilde\Phi_1+\bp 3 & 4 \\ -4 & -4 \ep\tilde\Phi_1.\\
\end{split}
\ee

Therefore, $\Phi_1=c\bp 2 \\ -1  \ep+\frac{4}{3}(-\omega b_r+\partial_{\hat x}b_i)\bp 1 \\ -2  \ep$ and $\tilde\Phi_1=\tilde c\bp 2 \\ -1  \ep+\frac{4}{3}(-\partial_{\hat x} b_r-\omega b_i)\bp 1 \\ -2  \ep$.\\
We also arrive at the compatibility conditions:
\be 
\begin{split}
	\eps^2 e^{i(\omega\hat x+\frac{1}{2}x)}:\quad& \partial_{\hat t}b_r\bp 2 \\ -1 \ep=	 (\partial^2_{\hat x}b_r-\omega^2 b_r+2\omega \partial_{\hat x}b_i)D\bp 2 \\ -1 \ep+b_r\bp 2 \\ -2 \ep-\omega D\Phi_1+D\partial_{\hat x}\tilde\Phi_1\\
	&-\frac{1}{4}D\Phi_2+\bp 3 & 4 \\ -4 & -4 \ep\Phi_2+(2(-\frac{1}{2}-\frac{1}{2}-\frac{1}{2})2b_r+(1+2)(-b_r))|A_{\omega}|^2\bp 1 \\ -1 \ep,\\
	-i\eps^2 e^{i(\omega\hat x+\frac{1}{2}x)}:\quad& \partial_{\hat t}b_i\bp 2 \\ -1 \ep=	 (\partial^2_{\hat x}b_i-\omega^2 b_i-2\omega \partial_{\hat x}b_r)D\bp 2 \\ -1 \ep+b_i\bp 2 \\ -2 \ep-D\partial_{\hat x}\Phi_1-\omega D\tilde\Phi_1\\
		&-\frac{1}{4}D\tilde\Phi_2+\bp 3 & 4 \\ -4 & -4 \ep\tilde\Phi_2+(2(-\frac{1}{2}-\frac{1}{2}+\frac{1}{2})2b_i+(-1+2)(-b_i))|A_{\omega}|^2\bp 1 \\ -1 \ep.
\end{split}
\ee

Similar to \eqref{e^3}-\eqref{com}, we have the compatibility conditions:
\be 
\begin{split}
	0=	& -3\partial_{\hat t}b_r+2b_r-\omega(-24)\frac{4}{3}(-\omega b_r+\partial_{\hat x}b_i)+(-24)\frac{4}{3}(-\partial^2_{\hat x} b_r-\omega \partial_{\hat x}b_i)-9|A_{\omega}|^2b_r,\\
	0=	& -3\partial_{\hat t}b_i+2b_i-(-24)\frac{4}{3}(-\omega \partial_{\hat x}b_r+\partial^2_{\hat x}b_i)-\omega(-24)\frac{4}{3}(-\partial_{\hat x} b_r-\omega b_i)-3|A_{\omega}|^2b_i.
\end{split}
\ee
Finally, we arrive at
\be 
\begin{split}
	\partial_{\hat t}\bp b_r \\ b_i  \ep=	& \bp \frac{32}{3}\partial^2_{\hat x}-\frac{4}{3}(1-16\omega^2) & \frac{64}{3}\omega\partial_{\hat x}\\ -\frac{64}{3}\omega\partial_{\hat x}& \frac{32}{3}\partial^2_{\hat x} \ep\bp b_r \\ b_i  \ep.
\end{split}
\ee
Hence, the spectral matrix of the linearized operator is of the form
\be\label{gldisp}
\begin{split}
	0
	&=\bp -\frac{4}{3}(1-16\omega^2)-\frac{32}{3}\hat\sigma^2-\hat\lambda & \frac{64}{3}i\omega\hat\sigma \\ -\frac{64}{3}i\omega\hat\sigma & -\frac{32}{3}\hat\sigma^2-\hat\lambda \ep\bp \b_1 \\ \b_2\ep.
\end{split}
\ee

Then, for $|\hat\sigma|<<1$
	\be\label{eGL}
	 		 		 	\begin{split}
	 		 		 	&\hat\lambda_{1}(\hat\sigma)=-\frac{4}{3}(1-16\omega^2)-\frac{32(1+16\omega^2)}{3(1-16\omega^2)}\hat\sigma^2+\mathcal O(\hat\sigma^3),\\
	 		 		 	&\hat\lambda_{2}(
	 		 		 	\hat\sigma)=-\frac{32(1-48\omega^2)}{3(1-16\omega^2)}\hat\sigma^2+\mathcal O(\hat\sigma^3).
	 		 		 		\end{split}
	 		 		 			\ee

Note that, as in the existence part, the derivation by rigorous Lyapunov-Schmidt reduction, in the Ginzburg Landau scaling, 
agrees at each step/order with that by formal Ginzburg Landau approximation as can be seen below.

Now that we know the explicit form of solution in the existence part we use the scaling $\lambda=\eps^2\hat\lambda, \,\sigma=2\eps\hat \sigma$ to go through the spectral matrix steps and compare them to the steps in section 5.\\
We now consider the eigenvalue problem of $B(\eps,\omega,\sigma)$:
\be
0=\Big[ B(\eps,\omega,2\hat\sigma\eps)-\hat\l\eps^2 I \Big] W.
\ee
In order to use the Lyapunov-Schmidt reduction, we decompose $W = \b_1 U_{1}+\b_{2} U_{2} + \mathcal V$ and we first solve
\be\label{scaling}
\begin{split}
	0
	& =(I-Q)\Big[ B(\eps,\omega,2\hat\sigma\eps)-\hat\l\eps^2 I\Big] (\b_1 U_{1}+\b_{2} U_{2} + \mathcal V),
\end{split}
\ee
It is clear that the relation between $\b$ and $\mathcal V$ is linear. Then, let $\mathcal V=\mathcal V_1\b_1+\mathcal V_2\b_2$. Now let us find asymptotic expansions of $\mathcal V_1$ and $\mathcal V_2$ with respect to parameter $\eps$.\\
1.) First, we compute $\mathcal V_i|_{\eps=0}$. \\
\be
\begin{split}
	0
	& =(I-Q) B(0,\omega,0) (\b_1 (U_{1}+\mathcal V_1|_{\eps=0})+\b_{2} (U_{2}+\mathcal V_i|_{\eps=0})).
\end{split}
\ee

Since $(I-Q) B(0,\omega,0) (\b_1 U_{1}+\b_{2} U_{2})=0$ and $(I-Q) B(0,\omega,0)(I-P)$ is invertible, $\mathcal V_i|_{\eps=0}=0$.

2.) Now, we differentiate  equation \eqref{scaling} with respect to $\eps$.
\be
\begin{split}
	0
	& =(I-Q)\big\{\partial_{\eps}B(0,\omega,0) (\b_1 U_{1}+\b_{2} U_{2})+ B(0,\omega,0)(\b_1 \partial_{\eps}\mathcal V_1|_{\eps=0}+\b_{2} \partial_{\eps}\mathcal V_i|_{\eps=0})\big\}.
\end{split}
\ee
Or,
\be\label{derPhi}
\begin{split}
	0
	& =(I-Q)\big\{[-\omega(\b_1 DU_{1}+\b_{2}D U_{2})]+i\hat\sigma(-\b_1 DU_{2}+\b_{2}D U_{1})+ L_{per}(\b_1 \partial_{\eps}\mathcal V_1|_{\eps=0}+\b_{2} \partial_{\eps}\mathcal V_i|_{\eps=0})\big\}\\
	&=-\omega(\b_1 DU_{1}+\b_{2}D U_{2})+i\hat\sigma(-\b_1 DU_{2}+\b_{2}D U_{1})+(I-Q)L_{per}(\b_1 \partial_{\eps}\mathcal V_1|_{\eps=0}+\b_{2} \partial_{\eps}\mathcal V_i|_{\eps=0}).
\end{split}
\ee
Therefore,

\be\label{Phi}
\begin{split}
	&\b_1 \partial_{\eps}\mathcal V_1|_{\eps=0}+\b_{2} \partial_{\eps}\mathcal V_i|_{\eps=0}= \Phi_1  \cos \xi+\tilde\Phi_1  \sin \xi,
\end{split}
\ee
where $\Phi_1=\phi_1\bp 1 \\ -2 \ep$ and $\tilde\Phi_1=\tilde\phi_1\bp 1 \\ -2 \ep$.
Substituting \eqref{Phi} into \eqref{derPhi} leads to
\be
\begin{split}
	0
	&=-\omega(\b_1 DU_{1}+\b_{2}D U_{2})+i\hat\sigma(-\b_1 DU_{2}+\b_{2}D U_{1})+\big[-\frac{1}{4}D+\bp 3 & 4 \\ -4 & -4 \ep\big](\Phi_1  \cos \xi+\tilde\Phi_1  \sin \xi).
\end{split}
\ee
Therefore, we have the following two equations:
\be 
\begin{split}
	0=&-\omega \b_1D \bp 2 \\ -1  \ep+i\hat\sigma\b_2D \bp 2 \\ -1  \ep-\frac{1}{4}D\Phi_1+\bp 3 & 4 \\ -4 & -4 \ep\Phi_1,\\
	0=&-i\hat\sigma \b_1D \bp 2 \\ -1  \ep-\omega \b_2D \bp 2 \\ -1  \ep-\frac{1}{4}D\tilde\Phi_1+\bp 3 & 4 \\ -4 & -4 \ep\tilde\Phi_1.\\
\end{split}
\ee

Therefore, $\Phi_1=\frac{4}{3}(-\omega \b_1+i\hat\sigma\b_2)\bp 1 \\ -2  \ep$ and $\tilde\Phi_1=\frac{4}{3}(-i\hat\sigma \b_1-\omega \b_2)\bp 1 \\ -2  \ep$.\\

Hence, $W = \b_1 U_{1}+\b_{2} U_{2} + \mathcal V=\b_1 U_{1}+\frac{4}{3}\eps(-\omega \b_1+i\hat\sigma\b_2)\bp 1 \\ -2  \ep+\b_{2} U_{2} +\frac{4}{3}\eps(-i\hat\sigma \b_1-\omega \b_2)\bp 1 \\ -2  \ep+\mathcal O(\eps^2)$. Next, we plug $W$ into 
\be
\begin{split}
	0
	& =\tilde Q\Big[ B(\eps,\omega,2\hat\sigma\eps)-\hat\l\eps^2 I\Big] W.
\end{split}
\ee
Then we arrive at
\be 
\begin{split}
	0=	& -\hat\lambda\eps^2\b_1+\frac{2}{3}\eps^2\b_1-\omega(-8)\frac{4}{3}\eps^2(-\omega \b_1+i\hat\sigma\b_2)+(-8i\hat\sigma)\frac{4}{3}\eps^2(-i\hat\sigma \b_1-\omega \b_2)-3\eps^2\frac{2}{3}(1-16\omega^2)\b_1,\\
	0=	& -\hat\lambda\eps^2\b_2+\frac{2}{3}\eps^2\b_2-(-8i\hat\sigma)\frac{4}{3}\eps^2(-\omega \b_1+i\hat\sigma\b_2)-(-8)\frac{4}{3}\eps^2(-i\hat\sigma \b_1-\omega \b_2)-\eps^2\frac{2}{3}(1-16\omega^2)\b_2.
\end{split}
\ee
Or,
\be
\begin{split}
	0
	&=\bp -\frac{4}{3}(1-16\omega^2)-\frac{32}{3}\hat\sigma^2-\hat\lambda & \frac{64}{3}i\omega\hat\sigma \\ -\frac{64}{3}i\omega\hat\sigma & -\frac{32}{3}\hat\sigma^2-\hat\lambda \ep\bp \b_1 \\ \b_2 \ep\eps^2+O(\eps^3).
\end{split}
\ee

This yields in passing agreement of the critical linearized dispersion relations.

\begin{proof}[Proof of Theorem \ref{dispthm}]
	Note that the exact reduced spectral system given by \eqref{det} and Corollary \ref{errcorr}
	agrees after Ginzburg Landau scaling ($\sigma =:2\eps \hat \sigma$, $\lambda_j=:\eps^2 \hat \lambda_j$) to appropriate order with the matrix eigenvalue problem \eqref{gldisp}.
	Also, it follows from formulas \eqref{eBM} and \eqref{eGL} that the roots likewise agree to lowest order, 
	giving the first result, \eqref{exactdisp}.
	The second result, reality of $\lambda_j$ for $|\hat\sigma|\leq C$, $16-\omega^2>0$, and $\eps>0$ sufficiently small,
	follows from \eqref{sqrtexp} by the observation that the argument of the square root is dominated
	by $\tilde c>0$ for $\eps$ sufficiently small.
\end{proof}



\begin{thebibliography}{GMWZ7}

\bibitem[DSSS]{DSSS}
A. Doelman, B. Sandstede, A. Scheel, and G. Schneider,
{\it The dynamics of modulated wavetrains,}
 Mem. Amer. Math. Soc.  199  (2009),  no. 934, viii+105 pp. ISBN: 978-0-8218-4293-5.

\bibitem[CK]{CK} T.K. Callahan and E. Knobloch, 
{\it Pattern formation in three-dimensional reaction-diffusion systems,} Phys. D 132 (1999) 339--362.


\bibitem[C]{C} M. Cross,
{\it Notes on the Turing Instability and Chemical Instabilities,} unpublished lecture notes, 
http://www.cmp.caltech.edu/\~mcc/BNU/Notes7\_2.pdf.

\bibitem[CL]{CL} P. Chossat and R. Lauterbach {\it Methods in Equivariant Bifurcations and Dynamical Systems,} Advanced series in nonlinear dynamics, World Scientific, 2000.

\bibitem[E]{E} W. Eckhaus, {\it Studies in nonlinear stability theory,} Springer tracts in Nat. Phil. Vol. 6, 1965.

\bibitem[GLSS]{GLSS} G. Gambino, M.C. Lombardo, M. Sammartino, and V. Sciacca,
{\it Turing pattern formation in the Brusselator system 
with nonlinear diffusion,}
Physical Review E, 88, 4, pp. 042925, (2013).

\bibitem[GS]{GS} M. Golubitsky and D. Schaeffer,
{\it Singularities and Groups in Bifurcation Theory,} Volume I, Applied Mathematical Sciences 51, Springer--Verlag New York, 1985.
\bibitem
[JNRZ1]{JNRZ1} M. Johnson, P. Noble, L.M. Rodrigues,  and K. Zumbrun,
{\it Nonlocalized modulation of periodic reaction diffusion waves:
Nonlinear stability,}
Arch. Ration. Mech. Anal. 207 (2013), no. 2, 693--715. 

\bibitem
[JNRZ2]{JNRZ2} M. Johnson, P. Noble, L.M. Rodrigues,  and K. Zumbrun,
{\it Nonlocalized modulation of periodic reaction diffusion waves:
The Whitham equation,}
 Arch. Ration. Mech. Anal. 207 (2013), no. 2, 669--692. 

\bibitem[JZ]{JZ} M. Johnson and K. Zumbrun,
{\it Nonlinear stability of spatially-periodic traveling-wave solutions of systems of reaction-diffusion equations,}
Ann. Inst. H. Poincaré Anal. Non Linéaire 28 (2011), no. 4, 471--483. 

\bibitem
[K]{K} T. Kato,
{\it Perturbation theory for linear operators},
Springer--Verlag, Berlin Heidelberg (1985).

\bibitem[KS]{KS} K. Kirchg\"assner and P. Sorger,
{\it  Stability analysis of branching solutions of the Navier 
Stokes equations,} In "Proceedings of the 12th Congress of Applied Mechanics (Stanford 1968), M. 
H'etenyi, G. Vincenti (eds), Springer Verlag 1969", pp. 257 268.

\bibitem[M1]{M1} A. Mielke, 
{\it A new approach to sideband-instabilities using the principle of 
reduced instability,}
 Nonlinear dynamics and pattern formation in the natural environment 
(Noordwijkerhout, 1994), 206--222, 
Pitman Res. Notes Math. Ser., 335, Longman, Harlow, 1995. 

\bibitem[M2]{M2} A. Mielke, 
{\it Instability and stability of rolls in the Swift-Hohenberg equation,}
Comm. Math. Phys. 189 (1997), no. 3, 829--853. 

\bibitem[M3]{M3} A. Mielke, 
{\it The Ginzburg-Landau equation in its role as a modulation equation,}
Handbook of dynamical systems, Vol. 2, 759--834, 
North-Holland, Amsterdam, 2002.

\bibitem[NPL]{NPL} A.C. Newell, T. Passot, and J. Lega,
{\it Order parameter equations for patterns,} Annual Reviews Fluid Mech. 25 (1993) 399--453.

\bibitem[NW]{NW} A.C. Newell and J. Whitehead, {\it Finite bandwidth, finite amplitude
	convection,} J. Fluid Mech. 39 (1969) 279--303.


\bibitem[PP-G]{PP-G} B. Pe\~na and C. Pérez-García, 
{\it Stability of Turing patterns in the Brusselator model,}
Phys. Rev. E (3) 64 (2001), no. 5, part 2, 056213, 9 pp.


\bibitem[PYZ]{PYZ} A. Pogan, J. Yao, and K. Zumbrun,
 {\it O(2) Hopf bifurcation of viscous shock waves in a channel,} Phys. D 308 (2015), 59--79.

\bibitem[PL]{PL} I. Prigogene, R. Lefever, 
{\it Symmetry breaking instabilities in dissipative systems II,} J. Chem. Phys. 48 (1968) 1665--1700.

\bibitem[SSSU]{SSSU}
B. Sandstede, A. Scheel, G. Schneider, and H. Uecker, 
{\it Diffusive mixing of periodic wave trains in reaction-diffusion systems,}
J. Diff. Eq. 252 (2012), no. 5, 3541--3574. 

\bibitem[S1]{S1} G. Schneider, 
{\it Diffusive stability of spatial periodic solutions of the Swift-Hohenberg equation,}
Commun. Math. Phys. 178, 679--202 (1996). 

\bibitem
[S2]{S2} G. Schneider, {\it Nonlinear diffusive stability
of spatially periodic solutions-- abstract theorem and higher space
dimensions},
Proceedings of the International Conference on Asymptotics
in Nonlinear Diffusive Systems (Sendai, 1997),  159--167,
Tohoku Math. Publ., 8, Tohoku Univ., Sendai, 1998.

\bibitem[TB]{TB} L.S. Tuckerman and D. Barkley,
{\it Bifurcation analysis of the Eckhaus instability,} Phys. D. 46 (1990) 57--86.

\bibitem[T]{T} A. Turing, 
{\it The chemical basis of morphogenesis,} Philos. Trans. Roy. Soc. Ser. B 237 (1952) 37--72.

\end{thebibliography}
\end{document}